\documentclass[twoside,a4paper,10pt]{amsart}
\usepackage{mathabx}
\usepackage{xcolor}
\definecolor{webgreen}{rgb}{0,.5,0}
\definecolor{webbrown}{rgb}{.6,0,0}
\definecolor{RoyalBlue}{cmyk}{1, 0.50, 0, 0}
\usepackage[colorlinks=true, breaklinks=true, urlcolor=webbrown, linkcolor=RoyalBlue, citecolor=webgreen,backref=page]{hyperref}

\usepackage{epsfig, graphicx, subfigure}
\usepackage{verbatim, setspace}
\usepackage{amsmath, amssymb}

\normalfont
\usepackage[T1]{fontenc}

\newcommand{\R}		{\mathbb{R}}
\newcommand{\C}		{\mathbb{C}}
\newcommand{\N}		{\mathbb{N}}
\newcommand{\Z}		{\mathbb{Z}}
\newcommand{\ZZ}		{\mathbb{Z}_{\geq0}}
\newcommand{\cws}{\stackrel{*}{\to}}

\newcommand{\cp}{\mathrm{cap}}
\newcommand{\dist}{\mathrm{dist}}

\newcommand{\supp}{\mathrm{supp}}
\newcommand{\diag}{\mathrm{diag}}

\newcommand{\Arg}{\mathrm{Arg}}
\renewcommand{\det}{\mathrm{det}}

\newcommand{\qasq}{\quad \text{as} \quad}
\newcommand{\qandq}{\quad \text{and} \quad}
\newcommand{\qorq}{\quad \text{or} \quad}

\newcommand{\dd}{\mathrm{d}}
\newcommand{\ic}{\mathrm{i}}
\newcommand{\vc}    {{\vec c}}
\newcommand{\RS}{\boldsymbol{\mathfrak{R}}}

\newcommand{\z}	{{\boldsymbol z}}
\newcommand{\x}	{{\boldsymbol x}}

\newcommand{\n}	{{\vec n}}
\newcommand{\Par}	{\mathsf{P}}
\newcommand{\J}	{\mathcal{J}_{\vec\kappa}}

\newcommand{\rhy}   {\textnormal{RHP}-${\boldsymbol Y}$}
\newcommand{\rhx}   {\textnormal{RHP}-${\boldsymbol X}$}

\newcommand{\rhz}   {\textnormal{RHP}-${\boldsymbol Z}$}
\newcommand{\rhn}   {\textnormal{RHP}-${\boldsymbol N}$}
\newcommand{\rhp}   {\textnormal{RHP}-${\boldsymbol P}$}
\newcommand{\rhpsi}   {\textnormal{RHP}-${\boldsymbol \Psi}$}
\newcommand{\rhphi}   {\textnormal{RHP}-${\boldsymbol \Phi}$}
\newcommand{\rhwphi}   {\textnormal{RHP}-${\widetilde{\boldsymbol \Phi}}$}

\newtheorem{theorem}{Theorem}[section]

\newtheorem{proposition}{Proposition}[section]
\newtheorem{lemma}{Lemma}[section]
\newtheorem*{definition}{Definition}

\numberwithin{equation}{section}

\begin{document}

\title[Jacobi matrices on trees generated by Angelesco systems]{ Jacobi matrices on trees generated by Angelesco systems:  asymptotics of coefficients and essential spectrum }

\author[A.I. Aptekarev]{Alexander I. Aptekarev}
\address{Keldysh Institute of Applied Mathematics, Russian Academy of Science, Miusskaya Pl. 4, Moscow, 125047 Russian Federation}
\email{\href{mailto:aptekaa@keldysh.ru}{aptekaa@keldysh.ru}}

\author[S. Denisov]{Sergey A. Denisov}
\address{Department of Mathematics, University of Wisconsin-Madison, 480n Lincoln Dr., Madison, WI 53706, USA}
\address{Keldysh Institute of Applied Mathematics, Russian Academy of Science, Miusskaya Pl. 4, Moscow, 125047 Russian Federation}
\email{\href{mailto:denissov@math.wisc.edu}{denissov@math.wisc.edu}}

\author[M. Yattselev]{Maxim L. Yattselev}
\address{Department of Mathematical Sciences, Indiana University-Purdue University Indianapolis, 402~North Blackford Street, Indianapolis, IN 46202, USA}
\address{Keldysh Institute of Applied Mathematics, Russian Academy of Science, Miusskaya Pl. 4, Moscow, 125047 Russian Federation}
\email{\href{mailto:maxyatts@iupui.edu}{maxyatts@iupui.edu}}

\thanks{The work of the first author was supported by a grant of the Russian Science Foundation (project RScF-19-71-30004). 
The research of the second author was supported by the Moscow Center for Fundamental and Applied Mathematics, project No. 20-03-01, grants NSF-DMS-1464479,  NSF DMS-1764245, and Van Vleck Professorship Research Award. 
The research of the third author was supported by the Moscow Center for Fundamental and Applied Mathematics, project No. 20-03-01, and grant CGM-354538 from the Simons Foundation.}

\begin{abstract} We continue  studying  the connection between  Jacobi matrices defined on a tree and  multiple orthogonal polynomials (MOPs) that was discovered in ~\cite{uApDenY}. In this paper, we consider  Angelesco systems formed by two analytic weights and obtain  asymptotics of the recurrence coefficients and strong asymptotics of MOPs  along all directions (including the marginal ones). These results are then applied to show that the essential spectrum of the related Jacobi matrix is the union of intervals of orthogonality.
\end{abstract}

\subjclass[2010]{Primary 47B36, 47A10, 42C05}

\keywords{Jacobi matrices on trees, essential spectrum, multiple orthogonal polynomials, Angelesco systems}

\maketitle

\setcounter{tocdepth}{2}
\tableofcontents

\section{Introduction}

It is well-known \cite{akh1} that the spectral theory of one-sided self-adjoint Jacobi matrices can be naturally studied in the context of orthogonal polynomials on the real line and, conversely, many results in the latter topic find an operator-theoretic interpretation. In \cite{uApDenY}, we discovered that a wide class of multiple orthogonal polynomials (MOPs), e.g., celebrated Angelesco systems, is connected to self-adjoint Jacobi matrices defined on rooted Cayley trees. The present paper makes a further step in this direction. We perform a case study of Angelesco systems with two measures of orthogonality given by analytic weights. Our analysis of the related matrix Riemann-Hilbert problem provides the asymptotics of the recurrence coefficients and strong asymptotics of MOPs for all large indices. One application of this precise asymptotic analysis is a characterization of the essential spectrum of the associated Jacobi matrix.

We  start this introduction by recalling some definitions and main
relations connecting Jacobi matrices on trees and
MOPs and then state the main results of the paper.
In what follows, we let \( \N:=\{1,2,\ldots\} \) and \(
\ZZ:=\{0,1,2,\ldots\} \). We write \( |\n| := n_1+\cdots+n_d \) for
\( \n=(n_1,\ldots, n_d)\in\ZZ^d \), and let \( \vec
e_1=(1,0,\ldots,0),\ldots, \vec e_d=(0,\ldots,0,1) \), \( \vec 1 =
(1,\ldots,1)= \vec e_1+\cdots+\vec e_d\). Given an operator $\mathcal{A}$ in the Banach space, the symbols $\sigma(\mathcal{A})$ and $\sigma_{\rm ess}(\mathcal{A})$ will denote its spectrum and essential spectrum, respectively \cite{rs1}. In a metric space, $B_r(X)$ denotes the closed ball with center at $X$ and radius $r$. For a complex number $z$, $\Re z$ and $\Im z$ are its real and imaginary parts, respectively. For a function $f(z)$, holomorphic in $\C^+$, the upper half-plane, its boundary values on $\R$ are denoted by $f_+(x)$.

\subsection{Jacobi matrices on trees}

Denote by \( \mathcal T \) an infinite $(d+1)$-homogeneous rooted tree
(rooted Cayley tree) and by \( \mathcal V \) the set of its vertices with
\( O \) being the root. On the lattice \( \N^d \), consider an infinite path $\{\n^{(1)},\n^{(2)}, \ldots\}$ that starts at $\vec 1$ (i.e., $\n^{(1)}=\vec 1$) and satisfies $\n^{(j+1)}=\n^{(j)}+\vec{e}_{k_j}, k_j\in \{1,\ldots, d\}$ for every $j=0,1,\ldots$. Clearly, these are paths for which, as we move from $\vec 1$ to infinity, the multi-index of each next vertex is increasing by $1$ at exactly one position.  Each such path can be mapped  to non-selfintersecting path in \( \mathcal T \) that starts at $O$   (see Figure~\ref{fig:inf-tree} for
$d=2$) and this map is one-to-one. This construction defines a projection $\Pi:\mathcal{V}\to\N^d$ as follows: given $X\in \mathcal{V}$ we consider a path from $O$ to $X$, map it to a path on \( \N^d \) and let $\Pi(X)$ be the endpoint of the mapped path.
 Every vertex $Y\in
\mathcal{V}, Y\neq O$, has the unique parent, which we denote by \(
Y_{(p)} \). This allows us to define the following index function:
\begin{equation}
\label{imath}
\imath:\mathcal V \to \{1,\ldots,d\}, \quad Y\mapsto \imath_Y \text{ such that }
\Pi(Y) = \Pi(Y_{(p)}) + \vec e_{\imath_Y},
\end{equation}
and therefore to distinguish the 
``children'' of each vertex \( Y \in \mathcal V \) by denoting \(
Z=Y_{(ch),\imath_Z} \)  when \( Y=Z_{(p)} \), see
Figure~\ref{fig:inf-tree} (for $d=2$).
\begin{figure}[h!]
\includegraphics[scale=.6]{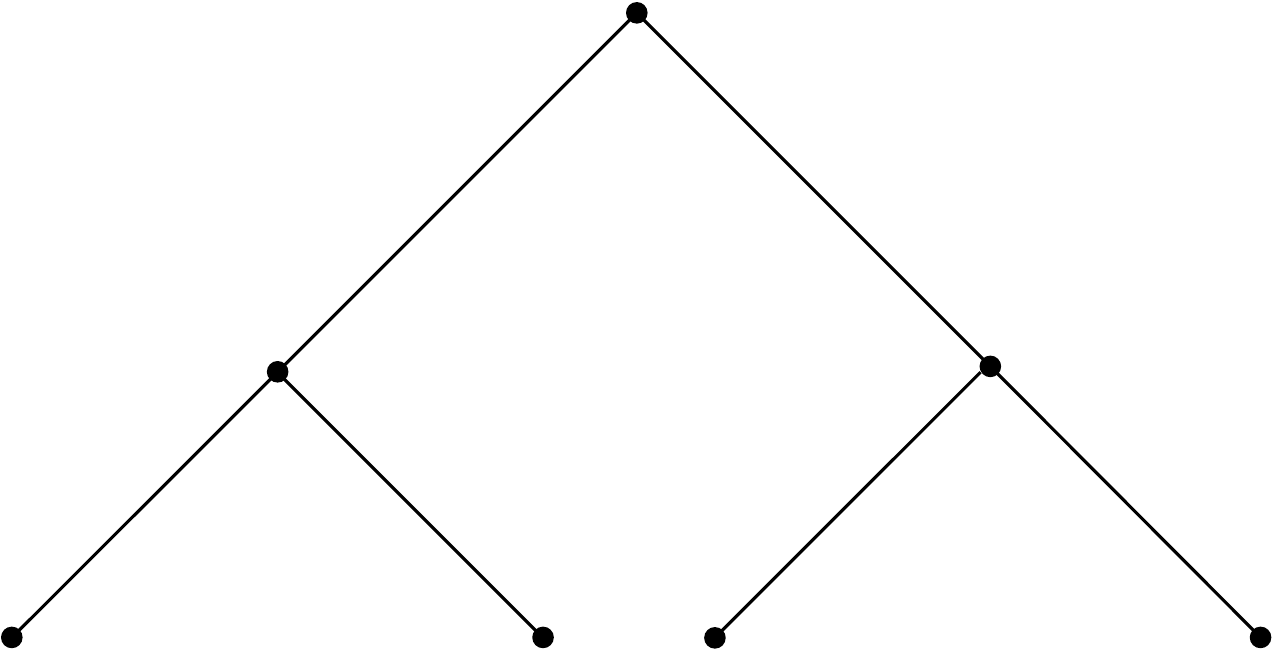}
\begin{picture}(0,0)
\put(-110,110){$(1,1)\sim O=Y_{(p)}$}
\put(-170,47){$(2,1)$}
\put(-48,47){$(1,2)\sim Y=O_{(ch),2}$}
\put(-215,0){$(3,1)$}
\put(-128,0){$(2,2)$}
\put(-95,0){$(2,2)\sim Y_{(ch),1}$}
\put(-2,0){$(1,3)\sim Y_{(ch),2}$}
\end{picture}
\caption{Three generations of \( \mathcal T \) (for $d=2$).}
\label{fig:inf-tree}
\end{figure}
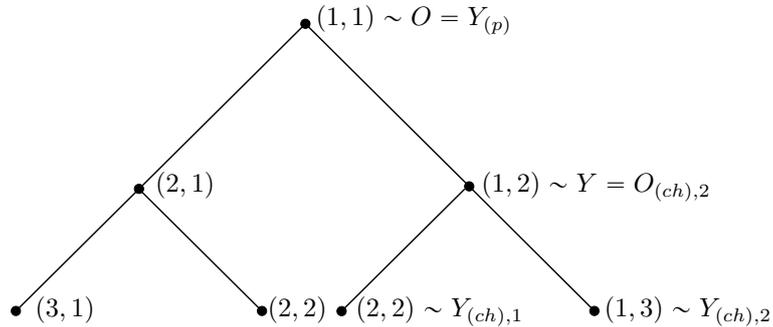

Let \( \Par := \{a_{\n,i},b_{\n,i}\}_{\n\in\ZZ^d,\,i\in\{1,\ldots,d\}} \)
be a collection of real parameters satisfying conditions
\begin{equation}
\label{bounded}
\left\lbrace
\begin{array}{l}
0<a_{\n,i}\,\,{\rm for\,\, all}\,\, \n\in \N^d, \quad i\in \{1,\ldots,d\},\medskip\\
\sup \limits_{\n\in \N^d,i\in \{1,\ldots,d\}}a_{\n,i}<\infty\,,\,  \sup \limits_{\n\in\ZZ^d,i\in \{1\ldots,d\}}|b_{\n,i}|<\infty.
\end{array}
\right.
\end{equation}
For a function \( f \) on \( \mathcal V \), we denote by \( f_Y \)
its value at a vertex \( Y\in\mathcal V \). Given \( \Par \)
satisfying \eqref{bounded} and \( \vec\kappa\in\R^d \) with \(
|\vec\kappa|  = 1 \), we define the corresponding Jacobi operator
 \( \J \) by
\begin{equation}
\label{Jacobi}
\left\{
\begin{array}{ll}
(\J f)_Y := a_{\Pi(Y_{(p)}),\imath_Y}^{1/2}f_{Y_{(p)}} + b_{\Pi(Y_{(p)}),\imath_Y}f_Y
 + \sum_{i=1}^{d} a_{\Pi(Y),i}^{1/2}f_{Y_{(ch),i}}
 ,  & Y\neq O, \medskip \\
(\J f)_O := \sum_{i=1}^{d}\kappa_i b_{\vec{1}-\vec e_{i},i} f_O
+ \sum_{i=1}^{d} a_{\vec{1},i}^{1/2}f_{O_{(ch),i}}
, & Y=O.
\end{array}
\right.
\end{equation}
Thus defined operator \( \J \) is bounded and self-adjoint on \(
\ell^2(\mathcal V) \). 

The spectral theory of Jacobi matrices on trees enjoyed considerable progress in the last decade, see, e.g.,
\cite{t1,t5,dk,fs,t3,t2,t4,t6,at1,at2}. In this paper, we  study  Jacobi matrices on trees that are generated by multiple orthogonality conditions. For this class of Jacobi matrices, one can study subtle questions of spectral analysis, such as the spatial asymptotics of Green's function, by employing the powerful  asymptotical methods developed in the context of multiple orthogonal polynomials (see, e.g, formulas (4.30) and (4.31) in \cite{uApDenY}). In the current work, we focus on characterizing the so-called $\mathcal{R}$-limits and on detecting the essential spectrum in the case, when the multiple orthogonal polynomials are given by the Angelesco system with analytic weights.


\subsection{Multiple orthogonal polynomials and recurrence relations}

In \cite{uApDenY}, we investigated 
properties of the operator  \( \J \) in the case when the coefficients \( \Par
\) are the recurrence coefficients for
MOPs. We now recall some basic facts about multiple orthogonal polynomials.

 Let  $\vec{\mu}:=(\mu_1,\ldots,\mu_d), \, d\in\N$, be a vector
of positive finite Borel measures defined on $\mathbb{R}$ and \(\n\) be a given
a multi-index in \( \ZZ^d \), \( |\n|\geq1 \). Type I MOPs
$\big\{A_{\n}^{(j)}\big\}_{j=1}^d$ are not identically zero polynomial coefficients
of the linear form
$$Q_{\n}(x):=\sum\limits_{j=1}^d
A_{\n}^{(j)}(x)\dd \mu_j(x),\quad \deg A_\n^{(i)} < n_i, \quad  i\in\{1,\ldots, d\}, $$
defined by the conditions
\begin{equation}\label{typeI}
\int x^lQ_\n(x) = 0, \quad l<|\n|-1, \quad A_{\vec{1}-\vec{e}_{i}}^{(i)} \equiv 0.
\end{equation}
Type II MOPs \( P_\n(x) \), \( \deg\big(P_\n\big) \leq |\n| \), are not identically zero and defined by
\begin{equation}
\label{typeII}
\int P_{\n}(x)x^l \dd\mu_i(x) =0, \quad l<n_i, \quad i\in\{1,\ldots, d\}.
\end{equation}
The polynomials of both types always exist, but their uniqueness is not guaranteed. If $\deg(P_\n)=|\n|$ for every non-identically zero polynomial \( P_\n(x) \) satisfying \eqref{typeII}, then the multi-index $\vec{n}$ is called normal. In this case \( P_\n(x) \) is unique up to a multiplicative factor and we normalize it to be monic, i.e., \( P_\n(x)=x^{|\vec{n}|}+\cdots\). It turns out that $\n$ is normal  if and only if the linear form $ Q_{\n}(x) $ is defined uniquely up to multiplication by a constant. In this case, we will normalize it by
\begin{equation}\label{n_2}
\int x^{|\n|-1}Q_{\n}(x)=1\,.
\end{equation}
 We will say that a vector $\vec{\mu}$  is called \textit{perfect} if all the multi-indices
$\n\in \ZZ^d$ are normal.

When $\vec{\mu}$ is perfect, it is known \cite{VA11}  that   the
polynomials \( P_\n(x) \) and the forms \( Q_\n(x) \) satisfy the
following nearest-neighbor recurrence relations (NNRRs):
\begin{equation}
\label{recurrence}
\left\{
\begin{array}{l}
z P_\n(z) = P_{\n+\vec e_j}(z) + b_{\n,j}P_\n(z) +
\sum_{i=1}^{d} a_{\n,i}P_{\n-\vec e_i}(z), \medskip \\
z Q_\n(z) = Q_{\n-\vec e_j}(z) + b_{\n-\vec e_j,j}Q_\n(z) +
\sum_{i=1}^{d} a_{\n,i}Q_{\n+\vec e_i}(z),
\end{array}
\right. \,\, \text{for each} \,\, j\in\{1,\ldots, d\} .
\end{equation}
For the coefficients \(  \{a_{\n,i},b_{\n,i}\}
\), we have representations \cite{uApDenY}:
\begin{equation}\label{por1}
a_{\n,j}=\frac{
\int
P_\n(x)\,x^{n_j}\dd\mu_j(x)}{
\int
P_{\n-\vec{e}_j}(x)\, x^{n_j-1}\dd\mu_j(x)}, \qquad
b_{\n-\vec{e}_j,j}=\int
x^{|\n|}Q_{\n}(x)-\int
x^{|\n|-1}Q_{\n-\vec{e}_j}(x)\,.
\end{equation}
 If $d>1$, unlike in one-dimensional case, we can not prescribe $\{a_{\n,j}\}$
 and $\{b_{\n,j}\}$ arbitrarily. In fact, these coefficients
 satisfy the so-called ``consistency conditions'' which is a system of nonlinear difference equations.
 This  discrete integrable system and the associated Lax pair were studied in
 \cite{ADVA,VA11}.

 \subsection{Angelesco systems and ray limits of NNRR coefficients}
 We recall that $\vec\mu $ is an \textit{Angelesco} system of measures if
\begin{equation}\label{1.10}
\supp\,\mu_j=\Delta_j:=[\alpha_j,\beta_j]:\qquad
\Delta_i\cap\Delta_j=\varnothing,\quad i \neq j,\quad i,j\in\{1,\ldots,d\},
\end{equation}
i.e., the supports  of measures form a system of $d$ closed segments
separated by $d-1$ nonempty open intervals. We can always assume without loss of generality that $\beta_j<\alpha_{j+1}, j\in\{1,\ldots,d-1\}$.

Angelesco systems form an important subclass of the perfect systems.
They were  studied by Angelesco already in 1919, \cite{Ang19}.
 It is not  difficult to
see \cite{uApDenY}
that the corresponding NNRR
coefficients satisfy
conditions \eqref{bounded} and thus define the Jacobi matrix $\J$ by \eqref{Jacobi}.

The  asymptotic behavior of these  coefficients \(
\{a_{\n,j},b_{\n,j}\} \) for the \textit{ray sequences regime},
namely when
\begin{equation}
\label{multi-indices} \mathcal
N_\vc=\{\n\} \,:\qquad \quad
n_i = c_i|\n| + o\big(\n\big), \quad i\in\{1,\ldots,d\}, \quad  \quad |\:\vc\:|:=\sum_{i=1}^dc_i=1\,,
\end{equation}
was studied  in \cite{uApDenY} for
$\vec{c}=(c_1,\ldots,c_d)\in(0,1)^d $ (hereafter, \(\lim_{\mathcal N_{\vec c}} \) stands for the limit as \( |\n|\to\infty \), \( \n\in\mathcal N_{\vec c} \)). The following theorem was proved.

\begin{theorem}[\cite{uApDenY}]
\label{thm:recurrenceOld}
Let $\vec\mu $ be an Angelesco system \eqref{1.10} such that  for
each \( i\in\{1,\ldots,d\} \)  the measure $\mu_i$ is absolutely
continuous with respect to the Lebesgue measure on \( \Delta_i \)
and the density \( \mu_i^\prime(x):=\dd\mu_i(x)/\dd x \) extends to a
holomorphic and non-vanishing function in some neighborhood of
$\Delta_i$. Then the ray limits \eqref{multi-indices} of  coefficients
$\big\{a_{\n,i},b_{\n,i}\big\}$ from \eqref{recurrence} exist for
any  $\vec{c}\in(0,1)^d $:
\begin{equation}
\label{limit}
\lim_{\mathcal N_\vc}a_{\n,i} =A_{\vc,i} \quad \text{and} \quad \lim_{\mathcal N_\vc}b_{\n,i} =B_{\vc,i}, \quad i\in\{1,\ldots,d\}.
\end{equation}
\end{theorem}
This result and expressions for  $A_{\vc,i}$ and $B_{\vc,i}$ were obtained from the strong asymptotics of the  MOPs also established in \cite{uApDenY} (along the ray \( \vc = (1/d,\ldots,1/d) \) the limits in \eqref{limit} can be deduced from the results in \cite{ApKalLLRocha06}). As it happens, the numbers $A_{\vc,i}$ and $B_{\vc,i}$ depend only on the vector \( \vc \) and the intervals  $\{\Delta_i\}_{i=1}^d$ (see \eqref{AngPar1} for the case \( d=2 \) where \( \vc=(c,1-c) \) and $A_{\vc,i}=A_{c,i}$, $B_{\vc,i}=B_{c,i}$).

\subsection{Results and structure of the paper} In this  paper, we restrict ourselves to the case $d=2$. Our main technical achievement is an extension of the results in \cite{uApDenY} on the strong asymptotics of the Angelesco MOPs  to the full range of $\vec{c}\,$:  $\vec{c}\in[0,1]^2 $. As a corollary of this extension, we get the following result.
 \begin{theorem}
\label{thm:recurrence}
Let $\vec\mu $ be as in Theorem~\ref{thm:recurrenceOld} with \( d=2 \). Then the ray limits
\begin{equation}
\label{AngPar}
\lim_{\mathcal N_c} a_{\n,i} = A_{c,i} \qandq \lim_{\mathcal N_c} b_{\n,i} = B_{c,i}
\end{equation}
exist for any \( c\in[0,1] \) and \( i\in\{1,2\} \), where \( \mathcal N_c := \mathcal N_{(c,1-c)}\) is any sequence satisfying \eqref{multi-indices}. 
\end{theorem}
 Theorem \ref{thm:recurrence} can be used to characterize the essential spectrum of the Jacobi operator $\J$, defined in \eqref{Jacobi}, generated by an Angelesco system. 
 
 \begin{definition}
  Let \( \Par := \{\widehat a_{\n,i},\widehat b_{\n,i}\}_{\n\in\ZZ^2,\,i=1,2}\) be a set of real numbers that satisfy \eqref{bounded} for $d=2$ and the constants \( \{A_{c,1},A_{c,2},B_{c,1},B_{c,2}\}_{c\in[0,1]} \) be limits from \eqref{AngPar} (notice that $\Par$ does not have to be a set of the recurrence coefficients of any Angelesco system, but the limits \( \{A_{c,1},A_{c,2},B_{c,1},B_{c,2}\}_{c\in[0,1]} \) are generated by some $\Delta_1$ and $\Delta_2$). We say that
\(\Par\in\Par_{Ang}(\Delta_1,\Delta_2) \) if \(\Par \) satisfies
\begin{equation}
\label{AngPar_1}
\lim_{\mathcal N_c} \widehat a_{\n,i} = A_{c,i} \qandq \lim_{\mathcal N_c} \widehat b_{\n,i} = B_{c,i}
\end{equation}
for any \( c\in[0,1] \) and \( i\in\{1,2\} \), where, again, \( \mathcal N_c := \mathcal N_{(c,1-c)}\) is any sequence satisfying \eqref{multi-indices}. 
 \end{definition}
 
 By Theorem~\ref{thm:recurrence}, the class \(\Par_{Ang}(\Delta_1,\Delta_2) \) is not empty since the recurrence coefficients of any Angelesco system with analytic weights supported on $\Delta_1$ and $\Delta_2$ belong in \(\Par_{Ang}(\Delta_1,\Delta_2) \). 
Consider Jacobi matrix \(\J\) defined in \eqref{Jacobi} with coefficients in \(\Par_{Ang}(\Delta_1,\Delta_2) \). The following result gives characterization of its essential spectrum.

\begin{theorem}
\label{thm:spectrum} 
Let \( \J \) be the Jacobi operator defined by \eqref{Jacobi} and corresponding to a collection of parameters \( \Par\in\Par_{Ang}(\Delta_1,\Delta_2) \), then \( \sigma_\mathrm{ess}(\J)=\Delta_1\cup\Delta_2\). In particular, the essential spectrum of the Jacobi matrix generated by an Angelesco system with analytic weights supported on $\Delta_1$ and $\Delta_2$ is $\Delta_1\cup\Delta_2$.
\end{theorem}

We prove this theorem in Section~\ref{sec:10}. The necessary definitions and statements of the main results on strong asymptotics of MOPs
 are adduced in Section~\ref{sec:3}. Auxiliary results and their proofs are relegated to Sections~\ref{sec:4} and~\ref{sec:5}. Proofs of the main results can be found in Sections~\ref{sec:6} and~\ref{sec:8}.

\section{Expressions for the ray limits and proof of Theorem~\ref{thm:spectrum}}
\label{sec:10}

\subsection{Expressions for the ray limits}
\label{ss:21}
In this subsection we give formulas for the limits in \eqref{AngPar}.


Let \( \Delta_1=[\alpha_1,\beta_1] \) and \(
\Delta_2=[\alpha_2,\beta_2] \) be two intervals on the real line
such that \( \beta_1<\alpha_2 \). Denote by \( \omega_1 \) and \(
\omega_2 \) the arcsine distributions on \( \Delta_1 \) and \(
\Delta_2 \), respectively. It is known \cite{Ransford}
that
\[
E(\omega_i,\omega_i) \leq E(\nu,\nu), \quad E(\mu,\nu) := -\int\log|x-y|\dd\mu(x)\dd\nu(y),
\]
for any probability Borel \( \nu \) measure on \( \Delta_i \).  The
logarithmic potentials of these measures satisfy
\[
\ell_i - V^{\omega_i} \equiv 0 \quad \text{on} \quad \Delta_i,
\]
for some constants \( \ell_1 \) and \( \ell_2\), where \( V^\nu(z)
:= -\int\log|z-x|\dd\nu(x) \). Now, given \( c\in(0,1) \), define
\begin{equation}
\label{Mc}
M_c := \big\{(\nu_1,\nu_2):~\supp(\nu_i)\subseteq \Delta_i,~\|\nu_1\|=c,~\|\nu_2\|=1-c\big\}.
\end{equation}
It is known \cite{GRakh81} that there exists the unique pair
of measures \( (\omega_{c,1},\omega_{c,2})\in M_c \) such that
\begin{equation}
\label{vecI}
I(\omega_{c,1},\omega_{c,2}) \leq I(\nu_1,\nu_2), \quad I(\nu_1,\nu_2) := 2E(\nu_1,\nu_1) + 2E(\nu_2,\nu_2) + E(\nu_1,\nu_2) + E(\nu_2,\nu_1),
\end{equation}
for all pairs \( (\nu_1,\nu_2)\in M_c \). Moreover,  \(
\supp(\omega_{c,1})=[\alpha_1,\beta_{c,1}] =:\Delta_{c,1}\) and  \(
\supp(\omega_{c,2})=[\alpha_{c,2},\beta_2]=:\Delta_{c,2} \).
Similarly to the case of a single interval,
there exist constants \( \ell_{c,i}\), \( i\in\{1,2\} \), such that
\begin{equation}
\label{var}
\left\{
\begin{array}{lll}
\ell_{c,1} - V^{2\omega_{c,1}+\omega_{c,2}} \equiv 0 & \text{on} & \supp(\omega_{c,1}), \medskip \\
\ell_{c,2} - V^{\omega_{c,1}+2\omega_{c,2}} \equiv 0 & \text{on} & \supp(\omega_{c,2}).
\end{array}
\right.
\end{equation}
The
dependence of the intervals \( \Delta_{c,i} \) on the parameter \(
c \) is described in greater detail in Section~\ref{sec:4}.

Let $\RS_c$, \( c\in(0,1) \), be a 3-sheeted Riemann surface
realized as follows: cut a copy of \( \overline \C \) along  \(
\Delta_{c,1}\cup \Delta_{c,2} \), which henceforth is denoted by
$\RS_c^{(0)}$, the second copy of \( \overline\C \) is cut along \(
\Delta_{c,1} \) and is denoted by $\RS_c^{(1)}$, while the third
copy is cut along \( \Delta_{c,2} \) and is denoted by \(
\RS_c^{(2)} \). These copies are then glued to each other crosswise
along the corresponding cuts, see Figure~\ref{fig:surfaceA}.
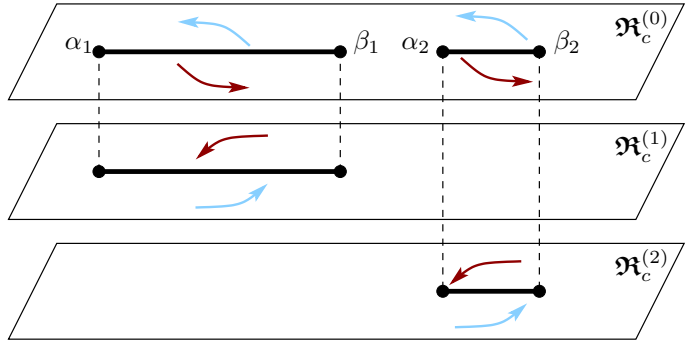
\begin{figure}[h!]
\includegraphics[scale=.5]{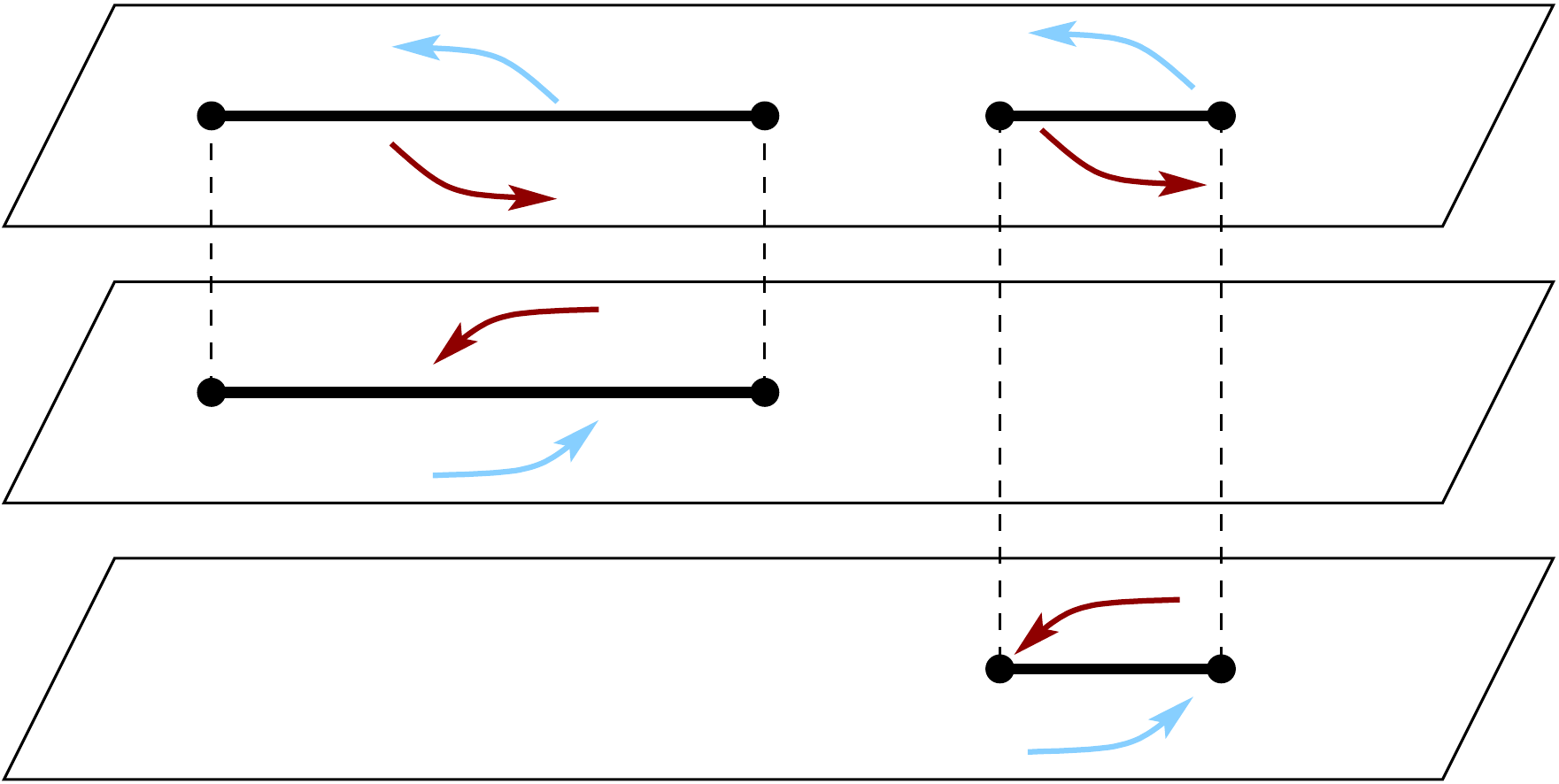}
\begin{picture}(0,0)
\put(-236,110){$\alpha_1$}
\put(-128,110){$\beta_1$}
\put(-110,110){$\alpha_2$}
\put(-53,110){$\beta_2$}
\put(-30,115){$\RS_c^{(0)}$}
\put(-30,70){$\RS_c^{(1)}$}
\put(-30,25){$\RS_c^{(2)}$}
\end{picture}
\caption{Surface \( \RS_c \) when \( \beta_{c,1}=\beta_1 \) and \( \alpha_{c,2}=\alpha_2\).}
\label{fig:surfaceA}
\end{figure}
It can be easily verified that thus constructed Riemann surface has genus
0. We denote by $\pi$ the natural projection from $\RS_c$ to
$\overline\C$ and employ the notation \( \z \) for a generic point
on \( \RS_c \) with \( \pi(\z)=z \) as well as $z^{(i)}$ for a
point on $\RS_c^{(i)}$ with $\pi(z^{(i)})=z$. We call a linear combination \( \sum n_i\z_i\), \( n_i\in\Z \), a \emph{divisor}. The degree of \( \sum n_i\z_i\) is defined as \( \sum n_i \). We say that \( \sum n_i\z_i\) is a zero/pole divisor of a rational function on \( \RS_c \) if this function has a zero at \( \z_i \) of multiplicity \( n_i \) when \( n_i>0 \), a pole at \( \z_i \) of order \( -n_i \) when \( n_i<0 \), and has no other zeros of poles. Zero/pole divisors necessarily have degree zero. Since $\RS_c$ has
genus zero, one can arbitrarily prescribe zero/pole divisors of
rational functions on $\RS_c$ as long as the degree of the divisor
is zero. A rational function with a given divisor is
unique up to multiplication by a constant.

\begin{proposition}
\label{prop:angelesco} Let \( \RS_c \), \( c\in(0,1) \), be as
above and \( \chi_c(\z) \) be the conformal map of \( \RS_c \) onto
\( \overline\C \) such that
\begin{equation}
\label{chi}
\chi_c\big( z^{(0)} \big) = z + \mathcal O\big(z^{-1} \big) \qasq z\to\infty.
\end{equation}
Further, let the numbers \( A_{c,1},A_{c,2},B_{c,1},B_{c,2} \) be defined by
\begin{equation}
\label{AngPar1}
\chi_c\big( z^{(i)} \big) =: B_{c,i} + A_{c,i}z^{-1} + \mathcal O\big(z^{-2} \big) \qasq z\to\infty, \; i\in\{1,2\}.
\end{equation}
Finally, let \( w_i(z):=\sqrt{(z-\alpha_i)(z-\beta_i)} \) be the
branch holomorphic outside of \( \Delta_i \) and normalized so that
\( w_i(z)/z\to 1 \) as \( z\to\infty \), in which case
\begin{equation}
\label{varphii}
\varphi_i(z) := \frac12\left(z-\frac{\beta_i+\alpha_i}2 + w_i(z)\right)
\end{equation}
is the conformal map of \( \overline\C\setminus\Delta_i \) onto the
complement of the disk \( B_{(\beta_i-\alpha_i)/4}(0) \)
satisfying \( \varphi_i(z) = z + \mathcal O(1) \) as \( z\to\infty
\). Then it holds that
\begin{equation}
\label{AngPar2}
\lim_{c\to0} \left\{
\begin{array}{lllll}
A_{c,2} & = & \big[(\beta_2-\alpha_2)/4\big]^2 &=:& A_{0,2}, \medskip \\
B_{c,2} & = & (\beta_2+\alpha_2)/2 &=:& B_{0,2}, \medskip \\
A_{c,1} & = & 0 &=:& A_{0,1}, \medskip \\
B_{c,1} & = & B_{0,2}+\varphi_2(\alpha_1) &=:& B_{0,1},
\end{array}
\right.
\end{equation}
and analogous limits hold when \( c\to1 \). Moreover, all the
constants  \( A_{c,1},A_{c,2},B_{c,1},B_{c,2} \) are continuous
functions of the parameter \( c\in[0,1] \).
\end{proposition}

Let us stress that the numbers \( A_{c,i}\) and \( B_{c,i} \) defined in \eqref{AngPar1} are precisely the ones appearing in \eqref{AngPar}. Even though the expression for \( B_{0,1} \) might seem strange, it has a meaning from the point of view of spectral theory of Jacobi matrices, see \eqref{mly4}.

We prove Proposition~\ref{prop:angelesco} in Section~\ref{sec:5}.
It is worth noting that the constants \( A_{c,1} \) and \( A_{c,2}
\) are always positive (except for \( A_{0,1} \) and \( A_{1,2} \), of course). Indeed, denote by \( \boldsymbol
\alpha_1,\boldsymbol\beta_{c,1},\boldsymbol\alpha_{c,2},\boldsymbol\beta_2\)
the ramification points of \( \RS_c \) with natural projections \(
\alpha_1,\beta_{c,1},\alpha_{c,2},\beta_2 \), respectively. Then
the symmetries of \( \RS_c \) and \( \chi_c(\z)\) yield that \(
\chi_c(\z) \) is real and changes from \( -\infty \) to \( \infty
\) when \( \z \)  moves along the cycle
\[
\infty^{(0)} \to \boldsymbol \alpha_1 \to \infty^{(1)} \to \boldsymbol \beta_{c,1} \to \boldsymbol \alpha_{c,2} \to \infty^{(2)} \to \boldsymbol \beta_2 \to \infty^{(0)}
\]
whose natural projection is the extended real line. Thus, \(
\chi_c(\z) \) is increasing when it moves past \( \infty^{(1)} \)
and \( \infty^{(2)} \), which yields the claim (this argument also
shows that \( B_{c,1}<B_{c,2} \)).

\subsection{Proof of Theorem \ref{thm:spectrum}}
\label{ss:2.2}

Our proof will be based on a characterization of the essential
support of a Jacobi matrix on a tree obtained in
\cite[Theorem~4]{BrDenEl18}. We need some preliminaries to
formulate this result. Suppose \( \mathcal T \) is a 3-homogeneous rooted
tree with root at $O$ (a binary tree), which means that $O$ has two neighbors and any other vertex has three neighbors.
Later in the text, we will use the notation
${Z}\sim {Y}$ to indicate that vertices ${Z}$
and ${Y}$ are neighbors and the symbol $\mathcal{V}$ will denote the set of all
 vertices of \( \mathcal T \).
Given a real function $V$ defined on \(\mathcal V\)  and a real positive function $W$ defined on all edges, we make an assumption
\begin{equation}
\label{Isad1}
\sup_{Y\in \mathcal V}|V_Y|<\infty,\,\, 0<W_{Z,Y},\,\, \sup_{Z\sim Y,Y\in \mathcal V}W_{Z,Y}<\infty\,,
\end{equation}
to introduce \( \mathcal J \), a bounded self-adjoint Jacobi matrix
\begin{equation}
\label{Isad2}
(\mathcal{J}f)_Y:= V_Yf_Y+\sum_{Z\sim Y}W_{Z,Y}^{1/2}f_Z, \quad {Y}\in \mathcal{V}\,,
\end{equation}
defined on $\ell^2( \mathcal{V})$.  One example one can think of is \(
\J \) introduced in \eqref{Jacobi}. Consider a set of distinct
vertices (a path) $\{{Y}_n\}, n\in \N$, in $\mathcal{V}$ such
that ${Y}_n\sim {Y}_{n+1}$ for every $n$. Clearly,
every such path on the tree  escapes to infinity, i.e., ${\rm
dist}(O,{Y}_n)\to\infty, n\to\infty$. We want
to define $\mathcal{R}$-limit (or ``right limit'') of $\mathcal{J}$ along this path. To
do that, suppose $\mathcal{G}$ is a 3-homogeneous tree (without a root),
${O}'$ is a fixed vertex on $\mathcal{G}$, and $\mathcal{J}'$ is a
bounded self-adjoint operator on $\mathcal{G}$. Recall that $B_r({Y})$ stands for the ball of radius
$r$ centered at ${Y}$ and denote the restriction operator to this ball by
$P_{B_r({Y})}$. Consider two finite matrices: $P_{B_r({Y}_{n_j})}\mathcal{J}P_{B_r({Y}_{n_j})}$ and $P_{B_r({O}')}\mathcal{J}'P_{B_r({O}')}$. If we identify $\ell^2(B_r({O}'))$ and $\ell^2(B_r({Y}_{n_j}))$ by following the structure of the tree (and there are many ways to do that), then these matrices are defined on the same finite dimensional Euclidean space. If this identification can be done so that all sections of  $\mathcal{J}'$ appear as the limits, we call  $\mathcal{J}'$ an 
 $\mathcal{R}$-limit or right limit:
\begin{definition}
We say that $\mathcal{J}'$ is an $\mathcal{R}$-limit of
$\mathcal{J}$ along $\{{Y}_n\}$ if there is a subsequence
$\{n_j\}$ such that
\[
P_{B_r({Y}_{n_j})}\mathcal{J}P_{B_r({Y}_{n_j})}\to
P_{B_r({O}')}\mathcal{J}'P_{B_r({O}')} \quad \text{as} \quad j\to\infty
\]
for every fixed $r\in \N$. Matrix $\mathcal{J}'$ is called simply an $\mathcal{R}$-limit of
$\mathcal{J}$ if there exists a path along which $\mathcal{J}'$ is an $\mathcal{R}$-limit of \( \mathcal J\).
\end{definition}

\noindent
{\bf Remark. }  For the rigorous definition
of $\mathcal{R}$-limit on more general graphs, see
\cite{BrDenEl18}.

\begin{theorem}[Theorem 4 in \cite{BrDenEl18}] We have
\[
\sigma_{\rm ess}(\mathcal{J})=\overline{\bigcup_{\mathcal{J}' {\rm \,is \, an\,} \mathcal{R}-{\,\rm limit \,of\,}\mathcal{J}} \sigma(\mathcal{J}')}\,.
\]
\label{ref_d}
\end{theorem}

\noindent
{\bf Remark.} \cite[Theorem 4]{BrDenEl18} was stated for the regular trees only, but the proof is valid for rooted trees as well.\smallskip

{\bf Auxiliary operators $\mathcal{L}_c^{(1)}$ and $\mathcal{L}_c^{(2)}$.} Recall that $\mathcal{T}$ denotes the $3$-homogeneous rooted tree with the root denoted by $O$ and $\mathcal{V}$ stands for the set of all its vertices.  There are two edges meeting at the root $O$. We label one of them type 1 and the other one -- type 2. Now, consider two vertices that are at distance $1$ from  $O$. Each of them is coincident with exactly three edges. One of the edges for each vertex was labelled already, and we label the remaining two as an edge type 1 and an edge of type 2. We continue inductively by considering all edges that are at distance $2,3,$ etc. from $O$ and calling one of the unlabelled edges type 1 and the other one type 2. Now that all edges of $\mathcal{T}$ have types assigned to them, we continue by labeling the vertices. If a vertex $Y$ meets two edges of type 1 and one edge of type 2, we call it a vertex of type 1; otherwise, if it is incident with two edges of type 2 and one edge of type 1, we call it type 2. We do not need to assign any type to the root $O$. At a vertex \( Y\neq O \) of type \( \iota_Y \), see \eqref{imath}, we  define both operators  $\mathcal{L}_c^{(1)}$ and $\mathcal{L}_c^{(2)}$ by the same formula:
\begin{eqnarray}
\label{sad6}
(\mathcal{L}_c^{(l)}\psi)_{Y}=\sum_{j\in \{1,2\},{Y'}\sim {Y},\,{\rm type\, of\, edge}\, ({Y},{Y'})=j}\sqrt{A_{c,j}}\psi_{{Y}'}+{B_{c,\iota_Y}}\psi_{{Y}}, \quad l\in \{1,2\};
\end{eqnarray}
and at the root $O$ we define the operators  $\mathcal{L}_c^{(1)}$ and $\mathcal{L}_c^{(2)}$ differently from each other by
\begin{eqnarray*}
(\mathcal{L}_c^{(l)}\psi)_{O}=\sum_{j\in \{1,2\},{Y'}\sim {O},\,{\rm type\, of\, edge}\, ({O},{Y'})=j}\sqrt{A_{c,j}}\psi_{{Y}'}+{B_{c,l}}\psi_{{O}}\,, \quad l\in\{1,2\}.
\end{eqnarray*}
Notice that these operators represent Jacobi matrices on \( \mathcal T \) when $c\in (0,1)$. However, if $c\in \{0,1\}$ either $A_{c,1}$ or $A_{c,2}$ becomes zero and  $\mathcal{L}_c^{(1)}, \mathcal{L}_c^{(2)}$ are no longer  Jacobi matrices, strictly speaking. \smallskip
 
\noindent {\bf Remark.} The operators  $\mathcal{L}_c^{(1)}$ and $\mathcal{L}_c^{(2)}$ already appeared in \cite{uApDenY} as the strong limits of Jacobi matrices on finite trees that correspond to $\{P_{\vec{n}}\}$, the polynomials of the second type (see formula (3.3) and Subsection~4.5 in \cite{uApDenY}). We defined $\mathcal{L}_c^{(1)}$ and $\mathcal{L}_c^{(2)}$ by assigning the ``types'' to vertices of the tree and then defining the Jacobi matrix accordingly. This is an example of more general construction that generates trees satisfying a finite cone type condition. The Laplacian defined on trees with finite cone type and its perturbations  were studied in,  e.g., \cite{t6,at1,at2}. 

\begin{lemma}
If $\mathcal{J}$ has coefficients in \(\Par_{Ang}(\Delta_1,\Delta_2) \), then the $\mathcal{R}$-limits of $\mathcal{J}$ and the $\mathcal{R}$-limits of $\mathcal{L}^{(l)}_c$, \( l\in\{1,2\} \), are related by the following identity
\begin{equation}\label{sad1}
\bigcup_{c\in [0,1]} \left\{\mathcal J^\prime:\mathcal J^\prime~{\rm is \, an\,}\mathcal{R}-{\rm limit\, of \,}\mathcal{L}_c^{(l)}\right\} = \left\{\mathcal J^{\prime\prime}:\mathcal J^{\prime\prime}~{\rm is \, an\,}\mathcal{R}-{\rm limit\, of \,}\mathcal{J}\right\}.
\end{equation}
\end{lemma}
\begin{proof} This follows from the definition of the \(\mathcal R\)-limit, construction of $\mathcal{L}_c^{(1)}$ and $\mathcal{L}_c^{(2)}$, and from the assumption \eqref{AngPar_1}.
\end{proof}
We further study auxiliary operators $\mathcal{L}_c^{(1)}$ and $\mathcal{L}_c^{(2)}$ in Appendix~\ref{appendix}.

\begin{proof}[Proof of Theorem~\ref{thm:spectrum}]
 
Assumptions \eqref{AngPar_1} characterize the behavior of the coefficients at infinity. Thus, Weyl's theorem on the essential spectrum \cite{rs1} implies that any two Jacobi matrices with parameters in $ \Par_{Ang}(\Delta_1,\Delta_2) $ have the same essential spectra. Moreover, by the same Weyl's theorem, this essential spectrum is independent of the choice of parameter $\vec\kappa$ in \eqref{Jacobi}. Hence, it is enough to prove the theorem for the Jacobi matrix $\mathcal{J}_{\vec\kappa}$ generated by some Angelesco system with analytic weights and with $\vec\kappa=\vec{e}_2$. In \cite[Section 4]{uApDenY} we established that
$\Delta_1\cup\Delta_2\subseteq \sigma(\mathcal{J}_{\vec{e}_2})$. Thus, $\Delta_1\cup\Delta_2\subseteq \sigma_{\rm ess}(\mathcal{J}_{\vec{e}_2})$ as follows from the definition of the essential spectrum.

To prove the opposite inclusion, take any $\mathcal{J}$ for which the coefficients belong to $ \Par_{Ang}(\Delta_1,\Delta_2) $. The application of Theorem~\ref{ref_d} and Theorem~\ref{sad15} to  $\mathcal{L}_c^{(1)}$ gives
\[
\overline{\bigcup_{\mathcal{J}' {\rm \,is \, an\,} \mathcal{R}-{\,\rm limit \,of\,}\mathcal{L}_c^{(1)}} \sigma(\mathcal{J}')} = \sigma_{\rm ess}(\mathcal{L}_c^{(1)}) = \Delta_{c,1}\cup\Delta_{c,2},
\]
which yields an inclusion
\begin{equation}\label{sad3}
\bigcup_{c\in [0,1]}{\bigcup_{\mathcal{J}' {\rm \,is \, an\,} \mathcal{R}-{\,\rm limit \,of\,}\mathcal{L}_c^{(1)}} \sigma(\mathcal{J}')}\subseteq  \bigcup_{c\in [0,1]}\Bigl( \Delta_{c,1}\cup\Delta_{c,2}\Bigr)=\Delta_1\cup\Delta_2\,,
\end{equation}
where the last equality follows from the properties of $\Delta_{c,1}$ and $\Delta_{c,2}$ (which we also discuss later in Proposition \ref{prop:1}). Moreover, since
\[
\sigma_{\rm ess}(\mathcal{J})=\overline{\bigcup_{c\in [0,1]}{\bigcup_{\mathcal{J}' {\rm \,is \, an\,} \mathcal{R}-{\,\rm limit \,of\,}\mathcal{L}_c^{(1)}} \sigma(\mathcal{J}')}}\,
\]
by Theorem \ref{ref_d} and \eqref{sad1}, we get from \eqref{sad3} that \( \sigma_{\rm ess}(\mathcal{J})\subseteq \Delta_1\cup\Delta_2 \), which proves the theorem.
\end{proof}

\section{Multiple Orthogonal Polynomials for Angelesco Systems}
\label{sec:3}

In this section we state the results on asymptotic behavior of the forms \( Q_\n(x) \) and polynomials \( P_\n(x) \) defined in \eqref{typeI}
and \eqref{typeII}, respectively, along ray sequences \( \mathcal N_c=\mathcal N_{(c,1-c)}\) defined in \eqref{multi-indices} under the assumption that the measures of orthogonality are as in Theorem~\ref{thm:recurrenceOld}. The study of strong asymptotics of multiple orthogonal polynomials has a long history, see for example \cite{Kal79,NutTr87,Apt_Szego,Y16}. Below, we follow the Riemann-Hilbert approach used in \cite{Y16}, where the strong asymptotics of MOPs was derived for Angelesco systems with analytic weights for non-marginal ray sequences. Here, we extend the results of \cite{Y16} to marginal sequences, which is a  non-trivial problem requiring new ideas.

As before, we assume that the intervals \(\Delta_1=[\alpha_1,\beta_1] \) and \( \Delta_2=[\alpha_2,\beta_2] \) are disjoint and \( \beta_1<\alpha_2 \). In accordance with the definition of the intervals \( \Delta_{c,1},\Delta_{c,2} \) after \eqref{vecI}, we shall also set \( \Delta_{0,1}:=\{\alpha_1\},\Delta_{0,2}:=\Delta_2 \) and \( \Delta_{1,1}:=\Delta_1,\Delta_{1,2}:=\{\beta_2\} \).

Throughout the paper, we use the following notation: given a system of Jordan arcs and curves \( \Sigma \), we denote by \( \Sigma^\circ \) the subset of \( \Sigma \) consisting of points that possess a neighborhood that is separated by \( \Sigma \) into exactly two connected components. In particular, \( \Delta_i^\circ = (\alpha_i,\beta_i) \), \( i\in\{1,2\}\).

\subsection{Fully Marginal Ray Sequences}

In this subsection we consider solely infinite ray sequences of the form
\begin{equation}
\label{fullymarginal}
\mathcal N_{i-1}=\left\{\n:\text{ there exists } C>0 \text{ such that }n_i\leq C\right\}, \quad i\in\{1,2\}.
\end{equation}

To describe the asymptotics we need to introduce the so-called
Szeg\H{o} functions of the measures \( \mu_1,\mu_2 \). To this end,
let us set
\begin{equation}
\label{rhoi}
\rho_i(x) := -2\pi\ic\mu_i^\prime(x), \quad x\in\Delta_i.
\end{equation}
Observe that \( (\rho_iw_{i+})(x)>0 \) for \( x\in\Delta_i^\circ:=(\alpha_i,\beta_i) \), where \( w_i(z) \) was introduced in Proposition~\ref{prop:angelesco}. Put
\begin{equation}
\label{szego}
S_{\rho_i}(z) := \exp\left\{\frac{w_i(z)}{2\pi\ic}\int_{\Delta_i}\frac{\log(\rho_iw_{i+})(x)}{z-x}\frac{\dd x}{w_{i+}(x)}\right\}, \quad i\in\{1,2\}.
\end{equation}
Then each \( S_{\rho_i}(z) \) is a holomorphic and non-vanishing
function in \( \overline\C\setminus \Delta_i \) that is uniquely (up to a sign) characterized by the properties\footnote{\( A(z)\sim B(z) \) as \( z\to z_0 \) means that the ratio \( A(z)/B(z) \) is uniformly bounded away from zero and infinity as \( z\to z_0\).}
\begin{equation}
\label{szego-pts}
\left\{
\begin{array}{l}
(S_{\rho_i+}S_{\rho_i-})(x)(\rho_iw_{i+})(x) \equiv 1, \quad x\in\Delta_i^\circ, \medskip \\
|S_{\rho_i}(z)|\sim|z-x_*|^{-1/4}, \quad z\to x_*\in\{\alpha_i,\beta_i\}.
\end{array}
\right.
\end{equation}
Notice also that if \( \rho_i(x) \) is replaced by \( \rho_i(x)/w_{i+}(x) \) in \eqref{szego}, then \( S_{\rho_i/w_{i+}}(z) \) retains all the described properties except it is actually bounded around \( \beta_i \) and \( \alpha_i \). The following theorem holds.

\begin{theorem}
\label{thm:asymp1} Under the conditions of
Theorem~\ref{thm:recurrence}, it holds that
\[
P_\n(z) = (1+o(1)) \big(S_{\rho_2}(z)/S_{\rho_2}(\infty)\big)S^{n_1}(z;\alpha_1)(z-\alpha_1)^{n_1}\varphi_2^{n_2}(z)
\]
uniformly on bounded subsets of \(
\overline\C\setminus(\Delta_{0,1}\cup\Delta_2 ) \) along any \(
\mathcal N_0 \) satisfying \eqref{fullymarginal}, where \(
\varphi_2(z) \) was introduced in \eqref{varphii} and
\begin{equation}
\label{Sx0}
S(z;x_0) := \left(\frac{\varphi_2(z)-\varphi_2(x_0)}{\varphi_2(x_0)\varphi_2(z)-A_{0,2}}\frac{\varphi_2(x_0)\varphi_2(z)}{z-x_0}\right)^{1/2}, \quad z\in\overline\C\setminus\Delta_2,
\end{equation}
\( x_0\in(-\infty,\infty)\setminus\Delta_2 \) and the root is
chosen so that \( S(\infty;x_0) = 1 \). An analogous asymptotic
formula holds along \( \mathcal N_1 \) satisfying
\eqref{fullymarginal}.
\end{theorem}

Since \( \varphi_{2+}(x)\varphi_{2-}(x) \equiv A_{0,2} \) for \( x\in\Delta_2 \), an explicit computation shows that 
\[
S(x;x_0)_+S(x;x_0)_- = |S(x;x_0)_\pm|^2 \equiv -\varphi_2(x_0)(x-x_0)^{-1}, \quad x\in\Delta_2^\circ. 
\]
As \( S(z;x_0) \) is non-vanishing and holomorphic in \( \overline\C\setminus\Delta_2 \) as well as bounded around \( \alpha_2,\beta_2 \) a standard argument shows that
\[
S(z;x_0) = S_\varrho(z)/S_\varrho(\infty), \quad S_\varrho^{-1}(\infty)=\sqrt{-\varphi_2(x_0)}, \quad \varrho(x) := (x-x_0)/w_{2+}(x),
\]
where \( S_\varrho(\infty)>0 \) when \( x_0<\alpha_2 \) while \( S_\varrho(\infty)\in\ic\R \) when \( x_0>\beta_2 \) with the choice of the square root depending on the determination of \( \log(x-x_0) \) used. We prove Theorem~\ref{thm:asymp1} in Section~\ref{sec:6}.

\subsection{Szeg\H{o} Functions on \( \RS_c \)}

Let us set \( \boldsymbol\Delta_{c,i} := \pi^{-1}(\Delta_{c,i}) \),
\( i\in\{1,2\} \), and orient it so that \( \RS_c^{(0)} \) remains
on the left when the cycle is traversed in the positive direction.
Put
\begin{equation}
\label{wi}
w_{c,i}(z):=\sqrt{(z-\alpha_{c,i})(z-\beta_{c,i})} = z + \mathcal O(1), \quad z\to\infty,
\end{equation}
to be the branch holomorphic outside of $\Delta_{c,i}$. In what
follows, it will be convenient to introduce the following notation
\[
F^{(k)}(z) := F\big(z^{(k)} \big), \quad k\in\{0,1,2\},
\]
for a function \( F(\z) \) defined on \(
\RS_c\setminus(\boldsymbol\Delta_{c,1}\cup\boldsymbol\Delta_{c,2})
\). Then the following proposition holds.

\begin{proposition}
\label{prop:szego} Given \( c\in(0,1) \) and functions \(
\rho_1(x),\rho_2(x) \) as in \eqref{rhoi} and
Theorem~\ref{thm:recurrenceOld}, there exists a function $S_c(\z)$
non-vanishing and holomorphic in
$\RS_c\setminus(\boldsymbol\Delta_{c,1}\cup\boldsymbol\Delta_{c,2})$
such that
\begin{equation}
\label{szego-pts2}
\left\{
\begin{array}{l}
S_{c\pm}^{(i)}(x) = S_{c\mp}^{(0)}(x)(\rho_iw_{c,i+})(x), \quad x\in \Delta_{c,i}, \medskip \\
\big(S_c^{(0)}S_c^{(1)}S_c^{(2)}\big)(z) \equiv 1, \quad z\in \overline\C, \medskip \\
\big|S_c^{(0)}(z)\big| \sim |z-x_0|^{-1/4} \quad \text{as} \quad z\to x_0\in\big\{\alpha_1,\beta_{c,1},\alpha_{c,2},\beta_2\big\}.
\end{array}
\right.
\end{equation}
Properties \eqref{szego-pts2} determine \( S_c(\z) \) uniquely up
to a multiplication by a cubic root of unity. Moreover, if \( c\to
c_\star \in(0,1) \), then
\begin{equation}
\label{szego-limit1}
S_c^{(k)}(z) = \big[1+o(1)\big] S_{c_\star}^{(k)}(z),
\end{equation}
locally uniformly in \( \overline \C\setminus\Delta_{c_\star,k}\) when
\( k\in\{1,2\} \), and in \( \overline \C\setminus
(\Delta_{c_\star,1}\cup \Delta_{c_\star,2})\) when \( k=0 \). Furthermore,
it holds that
\begin{equation}
\label{szego-limit2}
\frac{S_c^{(k)}(z)}{S_c^{(k)}(\infty)} = \big(1+o(1)\big)
\left\{
\begin{array}{rl}
S_{\rho_2}(z)/S_{\rho_2}(\infty), & k=0, \medskip \\
1, & k=1, \medskip \\
S_{\rho_2}(\infty)/S_{\rho_2}(z), & k=2,
\end{array}
\right.
\end{equation}
 as \( c\to0 \), where \( o(1) \) holds locally uniformly in \( \overline\C\setminus\Delta_{0,1} \) when \( k\in\{0,1\} \) and uniformly in \( \overline\C \) when \( k=2 \) (that is, including the traces on \( \Delta_2 \)), while it also holds that
\begin{equation}
\label{szego-limit3}
\lim_{c\to0} S_c^{(0)}(\infty)c^{1/3}= VS_{\rho_2}(\infty), \quad \lim_{c\to0} S_c^{(1)}(\infty)c^{-2/3} = V^{-2}, \quad \text{and} \quad \lim_{c\to0} S_c^{(2)}(\infty)c^{1/3} = V/S_{\rho_2}(\infty),
\end{equation}
where \(  V:=\big(2\pi\mu^\prime_1(\alpha_1)|w_2(\alpha_1)|S_{\rho_2}(\alpha_1)\big)^{-1/3}\). Limits analogous to \eqref{szego-limit2} and \eqref{szego-limit3} also hold as \( c\to
1 \).
\end{proposition}

The construction leading to Proposition~\ref{prop:szego} is not new. As soon as strong asymptotics of MOPs became a question of interest, it was well understood that classical Szeg\H{o} functions need to be replaced by solutions to a boundary value problem \eqref{szego-pts2}. The original approach reformulated \eqref{szego-pts2} as a certain extremal problem, see \cite{Apt_Szego}. Another approach using discontinuous Cauchy kernels on the corresponding Riemann surface was developed in \cite{ApLy10}. The latter construction is exactly the one we adopt in Section~\ref{sec:5} to prove Proposition~\ref{prop:szego}. Even though out of necessity, but unlike previous works, we do examine here what happens to the Szeg\H{o} functions \( S_c(\z) \) when one of the intervals \( \Delta_{c,1} \), \( \Delta_{c,2} \) is collapsing.

\subsection{Non-Fully Marginal and Non-Marginal Ray Sequences}
\label{sec:3.3}

In this section we assume that sequences \( \mathcal N_c \), \(
c\in[0,1] \), satisfy
\begin{equation}
\label{vareps}
\varepsilon_\n := 1/\min\{n_1,n_2\} \to 0 \quad as \quad |\n|\to\infty, \quad \n\in\mathcal N_c.
\end{equation}

We start by introducing an analog of the functions \(
\varphi_1(z),\varphi_2(z) \) in the non-fully marginal and non-marginal cases. Given a
multi-index \( \n \), let
\begin{equation}
\label{cn}
c_\n := n_1/|\n|.
\end{equation}
To alleviate the notation, in what follows we shall use the subindex \( \n \) instead of \( c_\n \) for quantities depending on \( c_n \) such that \( \RS_\n = \RS_{c_\n} \), \( S_\n(\z) = S_{c_\n}(\z) \), etc. We shall denote by $\Phi_\n(\z)$ a rational function on $\RS_\n$ which is non-zero and finite everywhere except at the points on top of infinity, has a pole of order $|\n|$ at
$\infty^{(0)}$, a zero of multiplicity $n_i$ at 
$\infty^{(i)}$ for each \( i\in\{1,2\} \), and satisfies
\begin{equation}
\label{normalization}
\big(\Phi_{\n}^{(0)}\Phi_{\n}^{(1)}\Phi_{\n}^{(2)}\big)(z) \equiv 1, \quad z\in\overline\C.
\end{equation}
Equality in \eqref{normalization} is a simple matter of a
normalization since the logarithm of the absolute value of the
left-hand side of \eqref{normalization} extends to a harmonic
function on $\C$ which has a well defined limit at infinity and
therefore is a constant.

\begin{theorem}
\label{thm:asymp2} Under the conditions of
Theorem~\ref{thm:recurrence}, let \( P_\n(z) \) be the polynomials
satisfying \eqref{typeII}. Given \( c\in[0,1] \), let \( \mathcal
N_c=\{\n\} \) be a sequence for which \eqref{vareps} holds. Then \( \n\in\mathcal N_c \) we have
that
\[
\left\{
\begin{array}{lll}
P_\n(z) &=& (1 + o(1)) \gamma_\n\big(S_\n\Phi_\n\big)^{(0)}(z), \medskip \\
P_\n(x) &=& (1 + o(1)) \gamma_\n\big(S_\n\Phi_\n\big)_+^{(0)}(x) + (1 + o(1)) \gamma_\n\big(S_\n\Phi_\n\big)_-^{(0)}(x),
\end{array}
\right.
\]
where the relations holds uniformly on closed subsets of \( \overline \C \setminus (\Delta_{c,1}\cup \Delta_{c,2}) \) and  compact subsets \( \Delta_{c,1}^\circ\cup\Delta_{c,2}^\circ \), respectively, and \( \gamma_\n \) is the constant such that
\[
\lim_{z\to\infty}\gamma_\n z^{|\n|}\big(S_\n\Phi_\n\big)^{(0)}(z)=1.
\]
When \( c\neq c^*,c^{**} \), see Proposition~\ref{prop:1} further below, the error rate \( o(1) \) can be replaced by \( \mathcal O_c(\varepsilon_\n) \), where the dependence of \( \mathcal O_c(\varepsilon_\n) \) on \( c \) is uniform for \( c \) on compact subsets \( [0,1]\setminus\{c^*,c^{**}\}\).
\end{theorem}

In the above theorem the functions \( S_\n^{(0)}(z) \) could be replaced by their limits as discussed in Proposition~\ref{prop:szego}. However, we can do this only at the expense of  the error rate \( \mathcal O_c(\varepsilon_\n) \).

To describe asymptotic behavior of the forms \( Q_\n(x) \), we need
to introduce one additional function. Let \( \Pi_\n(\z) \) be a
rational function on \( \RS_\n \) with the zero/pole divisor and
the normalization given by
\[
2\big(\infty^{(1)} + \infty^{(2)} \big) - \boldsymbol\alpha_1 - \boldsymbol\beta_{\n,1} - \boldsymbol\alpha_{\n,2} - \boldsymbol\beta_2 \qandq \Pi_\n^{(0)}(\infty) =1,
\]
where \(
\boldsymbol\alpha_1,\boldsymbol\beta_{\n,1},\boldsymbol\alpha_{\n,2},\boldsymbol\beta_2\)
are the ramification points of \( \RS_\n \). Then the following
theorem holds.

\begin{theorem}
\label{thm:asymp3} Under the conditions of
Theorem~\ref{thm:recurrence}, let \( A_\n^{(i)}(z) \) be the
polynomials defined in \eqref{typeI}, \( i\in\{1,2\} \). Given \(
c\in[0,1] \), let \( \mathcal N_c=\{\n\} \) be a sequence for which
\eqref{vareps} holds. Then for \(
\n\in\mathcal N_c \) we have that
\[
A_\n^{(i)}(z) =  -(1 + o(1))\frac{(\Pi_\n^{(i)}w_{\n,i})(z)}{\gamma_\n(S_\n\Phi_\n)^{(i)}(z)}, \medskip \\
\]
uniformly on closed subsets of \( \overline \C \setminus
\Delta_{c,i} \) for \( i\in\{1,2\} \) when \( c\in(0,1) \), \( i=2
\) when \( c=0 \), and \( i =1 \) when \( c=1 \), while
\[
A_\n^{(i)}(z) = o(1) \left(\tau_\n\big(w_{\n,i}\Phi_\n^{(i)}\big)(z)\right)^{-1},
\]
uniformly on closed subsets of \( \overline \C \setminus
\Delta_{0,1} \) for \( i=1 \) when \( c=0 \) and of \( \overline \C
\setminus \Delta_{1,2} \) for \( i=2 \) when \( c=1 \), where \(
\tau_\n:=\gamma_\n S_\n^{(0)}(\infty) \), i.e., it is a constant
such that \( \lim_{z\to\infty}\tau_\n|z|^{|\n|}\Phi_\n^{(0)}(z) =1
\). Moreover,
\[
A_\n^{(i)}(x) = -(1 + o(1)) \frac{(\Pi_\n^{(i)}w_{\n,i})_+(x)}{\gamma_\n(S_\n\Phi_\n)_+^{(i)}(x)}  - (1 + o(1)) \frac{(\Pi_\n^{(i)}w_{\n,i})_-(x)}{\gamma_\n(S_\n\Phi_\n)_-^{(i)}(x)},
\]
uniformly on compact subsets of \( \Delta_{c,i}^\circ \), \( i\in\{1,2\} \). As in the case of Theorem~\ref{thm:asymp2}, the error rate can be improved to \( \mathcal O_c(\varepsilon_\n) \) when \( c\in[0,1]\setminus\{c^*,c^{**}\} \) with dependence on \( c \) being locally uniform.
\end{theorem}

Let $\big(\widehat\mu_1,\widehat\mu_2\big)$ be a vector of Markov
functions of the measures \( \mu_i \), that is,
\[
\widehat\mu_i(z) := \int\frac{\dd\mu_i(x)}{z-x} = \frac1{2\pi\ic}\int_{\Delta_i}\frac{\rho_i(x)}{x-z}\:\dd x, \qquad z\in\overline\C\setminus\Delta_i, \,  i\in\{1,2\}.
\]
Observe also that \( (\widehat\mu_{i+} - \widehat\mu_{i-})(x) =
\rho_i(x) \), \( x\in \Delta_i^\circ \), by Sokhotski-Plemelj
formulae. Then one can deduce from orthogonality relations
\eqref{typeII} that there exist polynomials \( P_\n^{(i)}(z) \)
such that
\[
R_\n^{(i)}(z) := \big(P_\n \widehat\mu_i - P_\n^{(i)}\big)(z) = \mathcal{O}\big(z^{-(n_i+1)}\big) \quad \text{as} \quad z\to\infty,
\]
\( i\in\{1,2\} \). The vector of rational functions \(
\big(P_\n^{(1)}/P_\n,P_\n^{(2)}/P_\n\big) \) is called the
\emph{Hermite-Pad\'e approximant} for \(
\big(\widehat\mu_1,\widehat\mu_2\big) \) corresponding to the
multi-index \( \n \). It further can be shown that
\begin{equation}
\label{Rni}
R_\n^{(i)}(z) =  \frac1{2\pi\ic}\int\frac{(P_\n\rho_i)(x)}{x-z}\dd x, \quad z\in\overline\C\setminus\Delta_i, \,  i\in\{1,2\}.
\end{equation}
It also follows from \eqref{typeI} that there exists polynomial \(
A_\n(x) \) such that
\begin{equation}
\label{Ln}
\mathcal O\big(z^{-|\n|}\big)= \sum_{i=1}^2 \big(A_\n^{(i)}\widehat\mu_i\big)(z) - A_\n(z) =: L_\n(z)  = \int \frac{Q_\n(x)}{z-x},
\end{equation}
where the asymptotic formula is valid for \( z\to\infty \). Then
the following result holds.

\begin{theorem}
\label{thm:asymp4} Under the conditions of Theorems~\ref{thm:asymp2}--\ref{thm:asymp2}, it holds for \( \n\in\mathcal N_c \) that
\[
R_\n^{(i)}(z) = (1 + o(1)) \gamma_\n\big(S_\n\Phi_\n\big)^{(i)}(z)w_{\n,i}^{-1}(z),
\]
uniformly on closed subsets of \( \overline\C\setminus\Delta_{c,i}
\), that is, including the traces on \(
\Delta_i\setminus\Delta_{c,i} \) for \( i\in\{1,2\} \) when \(
c\in(0,1) \), for \( i=2 \) when \( c=0 \), and for \( i=1 \) when
\( c =1 \), while
\[
R_\n^{(i)}(z) = o(1) \tau_\n\Phi_\n^{(i)}(z)w_{\n,i}^{-1}(z)
\]
uniformly on closed subsets of \( \overline\C\setminus\Delta_{0,1} \) for \( i=1 \) when \( c=0 \) and of \(
\overline\C\setminus\Delta_{1,2} \) for \( i=2 \) when \( c=1 \).
Moreover,
\[
L_\n(z) = (1 + o(1))\frac{\Pi_\n^{(0)}(z)}{\gamma_\n(S_\n\Phi_\n)^{(0)}(z)},
\]
uniformly on closed subsets of \( \overline\C\setminus
(\Delta_{c,1}\cup\Delta_{c,2}) \).  As in Theorems~\ref{thm:asymp2}
and~\ref{thm:asymp3} the error rate can be improved to \( \mathcal
O_c(\varepsilon_\n) \) when \( c\in[0,1]\setminus\{c^*,c^{**}\} \) with dependence on \( c \) being locally uniform.
\end{theorem}

Theorems~\ref{thm:asymp2}--\ref{thm:asymp4} are proven in
Chapter~\ref{sec:8}.

\section{On the Supports of the Equilibrium measures}
\label{sec:4}

In this section we discuss further properties of the  vector
equilibrium problem \eqref{vecI}--\eqref{var} as well as prove some auxiliary
lemmas needed later.

With the notation introduced in the beginning of Section~\ref{ss:21}, the following
proposition holds.

\begin{proposition}
\label{prop:1} There exist constants \( 0<c^* < c^{**}<1 \) such
that
\[
\left\{
\begin{array}{lll}
\beta_{c,1}<\beta_1, & \alpha_{c,2}=\alpha_2, &  0<c<c^*, \medskip \\
\beta_{c,1}=\beta_1, & \alpha_{c,2}=\alpha_2, & c^*\leq c \leq c^{**}, \medskip \\
\beta_{c,1}=\beta_1, & \alpha_{c,2}>\alpha_2, & 1>c>c^{**}.
\end{array}
\right.
\]
Moreover, it holds that\footnote{Given compactly supported measures \( \nu_n \), \( n\in\ZZ \), \( \nu_n\cws\nu_0 \) as \( n\to\infty \) means that \( \int f\dd\nu_n\to\int f\dd\nu_0 \) as \( n\to\infty \) for any compactly supported continuous function \( f \).}
\[
\begin{array}{lllllll}
\omega_{c,i} \cws \omega_{c_\star,i}, & \alpha_{c,2}\to \alpha_{c_\star,2}, & \beta_{c,1}\to \beta_{c_\star,1}, & \ell_{c,i}\to\ell_{c_\star,i}, & V^{\omega_{c,i}}\to V^{\omega_{c_\star,i}} & \text{as} & c\to c_\star\in(0,1)
\end{array}
\]
for \( i\in\{1,2\} \), where the convergence of potentials is uniform on compact subsets of $\C$. Furthermore,
\[
\left\{
\begin{array}{llllll}
\omega_{c,2} \cws \omega_2, & \beta_{c,1}\to \alpha_1, & \ell_{c,2}\to2\ell_2, &\ell_{c,1} \to V^{\omega_2}(\alpha_1) & \text{as} & c\to0, \medskip \\
\omega_{c,1} \cws \omega_1, & \alpha_{c,2}\to \beta_2, & \ell_{c,1}\to2\ell_1, &\ell_{c,2} \to V^{\omega_1}(\beta_2) & \text{as} & c\to1,
\end{array}
\right.
\]
and \( V^{\omega_{c,i}}\to V^{\omega_i} \) uniformly on compact subsets of \( \C \) as \( c\to 2-i \), \( i\in\{1,2\} \).
\end{proposition}

Further, recall the surface \( \RS_c \) constructed just before
Proposition~\ref{prop:angelesco}. Given a rational function \(
F(\z) \) on \( \RS_c \), we denote its divisor of zeros and poles by
\( (F) \) and write
\[
(F) = m_1\boldsymbol z_1 + \cdots + m_l \boldsymbol z_l - k_1 \boldsymbol p_1 - \cdots - k_t \boldsymbol p_t
\]
to mean that \( F(\z) \) has a zero of order \( m_i \) at \(
\boldsymbol z_i \) for each \( i\in\{1,\ldots,l\} \), a pole of order \( k_i \) at \(
\boldsymbol p_i \) for each \( i\in\{1,\ldots,t\} \), and otherwise it is non-vanishing and finite,
where necessarily \( \sum_{i=1}^l m_i = \sum_{i=1}^tk_i \).

It can be easily checked using Schwarz reflection principle, as it was done in \cite[Proposition~2.1]{Y16} for \( c \) rational, that the function
\begin{equation}
\label{Hc}
H_c(\z) :=
\left\{
\begin{array}{ll}
-V^{\omega_{c,1}+\omega_{c,2}}(z) + \frac{\ell_{c,1}+\ell_{c,2}}3, & \z\in\RS_c^{(0)}, \medskip \\
V^{\omega_{c,i}}(z) - \ell_{c,i} +  \frac{\ell_{c,1}+\ell_{c,2}}3, & \z\in\RS_c^{(i)}, \quad i\in\{1,2\},
\end{array}
\right.
\end{equation}
is harmonic on \(
\RS_c\setminus\big\{\infty^{(0)},\infty^{(1)},\infty^{(2)}\big\}
\). Therefore, the function \( h_c(\z) := 2\partial_z H_c(\z) \),
where \( 2\partial_z := \partial_x - \ic \partial_y \), is rational
on \( \RS_c \).  In fact, it holds that
\begin{equation}
\label{int-rep}
\left\{
\begin{array}{ll}
h_c^{(0)}(z) = \displaystyle \int\frac{\dd(\omega_{c,1}+\omega_{c,2})(x)}{z-x}, & z\in \C\setminus\big(\Delta_{c,1} \cup \Delta_{c,2} \big), \medskip \\
h_c^{(i)}(z) = \displaystyle \int\frac{\dd\omega_{c,i}(x)}{x-z}, & z\in \C\setminus \Delta_{c,i}, \quad i\in\{1,2\}.
\end{array}
\right.
\end{equation}
The importance of this function lies in the following: it was shown
in \cite[Propositions~2.1 and~2.3]{Y16} that
\begin{equation}
\label{PhiInt}
\Phi_\n(\z) = C_\n\exp\left\{|\n|\int_{\boldsymbol \beta_2}^\z h_{c_\n}(\x)\dd x \right\} \qandq \frac1{|\n|}\log\big|\Phi_{\n}(\z)\big| = H_{c_\n}(\z)
\end{equation}
for \( \z\in \RS_\n \), where the constant \( C_\n \) should be chosen so that \eqref{normalization} is satisfied.

\begin{proposition}
\label{prop:2} Let \( \mathcal D_c := \boldsymbol\alpha_1 +
\boldsymbol\beta_{c,1} + \boldsymbol\alpha_{c,2} +
\boldsymbol\beta_2 \) be the divisor of the ramification points of \( \RS_c \). It holds that
\begin{equation}
\label{divisor-hc}
(h_c) = \infty^{(0)} + \infty^{(1)} + \infty^{(2)} + \z_c - \mathcal D_c
\end{equation}
for some \( \z_c \in \RS_c^{(0)} \) such that \(
z_c\in[\beta_{c,1},\alpha_{c,2}] \). Moreover, \( z_c \) is a
continuous increasing function of \( c \) and
\[
\left\{
\begin{array}{ll}
z_c=\beta_{c,1}, &  c\leq c^*, \medskip \\
z_c=\alpha_{c,2}, & c\geq c^{**}.
\end{array}
\right.
\]
\end{proposition}

This proposition has the following implication: \emph{point \( z_c \) uniquely determines the vector equilibrium measure \( (\omega_{c,1},\omega_{c,2}) \)}.  Indeed, choose \( z_\star\in(\alpha_1,\beta_2) \). Set \( \beta_{\star,1}=\min\{\beta_1,z_\star\} \) and \( \alpha_{\star,2}=\max\{\alpha_2,z_\star\} \). Construct Riemann surface \( \RS_\star \) with respect to the cuts \( [\alpha_1,\beta_{\star,1}] \) and \( [\alpha_{\star,2},\beta_2] \) as before. Let \( h_\star(\z) \) be a rational function on \( \RS_\star \) with the zero/pole divisor
\[
(h_\star) = \infty^{(0)} + \infty^{(1)} + \infty^{(2)} + \z_\star - \boldsymbol \alpha_1 - \boldsymbol \beta_{\star,1} - \boldsymbol \alpha_{\star,2} - \boldsymbol \beta_2,
\]
where \( \boldsymbol \alpha_1,\boldsymbol \beta_{\star,1},\boldsymbol \alpha_{\star,2},\boldsymbol \beta_2\) are the ramification points of \( \RS_\star \) and \( \z_\star \in\RS_*^{(0)} \). Clearly, \( h_\star(z^{(0)})+h_\star(z^{(1)})+h_\star(z^{(2)})\equiv0 \) as this sum must be an entire function that vanishes at infinity. Normalize \( h_\star(\z) \) so that \( h_\star(z^{(0)}) = 1/z + \mathcal O(1/z^2)\) as \( z\to\infty \). Set \( c_\star := -\lim_{z\to\infty}zh(z^{(1)}) \). Then \( \RS_\star = \RS_{c_\star} \), \( \z_\star = \z_{c_\star} \), and respectively \( h_\star(\z)=h_{c_\star}(\z) \). It further follows from Privalov's lemma \cite[Section~III.2]{Privalov} that
\[
\dd\omega_{c_\star,i}(x) =\left(h_{\star+}^{(i)}(x) - h_{\star-}^{(i)}(x)\right)\frac{\dd x}{2\pi\ic}, \quad i\in\{1,2\},
\]
and thus, we have recovered the vector equilibrium measure from \( z_\star \).

\begin{proof}[Proof of Propositions~\ref{prop:1} and~\ref{prop:2}]

Besides relations \eqref{var}, it also holds that the left hand sides of \eqref{var} are strictly less than zero on \( \Delta_1\setminus\Delta_{c,1} \) and \( \Delta_2\setminus\Delta_{c,2} \), respectively, see \cite{GRakh81}. In particular, we can write
\[
V^{\frac1c \omega_{c,1}}(x) + \frac1{2c}V^{\omega_{c,2}}(x) - \frac{\ell_{c,1}}{2c} \left\{ \begin{array}{lcl} \equiv0 & \text{on} & \supp(\omega_{c,1}), \medskip \\ \geq 0 & \text{on} & \Delta_1\setminus\supp(\omega_{c,1}), \end{array} \right.
\] 
which, in view of \cite[Theorem~I.3.3]{SaffTotik}, can be interpreted in the following way:  the measure \( \frac1c\omega_{c,1} \) is the weighted logarithmic equilibrium distribution on \( \Delta_1 \) in the presence of the external field \( \frac1{2c} V^{\omega_{c,2}}(x) \).  Hence, its
support maximizes the Mhaskar-Saff functional \cite[Chapter~IV]{SaffTotik}:
\[
F_c(K) := \log\cp(K) - \frac1{2c}\int V^{\omega_{c,2}}\dd\omega_K,
\]
where \( K\subseteq[\alpha_1,\beta_1] \) is compact, \( \cp(K) \) is the logarithmic capacity of \( K \), and \( \omega_K \) is the logarithmic equilibrium distribution on \( K \) (when \( K \) is an interval, \( \omega_K \) is the arcsine distribution on \( K \)). As mentioned before \eqref{var}, the maximizer of this functional is an interval containing \( \alpha_1 \) (this was proven in \cite{GRakh81}). Therefore, it is enough to consider compact sets \( K \) only of the form \( [\alpha_1,\beta] \). Thus, the functional \( F(K) \) reduces to the function
\[
F_c(\beta) := \log\frac{\beta-\alpha_1}4 -\frac1{2c}\int_{\alpha_1}^\beta V^{\omega_{c,2}}(x)\frac{\dd x}{\pi\sqrt{(\beta-x)(x-\alpha_1)}},
\]
where we used explicit expressions for the logarithmic capacity and
the equilibrium measure of an interval. To find the maximum of \( F_c(\beta) \) on \( \Delta_1 \), let us compute its derivative. To this end, it can be readily checked
that
\begin{multline*}
\frac1h\left(\int_{\alpha_1}^{\beta+h} f(x)\frac{\dd
x}{\pi\sqrt{(\beta+h-x)(x-\alpha_1)}}-\int_{\alpha_1}^\beta
f(x)\frac{\dd x}{\pi\sqrt{(\beta-x)(x-\alpha_1)}}\right) \\ =
\int_{\alpha_1}^\beta
\frac1h\left(f\left(x+h\frac{x-\alpha_1}{\beta-\alpha_1}\right) -
f(x)\right) \frac{\dd x}{\pi\sqrt{(\beta-x)(x-\alpha_1)}}
\end{multline*}
for every differentiable function \( f(x) \) on \( \Delta_1 \). Observe also that \( V^{\omega_{c,2}}(x) \) is harmonic off \( \Delta_2 \) and therefore \( f_c(x):=V^{\omega_{c,2}}(x) = - \int\log|x-y|\dd\omega_{c,2}(y) \) is a smooth function on \( \Delta_1 \). Hence, by taking the limit as \( h\to 0 \) in the above equality, we get
\begin{equation}
\label{Fprime}
F_c^\prime(\beta) = \frac 1{\beta-\alpha_1}- \frac1{4\pi c}\int_{-1}^1f_c^\prime\left(\frac{\beta-\alpha_1}2x+\frac{\beta+\alpha_1}2\right)\sqrt{\frac{1+x}{1-x}}\dd x.
\end{equation}
It is also obvious that \( f_c^\prime(x) = \int
(y-x)^{-1}\dd\omega_{c,2}(y), \) which is an increasing positive
function on \( \Delta_1 \). Thus, \( F_c^\prime(\beta) \) is a
decreasing function of \( \beta \) and therefore has at most one
zero. Moreover, it holds that
\begin{equation}
\label{fprime-bounds}
\frac {1-c}{\beta_2-\alpha_1} < f_c^\prime(x) < \frac{1-c}{\alpha_2-\beta_1}, \quad x\in\Delta_1.
\end{equation}
Hence, \( F_c^\prime(\beta_1)<0 \) for all \( c \) small. As \( \lim_{\beta\to\alpha_1^+}F_c^\prime(\beta) = + \infty \), we get that \( \beta_{c,1} \in (\alpha_1,\beta_1) \) for all \( c \) small. Using \( F_c^\prime(\beta_{c,1}) = 0 \) and the above estimates, we get from \eqref{Fprime} that
\begin{equation}
\label{bc1}
\frac{4c}{1-c}(\alpha_2-\beta_1) < \beta_{c,1} - \alpha_1 < \frac{4c}{1-c}(\beta_2-\alpha_1)
\end{equation}
for all small \( c \). This, in particular, implies that  \( \beta_{c,1}\to \alpha_1 \) as \( c\to0 \). An analogous argument shows that \( \alpha_{c,2} \) approaches \( \beta_2 \) when \( c\to1 \). It further follows from \eqref{fprime-bounds} that \( f_c^\prime(x) \) uniformly converges to zero on \( \Delta_1 \) as \( c\to 1 \). Thus, \( F_c^\prime(\beta)>0 \) for all \( \beta\in\Delta_1 \) and all \( c \) close to \( 1 \). That is, \( \Delta_{c,1}=\Delta_1 \) in this case. Similarly, we also get that \( \Delta_{c,2} = \Delta_2 \) for all \( c \) small. 

Let us now describe what happens to the components of the vector equilibrium measure and their potentials as \( c\to0 \). Clearly,  \(
V^{\omega_{c,1}}(z) \to 0 \) uniformly on compact subsets of \( \C\setminus \Delta_{0,1} \) in this case. To show that \( \omega_{c,2}\cws \omega_2 \) as \( c\to 0 \), notice that
\[
\|\sigma\|\ell_2 = \int V^{\omega_2}\dd\sigma = \int V^\sigma\dd\omega_2 \left\{\begin{array}{l} \geq \inf_{\Delta_2}V^\sigma, \medskip \\ \leq \sup_{\Delta_2}V^\sigma, \end{array} \right.
\]
for any Borel measure \( \sigma \) supported on \( \Delta_2 \) since \( \omega_2 \) is a probability measure. It follows from \eqref{var} that \( V^{\omega_{c,2}}(x) \) is continuous on \( \Delta_2=\Delta_{c,2} \). Therefore,
\[
\left\{
\begin{array}{l}
\displaystyle 2\ell_2(1-c) \geq \min_{\Delta_2}V^{2\omega_{c,2}} = V^{2\omega_{c,2}}(x_{\min}) = \ell_{c,2} - V^{\omega_{c,1}}(x_{\min}) =  \ell_{c,2} +o(1), \medskip \\
\displaystyle 2\ell_2(1-c) \leq \max_{\Delta_2}V^{2\omega_{c,2}} = V^{2\omega_{c,2}}(x_{\max}) = \ell_{c,2} - V^{\omega_{c,1}}(x_{\max}) =  \ell_{c,2} +o(1),
\end{array}
\right.
\]
which implies that \( \ell_{c,2} = 2\ell_2 + o(1) \) as \( c\to 0 \). Let \( \omega \) be a weak\(^*\) limit point of \(
\omega_{c,2} \) as \( c\to 0 \). Then \( \omega \) is a probability
measure and
\[
V^\omega(x) \leq  \liminf_{c\to0} V^{\omega_{c,2}}(x) = \liminf_{c\to0} \big( \ell_{c,2} - V^{\omega_{c,1}}(x)\big)/2 = \ell_2, \quad x\in \Delta_2,
\]
where the first inequality follows from the Principle of Descent \cite[Theorem~I.6.8]{SaffTotik}.
Therefore, \( E(\omega,\omega)\leq \ell_2 = E(\omega_2,\omega_2)
\), which implies that \( \omega=\omega_2 \) by the uniqueness of
the equilibrium measure. To deduce the behavior of the constants \( \ell_{c,1} \) as \( c\to0 \), observe that
\[
\left\{
\begin{array}{ll}
V^{2\omega_{c,1}+\omega_{c,2}}(x) \leq \ell_{c,1}, & x\in(-\infty,\alpha_1], \medskip \\
V^{2\omega_{c,1}+\omega_{c,2}}(x) \geq \ell_{c,1}, & x\in[\beta_{c,1},\beta_1],
\end{array}
\right.
\]
where the first claim can be easily obtained from \eqref{var} and the second one was already mentioned at the beginning of the proof. Then
\[
V^{2\omega_{c,1}+\omega_{c,2}}(\alpha_1-\epsilon) \leq  \ell_{c,1} \leq V^{2\omega_{c,1}+\omega_{c,2}}(\alpha_1+\epsilon)
\]
for any  \( \epsilon>0 \) since \( \beta_{c,1}<\alpha_1+\epsilon \)
for all \( c \) small enough. Hence, we get that
\[
V^{\omega_2}(\alpha_1-\epsilon) \leq \liminf_{c\to0}\ell_{c,1} \leq \limsup_{c\to0}\ell_{c,1} \leq V^{\omega_2}(\alpha_1+\epsilon).
\]
Since \( V^{\omega_2}(x) \) is continuous on the real line and \(
\epsilon \) is arbitrary, we get that \( \ell_{c,1} \to
V^{\omega_2}(\alpha_1) \) as \( c\to 0\). The respective claims for
the limits as \( c\to 1\) can be shown in a similar fashion.

Let us point out one consequence of the fact that \(
\omega_{c,2}\cws\omega_2 \) as \( c\to0 \) that will be useful to
us later. It holds that
\[
f_c^\prime(z) := \int \frac{\dd\omega_{c,2}(y)}{y-z} \to \int \frac{\dd\omega_2(y)}{y-z} = \frac1\pi\int_{\alpha_2}^{\beta_2}\frac{1}{y-z}\frac{\dd y}{\sqrt{(y-\alpha_2)(\beta_2-y)}} = - \frac1{w_2(z)},
\]
locally uniformly in \( \overline\C\setminus\Delta_2 \), where, as
before, \( w_2(z):=\sqrt{(z-\alpha_2)(z-\beta_2)} \). Therefore, we
can improve \eqref{bc1} to
\begin{equation}
\label{c-rate}
\frac{4c}{\beta_{c,1}-\alpha_1} = \frac1{|w_2(\alpha_1)|} + o(1)
\end{equation}
as \( c\to0 \), where we again used \eqref{Fprime}.

The facts that \( \omega_{c,i}\cws\omega_{c_\star,i} \) and \( \ell_{c,i}\to\ell_{c_\star,i} \) as \( c\to c_\star\in(0,1) \), \( i\in\{1,2\} \), were shown in the proof of \cite[Proposition~2.1]{Y16}. Let us now show that \( \beta_{c,1}\to \beta_{c_\star,1} \) in this case (that is, that \( \beta_{c,1} \) is a continuous function of \( c \)). Weak\(^*\) convergence of measures necessitates that \( \liminf_{c\to c_\star}\beta_{c,1} \geq \beta_{c_\star,1} \). Assume to the contrary that there exists a subsequence \( c_n\to c_\star \) such that \( \beta_{c_\star,1}< \beta_*:=\liminf_{n\to\infty}\beta_{c_n,1} \). Then
\[
\liminf_{n\to\infty} \ell_{c_n,1} = \liminf_{n\to\infty} V^{2\omega_{c_n,1}+\omega_{c_n,2}}(x) \geq V^{2\omega_{c_\star,1}+\omega_{c_\star,2}}(x) > \ell_{c_\star,1}
\]
for \( x\in(\beta_{c_\star,1},\beta_*) \) due to the Principle of Descent \cite[Theorem~I.6.8]{SaffTotik}. However, the above conclusion clearly contradicts the claim \( \ell_{c,1}\to\ell_{c_\star,1} \) as \( c\to c_\star \). The convergence \( \alpha_{c,2}\to \alpha_{c_\star,2} \) as \( c\to c_\star \) can be shown analogously (unfortunately, this convergence of the endpoints was asserted without justification in the proof \cite[Proposition~2.1]{Y16}). Given the convergence of the endpoint, the uniform convergence of the potentials as \( c\to c_\star\in(0,1) \)  was established in the proof of \cite[Proposition~2.1]{Y16} using harmonicity of \( H_c(\z) \). The same arguments can be applied to show that \(
V^{\omega_{c,i}}\to V^{\omega_i} \) uniformly on compact subsets of
\( \C \) as \( c\to 2-i \), \( i\in\{1,2\} \).

Let us now establish the existence of the constants \( 0<c^*<c^{**}<1 \) and the monotonicity properties of \( \beta_{c,1} \) and \( \alpha_{c,2} \). Claim \eqref{divisor-hc} was obtained in \cite[Proposition~2.3]{Y16}. There it was further shown that
\begin{equation}
\label{special-zero}
\beta_{c,1}<\beta_1 \quad \Rightarrow \quad z_c=\beta_{c,1} \qandq \alpha_{c,2}>\alpha_2 \quad \Rightarrow \quad z_c=\alpha_{c,2}.
\end{equation}
Assume now
that \( \beta_{c_1,1}=\beta_{c_2,1}< \beta_1 \). Then the functions \(
h_{c_1}(\z) \) and \( h_{c_2}(\z) \) are defined on the same Riemann
surface. Their difference has at
least four zeros (double zero at \( \infty^{(0)} \) and simple
zeros at \( \infty^{(1)} \) and \( \infty^{(2)} \)) and at most
three poles \( \boldsymbol \alpha_1,\boldsymbol
\alpha_2,\boldsymbol \beta_2 \). This is possible only if the
function is identically zero and therefore \( c_1=c_2 \) as \( h_c^{(1)}(z) = cz^{-1} + \mathcal O(z^{-2}) \) by \eqref{int-rep}. Since
\( \beta_{c,1} \to \alpha_1 \) as \( c\to0 \), this shows the
existence of \( c^* \) and proves monotonicity of \( \beta_{c,1} \)
as a function of \( c \) (it is a continuous and injective function
of \( c \)). The existence of \( c^{**} \) and monotonicity of \(
\alpha_{c,2} \) are proven analogously. It also follows from \eqref{special-zero} that \( c^*\leq c^{**} \). As it was shown in \cite[Proposition~2.3]{Y16} that \(
z_{c^*}=\beta_{c^*,1}(=\beta_1) \) and \(
z_{c^{**}}=\alpha_{c^{**},2}(=\alpha_2) \), we in fact get that \( c^*<c^{**} \). 

It only remains to prove that \( z_c \) is a continuous increasing function of \( c \) on \( [c^*,c^{**}] \).  To show monotonicity, take \( c^*\leq c_1<c_2 \leq c^{**} \). It follows easily from \eqref{int-rep} that each \( h_c(x^{(0)}) \) is a decreasing function of \( x\in(\beta_1,\alpha_2) \).  Thus, to prove that \( z_{c_1}<z_{c_2} \), it is enough to show that \( h(x^{(0)}) > 0 \) in \( (\beta_1,\alpha_2)
\), where \( h(\z):=(h_{c_2}-h_{c_1})(\z) \). Notice that \( h(x^{(0)}) = -
h(x^{(1)})-h(x^{(2)}) \) by \eqref{int-rep} and therefore it is sufficient to argue
that \( h(x^{(1)})< 0 \) on \( (\beta_1,\infty) \) and \(
h(x^{(2)})<0 \) on \( (-\infty,\alpha_2) \). These claims are
obvious for all \( |x| \) large enough since
\[
h(z^{(1)}) = -\frac{c_2-c_1}z + \mathcal O\big(z^{-2}\big) \quad \text{and} \quad h(z^{(2)}) = \frac{c_2-c_1}z + \mathcal O\big(z^{-2}\big)
\]
as \( z\to\infty \) according to \eqref{int-rep}. As explained after \eqref{special-zero},  \( h(\z) \) vanishes only at \( \infty^{(0)} \), \( \infty^{(1)} \), and \( \infty^{(2)} \). Therefore, \( h(z^{(1)}) \) and \( h(z^{(2)}) \) cannot change sign on  \( (\beta_1,\infty) \) and \( (-\infty,\alpha_2) \), respectively. Hence, these functions are negative everywhere on the considered rays by continuity.

To show continuity of \( z_c \) as a function of \( c\in[c^*,c^{**}] \), we shall once again use the fact that \( h_c(x^{(0)}) \) is a decreasing function on \( (\beta_1,\alpha_2) \). When \( c\in(c^*,c^{**}) \), \( h_c(x^{(0)}) \) is unbounded on both ends of \( (\beta_1,\alpha_2) \) and therefore changes sign from \( + \) to \( - \) when passing through \( z_c \) (recall that \( h_c(\z) \) has poles at \( \boldsymbol\beta_1 \) and \( \boldsymbol\alpha_2 \) in this case). When \( c=c^* \), \( h_c(x^{(0)}) \) is unbounded only at \( \alpha_2 \) and, since it is non-vanishing, is negative on \( [\beta_1,\alpha_2) \). Similarly, when \( c=c^{**} \), it is unbounded at \( \beta_1 \) only and therefore is positive on \( (\beta_1,\alpha_2] \).  In any case, \( z_c \) is the point where the potential \( V^{\omega_{c,1}+\omega_{c,2}}(x) \) achieves its minimum on \( [\beta_1,\alpha_2] \). Thus, if \( z_{c_n}\to z_\star \) as \( c_n\to c_\star \) when \( n\to\infty \), \( c_n,c_\star\in(c^*,c^{**}) \), then
\[
V^{\omega_{c_0,1}+\omega_{c_0,2}}(z_\star) \leq \liminf_{n\to\infty} V^{\omega_{c_n,1}+\omega_{c_n,2}}(z_{c_n}) \leq \liminf_{n\to\infty} V^{\omega_{c_n,1}+\omega_{c_n,2}}(z_{c_\star}) = V^{\omega_{c_\star,1}+\omega_{c_\star,2}}(z_{c_\star}),
\]
where the first inequality follows from the weak\(^*\) convergence of measures and the Principle of Descent \cite[Theorem~I.6.8]{SaffTotik}, the second one from the just discussed extremal property of \( z_{c_n} \), and the last equality holds due to the weak\(^*\) convergence of measures and the fact that  \( z_{c_\star} \) does not belong to the supports of the measures in question. Since \( V^{\omega_{c_\star,1}+\omega_{c_\star,2}}(x) \) is smallest at \( z_{c_\star} \), we get that \( z_\star=z_{c_\star} \). When \( c_\star=c^* \), essentially the same argument works. One just needs to replace \( z_{c_\star}=\beta_1 \) with \( \beta_1 + \epsilon \) for any \( \epsilon>0 \). Since \( V^{\omega_{c_\star,1}+\omega_{c_\star,2}}(x) \) is increasing on \( [\beta_1,\alpha_2] \), this shows that \( z_\star\leq z_{c_\star}+\epsilon \) for any \( \epsilon>0 \) and therefore \( z_\star = z_{c_\star} \). Clearly, an analogous modification works when \( c_\star=c^{**} \).
\end{proof}

\section{Proof of Propositions~\ref{prop:angelesco} and~\ref{prop:szego}}
\label{sec:5}

On several occasions, we shall refer to the following consequences of Koebe's \(1/4\)-theorem, \cite[Theorem~1.3]{Pommerenke}. Given \( r> 0 \), let
\[
a(z) = \sum_{k=0}^\infty a_k(z-z_0)^k, \quad  b(z)=\sum_{k=0}^\infty b_kz^{-k}, \qandq d(z) = \sum_{k=-\infty}^1 d_kz^k 
\]
be univalent in \( D_a=\{|z-z_0|<r \} \),  \( D_b=\{|z|>1/r\} \), and \( D_d=\{|z|>r \} \), respectively. Then,
\begin{equation}
\label{koebe1/4}
\big\{ |z-a_0| < ra_1/4 \big\} \subseteq a(D_a), \quad \big\{|z-b_0|< rb_1/4 \big\} \subseteq b(D_b), \qandq \big\{|z|>4rd_1\} \subseteq d(D_d),
\end{equation}
where \( f(D) \) stands for image of a domain \( D \) under the function \( f(z) \).

\subsection{Proof of Proposition~\ref{prop:angelesco}}

Recall that \( \chi_c(\z) \) is univalent on \( \RS_c \) and \( \chi_c^{(0)}(z) = z + \mathcal O(z^{-1}) \) as \( z\to\infty \), see \eqref{chi}. Hence, it follows from \eqref{koebe1/4} that there exists a finite constant \( R \) independent of \( c \) such that  \( \{|z|>R\}\subset \chi_c\big(\RS_c^{(0)}\big) \) for all \( c\in(0,1) \). In particular, it holds that \( |\chi_c(\x)|\leq R \), \( \x\in\boldsymbol\Delta_{c,1} \), as well as \( |B_{c,i}|\leq R \), \( i\in\{1,2\} \), see \eqref{AngPar1}, for all \( c\in(0,1) \). For all \( c\leq c^{**} \) (in which case \( \Delta_{c,2}= \Delta_2\)), define
\[
\varphi(\z) := \frac12 \left\{
\begin{array}{ll}
z - (\beta_2+\alpha_2)/2+w_2(z), & z\in\RS_c^{(0)}\setminus \boldsymbol\Delta_{c,1}, \medskip \\
z - (\beta_2+\alpha_2)/2-w_2(z), & z\in\RS_c^{(2)}.
\end{array}
\right.
\]
This is a meromorphic function in \( \big(\RS_c^{(0)}\cup\RS_c^{(2)}\big)\setminus\boldsymbol\Delta_{c,1} \) with a simple pole at \( \infty^{(0)} \), a simple zero at \(
\infty^{(2)} \), and otherwise non-vanishing and finite. It is
normalized so that  \( \varphi(z^{(0)}) =z+\mathcal O(1) \) as \(
z\to\infty \). Observe that \( \varphi(\z) \) continuously extends
to the closed set \( \RS_c^{(0)}\cup\RS_c^{(2)} \).  It can be
readily checked that the image of \( \RS_c^{(0)}\cup\RS_c^{(2)} \)
under \( \varphi(\z) \) is equal to \( \overline \C \) and \(
\varphi(\z)\) is one-to-one everywhere except on \(
\boldsymbol\Delta_{c,1} \) that is mapped into an interval
\[
\varphi(\boldsymbol \Delta_{c,1}) =: I_{c,1} = \big[\varphi(\boldsymbol \alpha_1),\varphi(\boldsymbol \beta_{c,1})\big] \to \big\{\varphi(\boldsymbol \alpha_1)\big\} \quad \text{as} \quad c\to 0.
\]
Notice also that \( \varphi^{(0)}(z) = \varphi_2(z) \) for \( z\in\overline\C\setminus\Delta_2 \), see \eqref{varphii}.

Define \( f_c(z) :=
\big(\chi_c\big(\varphi^{-1}(z)\big)-B_{c,2}\big)/z \). Then \( f_c(z)
\) is a holomorphic function in \( \overline\C\setminus I_{c,1} \)
(there is no pole at the origin as \( \varphi^{-1}(0)=\infty^{(2)}
\) and \( \chi_c(\z) - B_{c,2} \) vanishes there) with bounded
traces on \( I_{c,1} \) that assumes value \(1\) at infinity.
Hence, it follows from Cauchy's integral formula that
\[
f_c(z) = 1 + \int_{I_{c,1}}\frac{(f_{c+}-f_{c-})(x)}{x-z}\frac{\dd x}{2\pi\ic}, \quad z\in \overline\C\setminus I_{c,1}.
\]
Since the traces \( f_{c\pm}(z) \) are bounded above in absolute
value on \( I_{c,1} \) independently of \( c \) and \( |I_{c,1}|\to
0 \) as \( c\to0 \), we see that \( f_c(z)\to1 \) as \( c\to0 \) locally uniformly in \(
\overline\C\setminus\{\varphi(\boldsymbol \alpha_1)\} \). Hence, it holds that
\[
\chi_c(\z) = B_{c,2} + \big(1+o(1) \big) \varphi(\z)
\]
locally uniformly on \( \big(\RS_c^{(0)}\cup\RS_c^{(2)}\big)\setminus\boldsymbol\Delta_{c,1} \). Since the image of \( \big(\RS_c^{(0)}\cup\RS_c^{(2)}\big)\setminus\boldsymbol\Delta_{c,1} \) under \( \varphi(\z) \) is \( \overline\C\setminus I_{c,1} \) and \( |I_{c,1}| \to0 \) as \( c\to0 \), for any \( \epsilon>0 \) there exists \( \delta>0 \) such that the image of \( \big(\RS_c^{(0)}\setminus\pi^{-1}(\{|z-\alpha_1|<\epsilon\})\big)\cup\RS_c^{(2)} \)  under \( \chi_c(\z) \) contains \( \overline\C\setminus\{|z-B_{c,2}-\varphi(\boldsymbol\alpha_1)|<\delta\} \). Due to univalency of \( \chi_c(\z) \) on \( \RS_c \), this means that the image of  \( \big(\RS_c^{(0)}\cap\pi^{-1}(\{|z-\alpha_1|<\epsilon\})\big)\cup\RS_c^{(1)} \) is contained in \( \{|z-B_{c,2}-\varphi(\boldsymbol\alpha_1)|<\delta\} \). Altogether, we get that
\begin{equation}
\label{prop-rec0}
\chi_c(\z) = B_{c,2} + \big(1+o(1) \big) \left\{
\begin{array}{rl}
\varphi (\z), & \z\in\RS_c^{(0)}\cup\RS_c^{(2)}, \medskip \\
\varphi(\boldsymbol \alpha_1), & \z\in\RS_c^{(1)},
\end{array}
\right.
\end{equation}
where \( o(1) \) holds uniformly on the entire surface \( \RS_c \). Since
\[
\varphi^{(0)}(z) = z - \frac{\beta_2+\alpha_2}2 + \mathcal O\left(\frac1z\right) \qandq \varphi^{(2)}(z) = \frac{(\beta_2-\alpha_2)^2}{16}\frac1z +  \mathcal O\left(\frac1{z^2}\right),
\]
the desired limits \eqref{AngPar2} easily follow.

Continuity of  \( A_{c,1},A_{c,2},B_{c,1},B_{c,2} \) as functions
of \( c \) comes from the continuous dependence of \( \alpha_{c,2}
\) and \( \beta_{c,1} \) on \( c \), see Proposition~\ref{prop:2},
and therefore the continuous dependence \( \chi_c(\z) \) on \( c
\).

\subsection{Auxiliary Estimates, I}

In the forthcoming analysis, the following functions will play an
important role:
\begin{equation}
\label{Upsilon}
\Upsilon_{c,i}(\z) := A_{c,i}\big(\chi_c(\z)-B_{c,i}\big)^{-1}, \quad i\in\{1,2\}.
\end{equation}
It follows from the properties of \( \chi_c(\z) \), see \eqref{chi} and \eqref{AngPar1}, that \( \Upsilon_{c,i}(\z) \) is a conformal map of \( \RS_c \) onto \( \overline \C \) that maps \( \infty^{(i)} \) into \( \infty \) and \( \infty^{(0)} \) into \( 0 \). Moreover, it holds that
\begin{equation}
\label{Upsilon2}
\Upsilon_{c,1}^{(1)}(z) = z + \mathcal O(1) \qandq \Upsilon_{c,1}^{(0)}(z) = A_{c,1}z^{-1} + \mathcal O\big(z^{-2}\big) \qasq z\to\infty.
\end{equation}
 It was explained in \cite[Section~7]{Y16}, see
\cite[Equation~(7.2)]{Y16}, that
\begin{equation}
\label{lem51-0}
\Upsilon_{c,i}(\z) \to \Upsilon_{c_\star,i}(\z) \qasq c\to c_\star\in(0,1),
\end{equation}
uniformly on \( \RS_{c_\star}\setminus\mathfrak U \) for each \( i\in\{1,2\} \), where \( \mathfrak U \) is any open set the containing ramification points of \( \RS_{c_\star} \) (if \( \mathfrak U_c\subset\RS_c \) is an open set such that \( \pi\big( \RS_{c_\star}^{(k)}\setminus\mathfrak U\big) = \pi\big( \RS_c^{(k)}\setminus\mathfrak U_c\big) \) for each \( k\in\{0,1,2\} \), then the bordered Riemann surfaces \( \RS_{c_\star}\setminus\mathfrak U \) and \( \RS_c\setminus\mathfrak U_c \) are identical for all \( c \) sufficiently close to \( c_\star \) and we can think of \(\Upsilon_{c,i}(\z)\) as a function on \(\RS_{c_\star}\setminus\mathfrak U\)). On the other hand, when \( c\to 0 \), the following is true.

\begin{lemma}
\label{lem:aux1} It holds that
\begin{equation}
\label{lem51-1}
\Upsilon_{c,2}(\z) =
(1+o(1)) \left\{
\begin{array}{rl}
\psi(\z), & z\in\RS_c^{(0)}\cup\RS_c^{(2)}, \medskip \\
\psi(\boldsymbol\alpha_1), & \z\in\RS^{(1)}_c,
\end{array}
\right.
\end{equation}
as \( c\to0 \), where \( o(1) \) holds uniformly on the entire surface \( \RS_c \) and
\[
 \psi(\z) := \frac{A_{0,2}}{\varphi(\z)} = \frac12 \left\{
\begin{array}{ll}
z - (\beta_2+\alpha_2)/2-w_2(z), & z\in\RS_c^{(0)}\setminus \boldsymbol\Delta_{c,1}, \medskip \\
z - (\beta_2+\alpha_2)/2+w_2(z), & z\in\RS_c^{(2)},
\end{array}
\right.
\]
that is \( \psi^{(2)}(z) \) maps \( \RS_c^{(2)} \) conformally onto \( \{|z|>(\beta_2-\alpha_2)/4\} \) and \(  \psi^{(0)}(z) \psi^{(2)}(z)\equiv A_{0,2}\). Moreover, it holds that\footnote{Given non-negative functions \( A_c(z) \) and \( B_c(z) \), we write \( A_c(z)\lesssim B_c(z) \) (resp. \( A_c(z)\sim B_c(z) \)) as \( c\to 0 \) on \( K_c \) for some family of closed sets \( \{K_c\} \), if there exists \( \epsilon> 0 \) such that \( A_c(z) \leq C B_c(z) \) (resp. \( C^{-1}A_c(z)\leq B_c(z) \leq CA_c(z) \)) for all \( z\in K_c \) and each \( c\in[0,\epsilon] \), where \( C \) depends only on \( \epsilon \).}
\begin{equation}
\label{lem51-2}
\big|\Upsilon_{c,1}^{(0)}(z)\big| \sim c|\phi_c^{-1}(z)|, \quad \big|\Upsilon_{c,1}^{(1)}(z)\big| \sim c|\phi_c(z)|, \qandq \big|\Upsilon_{c,1}^{(2)}(z)\big| \sim c^2
\end{equation}
on \( \overline \C \) (including the traces on \(
\Delta_{c,1}\cup \Delta_2\), \( \Delta_{c,1} \), and \( \Delta_2
\), respectively) as \( c\to 0 \), where
\begin{equation}
\label{lem51-3}
\phi_c(z) := \frac2{\beta_{c,1}-\alpha_1}\left(z - \frac{\beta_{c,1}+\alpha_1}2+w_{c,1}(z)\right)
\end{equation}
is the conformal map of \( \overline\C\setminus \Delta_{c,1} \)
onto \( \{|z|>1\} \) that fixes the point at infinity and has
positive derivative there. In addition, it holds that \(
\Upsilon_{c,1}^{(1)}(z) = z-\alpha_1 + \mathcal O(c) \) uniformly
in \( \overline\C \) as \( c\to 0 \).
\end{lemma}
\begin{proof}
Formula \eqref{lem51-1} follows immediately from \eqref{prop-rec0}, the very
definition \eqref{Upsilon},  and the first limit
in \eqref{AngPar2}. It also is immediate from  \eqref{Upsilon} and
\eqref{prop-rec0} that
\begin{equation}
\label{lem51-4}
\big|\Upsilon_{c,1}^{(2)}(z)\big| = \left|\frac{A_{c,1}}{(1+o(1))\varphi(\boldsymbol\alpha_1) + (1+o(1))\varphi^{(2)}(z)}\right| \sim A_{c,1}
\end{equation}
in \( \overline\C \) (including the traces on \( \Delta_2 \)) as \( c\to 0 \) since \( |\varphi^{(2)}(z)| \leq (\beta_2-\alpha_2)/4<|\varphi(\boldsymbol\alpha_1)| \), see \eqref{varphii}. It can be readily verified that the symmetric functions of the branches of a rational function on \( \RS_c \) must be rational functions on \( \overline\C \). Since \( \Upsilon_{c,1}^{(1)}(z) \) has a simple pole at infinity, \( \Upsilon_{c,1}^{(0)}(z) \) has a simple zero there, and \( \Upsilon_{c,1}^{(k)}(z) \), \( k\in\{0,1,2\} \), are otherwise non-vanishing and  finite, the product of three branches of \( \Upsilon_{c,1}(\z) \) must be a constant. Thus, similarly to \eqref{lem51-4}, it holds that
\begin{equation}
\label{lem51-5}
\big|\Upsilon_{c,1}^{(0)}(z)\Upsilon_{c,1}^{(1)}(z)\Upsilon_{c,1}^{(2)}(z)\big| = \frac{A_{c,1}^2}{B_{c,2}-B_{c,1}} = -\frac{A_{c,1}^2}{(1+o(1))\varphi(\boldsymbol\alpha_1)} \sim A_{c,1}^2
\end{equation}
in \( \overline\C \) as \( c\to0 \) (recall that \( \varphi(\boldsymbol\alpha_1)<0 \)). For each \( \z\notin\boldsymbol\Delta_{c,1}\cup\boldsymbol\Delta_{c,2} \), let \( \overline\z \) be the point on the same sheet of \( \RS_c \) as \( \z \) with \( \pi(\overline\z) = \overline z \) and then extend this definition by continuity to \( \boldsymbol\Delta_{c,1}\cup\boldsymbol\Delta_{c,2} \). The function \( \overline{\Upsilon_{c,1}(\overline\z)} \) is meromorphic on \( \RS_c \) and has the same zero/pole divisor and normalization as \( \Upsilon_{c,1}(\z) \). Therefore, \( \overline{\Upsilon_{c,1}(\overline\z)} = \Upsilon_{c,1}(\z) \). In particular, \( \Upsilon_{c,1}^{(2)}(x) \) is real on \( \Delta_{c,1} \) and the traces of \( \Upsilon_{c,1}^{(k)}(z) \) on \( \Delta_{c,1} \), \( k\in\{0,1\} \), are conjugate-symmetric. Hence, we get from \eqref{lem51-4} and \eqref{lem51-5} that
\begin{equation}
\label{lem51-6}
A_{c,1} \sim A_{c,1}^{-1}\big|\Upsilon_{c,1}^{(2)}(x)\Upsilon_{c,1\pm}^{(1)}(x)\Upsilon_{c,1\pm}^{(0)}(x)\big| \sim \big|\Upsilon_{c,1\pm}^{(1)}(x)\big|^2 =  \big|\Upsilon_{c,1\pm}^{(0)}(x)\big|^2, \quad x\in\Delta_{c,1},
\end{equation}
as \( c\to 0 \). Thus, \eqref{lem51-4}, \eqref{lem51-6}, and the maximum modulus principle applied to \( \Upsilon_{c,1}^{(0)}(z)\phi_c(z) \) and \( \Upsilon_{c,1}^{(1)}(z)/\phi_c(z) \) yield \eqref{lem51-2} with \( c^2 \) replaced by \( A_{c,1} \). That is, we need to show that \( A_{c,1} \sim c^2 \) as  \( c\to 0 \).

As is mentioned above, the sum \( \Upsilon_{c,1}^{(0)}(z) + \Upsilon_{c,1}^{(1)}(z) + \Upsilon_{c,1}^{(2)}(z) \) is a rational function on \( \overline\C \). Since it has only one pole, which is simple and located at infinity, it is a monic (see \eqref{Upsilon2}) polynomial of degree \( 1 \). In particular, it holds that
\begin{equation}
\label{lem51-7}
\beta_{c,1} - \alpha_1 = 2\Upsilon_{c,1}^{(0)}(\beta_{c,1}) +  \Upsilon_{c,1}^{(2)}(\beta_{c,1}) -  2\Upsilon_{c,1}^{(0)}(\alpha_1) -  \Upsilon_{c,1}^{(2)}(\alpha_1),
\end{equation}
where we used the fact that \( \Upsilon_{c,1}^{(0)}(\gamma)  = \Upsilon_{c,1}^{(1)}(\gamma)=\Upsilon_{c,1}(\boldsymbol\gamma) \) for \( \boldsymbol\gamma\in\{\boldsymbol\alpha_1,\boldsymbol\beta_{c,1}\} \). Thus, it follows from \eqref{bc1} and \eqref{lem51-7} (lower bound) together with \eqref{lem51-4} and \eqref{lem51-6} (upper bound) that
\[
c \lesssim  2\big|\Upsilon_{c,1}^{(0)}(\beta_{c,1})\big| + \big|\Upsilon_{c,1}^{(2)}(\beta_{c,1})\big| + 2\big|\Upsilon_{c,1}^{(0)}(\alpha_1)\big| +  \big|\Upsilon_{c,1}^{(2)}(\alpha_1)\big| \lesssim A_{c,1}^{1/2} + A_{c,1} \lesssim A_{c,1}^{1/2}
\]
as \( c\to 0 \), where we also used the fact that \( A_{c,1}\to0 \) as \( c\to 0 \) for the last inequality. On the other hand, it holds that
\[
\Upsilon_{c,2}^{(1)}(z) = -\frac{A_{c,2}}{B_{c,2}-B_{c,1}} + \frac{A_{c,1}A_{c,2}}{B_{c,2}-B_{c,1}}\frac1z + \mathcal O\left(\frac1{z^2}\right)
\]
as \( z\to\infty \) by the very definitions \eqref{Upsilon} and \eqref{chi}. Therefore, we can deduce from Cauchy's integral formula that
\begin{equation}
\label{lem51-8}
\frac{A_{c,1}A_{c,2}}{B_{c,2}-B_{c,1}} = \left|\frac1{2\pi\ic}\int_{\Delta_{c,1}}\left(\Upsilon_{c,2+}^{(1)}(x)-\Upsilon_{c,2-}^{(1)}(x)\right)\dd x\right| \leq \frac{\beta_{c,1}-\alpha_1}\pi \max_{\x\in\boldsymbol\Delta_{c,1}}\left|\Upsilon_{c,2}^{(1)}(\x) + Z\right|
\end{equation}
for any complex number \( Z \). Now, if we show that
\begin{equation}
\label{lem51-9}
\max_{\x\in\boldsymbol\Delta_{c,1}}\left|\Upsilon_{c,2}^{(1)}(\x) + \frac{A_{c,2}}{B_{c,2}-B_{c,1}} \right| \lesssim A_{c,1}^{1/2}
\end{equation}
as \( c\to 0 \), inequalities \eqref{bc1} and \eqref{lem51-8} together with limits \eqref{AngPar2} will allow us to conclude that \(A_{c,1}^{1/2} \lesssim c \) as \( c\to 0 \), which will finish the proof of \eqref{lem51-2}. To prove \eqref{lem51-9}, observe that
\[
\Upsilon_{c,2}(\z) = \frac{A_{c,2}}{\chi_c(\z)-B_{c,2}} = \frac{A_{c,2}}{B_{c,1}-B_{c,2} + A_{c,1}\Upsilon_{c,1}^{-1}(\z)} = \frac{A_{c,2}}{B_{c,2}-B_{c,1}}\frac{\Upsilon_{c,1}(\z)}{\frac{A_{c,1}}{B_{c,2}-B_{c,1}} - \Upsilon_{c,1}(\z)}
\]
according to their very definition \eqref{Upsilon}. Thus,
\[
\Upsilon_{c,2}(\z) + \frac{A_{c,2}}{B_{c,2}-B_{c,1}} = \frac{A_{c,2}}{B_{c,2}-B_{c,1}}\frac{A_{c,1}}{A_{c,1}-(B_{c,2}-B_{c,1})\Upsilon_{c,1}(\z)}
\]
The desired estimate \eqref{lem51-9} now follows from \eqref{lem51-6} and \eqref{AngPar2}.

To prove the last claim of the lemma, observe that \( \Upsilon_{c,1}^{(1)}(z) - (z-\alpha_1) \) is holomorphic in \( \overline\C\setminus\Delta_{c,1} \) and
\[
\big|\Upsilon_{c,1\pm}^{(1)}(x) - (x-\alpha_1)\big| \leq \max_{x\in\Delta_{c,1}}\big|\Upsilon_{c,1\pm}^{(1)}(x)\big| + \beta_{c,1}-\alpha_1 \lesssim c, \quad
x\in\Delta_{c,1},
\]
as \( c\to 0 \) by \eqref{bc1} and \eqref{lem51-6}. The desired claim now follows from the maximum modulus principle.
\end{proof}

In our analysis, it will be convenient to apply
Lemma~\ref{lem:aux1} in the following form.

\begin{lemma}
\label{lem:aux1a} 
For each \( 0<\delta\leq (\alpha_2-\beta_1)/2 \) fixed, it holds that
\begin{equation}
\label{lem52-1}
\left\{
\begin{array}{l}
c^{-1}\big|\Upsilon_{c,1}^{(0)}(z)\big|, \, c^{-1}\big|\Upsilon_{c,1}^{(1)}(z)\big|, \, c^{-2}\big|\Upsilon_{c,1}^{(2)}(z)\big| \sim 1, \bigskip \\
(1-c)^{-2}\big|\Upsilon_{c,2}^{(0)}(z)\big|, \, (1-c)^{-2}\big|\Upsilon_{c,2}^{(1)}(z)\big|, \, \big|\Upsilon_{c,2}^{(2)}(z)\big| \sim 1,
\end{array}
\right.
\end{equation}
on \( K_{c,\delta,1} := \{z:\dist(z,\Delta_{c,1})\leq c\delta\} \) for all \( c\in(0,1) \) and that
\begin{equation}
\label{lem52-2}
\left\{
\begin{array}{l}
c^{-2}\big|\Upsilon_{c,1}^{(0)}(z)\big|, \, \big|\Upsilon_{c,1}^{(1)}(z)\big|, \, c^{-2}\big|\Upsilon_{c,1}^{(2)}(z)\big| \sim 1, \bigskip \\
(1-c)^{-1}\big|\Upsilon_{c,2}^{(0)}(z)\big|, \, (1-c)^{-2}\big|\Upsilon_{c,2}^{(1)}(z)\big|, \, (1-c)^{-1}\big|\Upsilon_{c,2}^{(2)}(z)\big| \sim 1,
\end{array}
\right.
\end{equation}
on \( K_{c,\delta,2} := \{z:\dist(z,\Delta_{c,2})\leq (1-c)\delta\} \) for all \( c\in(0,1) \), where the constants of proportionality depend only on \( \delta \).
\end{lemma}
\begin{proof}
We provide the proofs only for \(  \Upsilon_{c,1}(\z) \), understanding that the arguments for \(  \Upsilon_{c,2}(\z) \) are essentially identical. Recall that \( \Upsilon_{c,1}(\z) \) is a conformal map of \( \RS_c \) onto \( \overline\C \) that maps \( \infty^{(0)} \) into \( 0 \) and \( \infty^{(1)} \) into \( \infty \). Let \( r:=\max\{|\alpha_1|,|\beta_2|\} \). Then it follows from \eqref{koebe1/4} and \eqref{Upsilon2} that
\[
\big\{|z|<A_{c,1}/4(r+\delta) \big\} \subset \Upsilon_{c,1}^{(0)}\big(\{|z|>r+\delta\}\big) \qandq \big\{|z|>4(r+\delta)\big\} \subset \Upsilon_{c,1}^{(1)}\big(\{|z|>r+\delta\}\big).
\]
Thus, it holds that
\[
\frac{A_{c,1}}{4(r+\delta)} \leq |\Upsilon_{c,1}(\z)| \leq 4(r+\delta) \quad \text{for all} \quad z\in K_{c,\delta,1}\cup K_{c,\delta,2}.
\]
Since \( A_{c,1}\to ((\beta_1-\alpha_1)/4)^2 \) by the limit analogous to the one for \( A_{c,2} \) in \eqref{AngPar2}, this establishes the desired bounds in \eqref{lem52-1} and \eqref{lem52-2} for all \( c\in[\epsilon,1) \) and any \( \epsilon>0 \) fixed with the constants of proportionality dependent on \( \epsilon \) and \( \delta \). On the other hand, the bounds for \( c\in(0,\epsilon] \) readily follow from \eqref{lem51-2} and \eqref{lem51-3} as
\begin{equation}
\label{lem52-3}
1\leq |\phi_c(z)| \leq 4\frac{c\delta+\beta_{c,1}-\alpha_1}{\beta_{c,1}-\alpha_1} < 4 + \frac\delta{\beta_2-\alpha_1}  \qandq c|\phi_c(z)| \sim |z-\alpha_1|
\end{equation}
on \( K_{c,\delta,1} \) and \( K_{c,\delta,2} \), respectively, as \( c\to 0 \) by elementary estimates and \eqref{bc1}. The estimates of \( \Upsilon_{c,2}^{(k)}(z) \) can be verified similarly.
\end{proof}

Let a function \( \Pi_c(\z) \) be defined on \( \RS_c \)
analogously to the way \( \Pi_\n(\z) \) was defined on \( \RS_\n \)
just before Theorem~\ref{thm:asymp3}. Further, let \( \Pi_{c,i}(\z)
\), \( i\in\{1,2\} \), be rational functions on \( \RS_c \) with
the divisors and normalization given by
\begin{equation}
\label{Pies}
(\Pi_{c,i}) = \infty^{(0)} + \infty^{(i)} + 2\infty^{(3-i)} - \mathcal D_c \qandq \Pi_{c,i}^{(i)}(z) = \frac 1z + \mathcal O\left(\frac1{z^2}\right),
\end{equation}
where \( \mathcal D_c\) is the divisor of the ramification points
of \( \RS_c \), see Proposition~\ref{prop:2}.

\begin{lemma}
\label{lem:aux2} It holds that
\begin{equation}
\label{Up-Pi1}
(-1)^{3-i}(w_{c,1}w_{c,2})(z)\Pi_{c,3-i}(\z) = \left\{
\begin{array}{ll}
\left(\Upsilon_{c,i}^{(2)}-\Upsilon_{c,i}^{(1)}\right)(z), & \z\in\RS_c^{(0)}, \medskip \\
\left(\Upsilon_{c,i}^{(0)}-\Upsilon_{c,i}^{(2)}\right)(z), & \z\in\RS_c^{(1)}, \medskip \\
\left(\Upsilon_{c,i}^{(1)}-\Upsilon_{c,i}^{(0)}\right)(z), & \z\in\RS_c^{(2)},
\end{array}
\right.
\end{equation}
for \( i\in\{1,2\} \) and
\begin{equation}
\label{Up-Pi2}
(w_{c,1}w_{c,2})(z)\Pi_c(\z) = \left\{
\begin{array}{ll}
\left(\Upsilon_{c,2}^{(2)}\Upsilon_{c,1}^{(1)}-\Upsilon_{c,2}^{(1)}\Upsilon_{c,1}^{(2)}\right)(z), & \z\in\RS_c^{(0)}, \medskip \\
\left(\Upsilon_{c,2}^{(0)}\Upsilon_{c,1}^{(2)}-\Upsilon_{c,2}^{(2)}\Upsilon_{c,1}^{(0)}\right)(z), & \z\in\RS_c^{(1)}, \medskip \\
\left(\Upsilon_{c,2}^{(1)}\Upsilon_{c,1}^{(0)}-\Upsilon_{c,2}^{(0)}\Upsilon_{c,1}^{(1)}\right)(z), & \z\in\RS_c^{(2)}.
\end{array}
\right.
\end{equation}
Moreover, it holds that
\begin{equation}
\label{lem52-0}
\Pi_c^{(0)}(z) = (1+o(1)) \frac{\psi^{(2)}(z)}{w_2(z)}\frac{z - \alpha_1 + \mathcal O(c)}{w_{c,1}(z)} = (1+o(1)) \frac{\psi^{(2)}(z)}{w_2(z)}
\end{equation}
as \( c\to0 \), where the first relation holds uniformly in \(
\overline \C \) (that is, including the traces on \( \Delta_{c,1}\cup\Delta_2 \))
and the second one locally uniformly in \(
\overline\C\setminus\Delta_{0,1} \).
\end{lemma}
\begin{proof}
Representations \eqref{Up-Pi1} and \eqref{Up-Pi2} can be easily
verified by observing that the right-hand sides are continuous
across \( \boldsymbol\Delta_{c,1} \) and \( \boldsymbol\Delta_{c,2}
\) and by comparing the zero/pole divisors and the normalizations
of the left-hand and right-hand sides, see \eqref{AngPar1},
\eqref{Upsilon}, and \eqref{Pies}. Asymptotic formula
\eqref{lem52-0} follows immediately from the first relation in
\eqref{Up-Pi2}, asymptotic formulae \eqref{lem51-1} and \eqref{lem51-2}, and the last claim of
Lemma~\ref{lem:aux1}.
\end{proof}

\subsection{Proof of Proposition~\ref{prop:szego}}

It was shown  in \cite[Section~6]{Y16} that the Szeg\H{o} functions
\( S_c(\z) \) satisfying \eqref{szego-pts2} is given by
\[
S_c(\z) := \exp\left\{\frac1{6\pi\ic}\sum_{i=1}^2\int_{\boldsymbol\Delta_{c,i}}\log(\rho_i w_{c,i+})(s)\mathcal C_\z(\boldsymbol s)\right\},
\]
where \( \mathcal C_\z(\boldsymbol s) \) is the third kind
differential on \( \RS_c \) with three simple poles at \(
\z,\z_1,\z_2 \) that have the same natural projection \( z \) and
respective residues \( -2,1,1 \). Limit \eqref{szego-limit1} was in
fact proven in \cite[Section~7]{Y16}. Thus, it only remains to show
the validity of \eqref{szego-limit2} and \eqref{szego-limit3}. In
order to do that we shall use  an alternative construction of \(
S_c(\z) \) that is more amenable to asymptotic analysis.

Since we are interested in what happens when \( c\to0 \), we shall assume that \( c\leq\min\{1/2,c^{**} \} \) (the choice of \( 1/2 \) is rather arbitrary, but convenient to use in \eqref{bc1}). Set
\[
D_{c,1}(z) := \left(\frac{z-(\beta_{c,1}+\alpha_{c,1})/2+w_{c,1}(z)}{2w_{c,1}(z)}\right)^{1/2}, \quad z\in \overline\C\setminus \Delta_{c,1},
\]
where we take the branch of the square root such that \( D_{c,1}(z) \) is holomorphic and non-vanishing in the domain of the definition and has value 1 at infinity. The traces of  \( D_{c,1}(z) \) on \( \Delta_{c,1} \) satisfy
\begin{equation}
\label{prop-szego13}
|D_{c,1\pm}(x)|^2 =  \big(D_{c,1+}D_{c,1-}\big)(x) = \frac{\beta_{c,1}-\alpha_1}{4|w_{c,1}(x)|}=\frac{\ic}4\frac{\beta_{c,1}-\alpha_1}{w_{c,1+}(x)}, \quad x\in \Delta_{c,1}.
\end{equation}
Let \( \delta>0 \) be as in Lemma~\ref{lem:aux1a}, that is, \( \delta\leq(\alpha_2-\beta_1)/2 \). Then it follows from \eqref{bc1} that \( \delta c\leq |\Delta_{c,1}|/8 \). Using \eqref{bc1} once more together with our assumption that \( c\leq1/2 \), we get that
\begin{equation}
\label{prop-szego12}
\left\{\begin{array}{ll}
\sqrt{3(\alpha_2-\beta_1)} < |w_{c,1}(s)| / (c\sqrt\delta ) < 3\sqrt{\beta_2-\alpha_1}, & |s-\alpha_1|=\delta c, \; |s-\beta_{c,1}|=\delta c, \medskip \\
\sqrt{\delta(\alpha_2-\beta_1)} < |w_{c,1\pm}(x)|/c < 8(\beta_2-\alpha_1), & \alpha_1+\delta c\leq x\leq \beta_{c,1} - \delta c,
\end{array} \right.
\end{equation}
(the constants in the above inequalities are in no way sharp, but sufficient for our purposes). Therefore, \eqref{prop-szego12} and similar straightforward estimates of \( |2z-\alpha_1-\beta_{c,1}|\) using \eqref{bc1} as well as \eqref{prop-szego13} and the maximum modulus principle for holomorphic functions applied to both \( D_{c,1}(z) \) and \( D_{c,1}^{-1}(z) \) yield that
\begin{equation}
\label{prop-szego1b}
\left\{
\begin{array}{ll}
|D_{c,1}(s)| \sim \delta^{-1/4}, & |s-\alpha_1|=\delta c, \, |s-\beta_{c,1}|=\delta c, \medskip \\
1 \lesssim |D_{c,1}(z)| \lesssim \delta^{-1/4},  & 0<\delta c \leq \dist(z,\{\alpha_1,\beta_{c,1}\}),
\end{array}
\right.
\end{equation}
uniformly on the respective sets, where the constants of proportionality do not depend on \( c,\delta \). Additionally, since \( \beta_{c,1}\to\alpha_1 \) as \( c\to 0 \) and therefore \( w_{c,1}(z)=z-\alpha_1 + o(1) \) locally uniformly in \( \C\setminus\Delta_{0,1} \) as \( c\to 0 \), it holds locally uniformly in \( \overline\C\setminus\Delta_{0,1} \) that
\begin{equation}
\label{prop-szego1a}
D_{c,1}(z) = 1+o(1) \qasq c\to0.
\end{equation}

Now, let \( D_{c,\rho_1}(z) \) be the Szeg\H{o} function of the restriction of \( \rho_1(x) \) to \( \Delta_{c,1} \) normalized to have value 1 at infinity.  That is,
\begin{equation}
\label{prop-szego8}
D_{c,\rho_1}(z) = \exp\left\{\frac{w_{c,1}(z)}{2\pi\ic}\int_{\Delta_{c,1}}\frac{\log\rho_1(x)}{z-x}\frac{\dd x}{w_{c,1+}(x)} - \int_{\Delta_{c,1}}\frac{\log\rho_1(x)}{w_{c,1+}(x)}\frac{\dd x}{2\pi\ic}\right\},
\end{equation}
\( z\in\overline\C\setminus\Delta_{c,1} \), where we set \( \log\rho_1(x) := \log\mu_1^\prime(x)+\log(2\pi)-\pi\ic/2 \), see \eqref{rhoi} and recall that \( \mu_1^\prime(x) \) is positive on \( \Delta_1 \).  Observe that
\begin{equation}
\label{prop-szego4}
-\int_{\Delta_{c,1}}\frac1{w_{c,1+}(x)}\frac{\dd x}{\pi\ic} = 1 \qandq \frac1{\pi\ic}\int_{\Delta_{c,1}}\frac 1{z-x}\frac{\dd x}{w_{c,1+}(x)} = - \frac 1{w_{c,1}(z)},
\end{equation}
by Cauchy's theorem and integral formula. Hence, \( D_{c,\rho_1}(z) = D_{c,\mu_1^\prime}(z)\) is a holomorphic and non-vanishing function in \( \overline\C \setminus \Delta_{c,1} \) with continuous and conjugate-symmetric traces on \( \Delta_{c,1} \) that satisfy
\begin{equation}
\label{prop-szego2}
\rho_1(x)|D_{c,\rho_1\pm}(x)|^2 = \big(\rho_1D_{c,\rho_1+}D_{c,\rho_1-}\big)(x) = G_{c,\rho_1} := \exp\left\{- \int_{\Delta_{c,i}}\frac{\log\rho_1(x)}{w_{c,1+}(x)}\frac{\dd x}{\pi\ic}\right\},
\end{equation}
according to Plemelj-Sokhotski formulae. Now, analyticity of \( \rho_1(x) \) in a neighborhood of \( \Delta_1 \) implies that \( \max_{x\in\Delta_{c,1}}|\rho_1(x)/\rho_1(\alpha_1) - 1| \to 0 \) as \( c\to 0 \).  Combining this estimate with \eqref{prop-szego4} yields that
\[
- \int_{\Delta_{c,1}}\frac{\log\rho_1(x)}{w_{c,1+}(x)}\frac{\dd x}{\pi\ic} = \log\rho_1(\alpha_1) - \int_{\Delta_{c,1}}\frac{\log(\rho_1(x)/\rho_1(\alpha_1))}{w_{c,1+}(x)}\frac{\dd x}{\pi\ic} = \log\rho_1(\alpha_1) + o(1)
\]
when \( c\to0 \) as well as that
\begin{eqnarray*}
\frac{w_{c,1}(z)}{\pi\ic}\int_{\Delta_{c,1}}\frac{\log\rho_1(x)}{z-x}\frac{\dd x}{w_{c,1+}(x)} &=& \frac{w_{c,1}(z)}{\pi\ic}\int_{\Delta_{c,1}}\frac{\log(\rho_1(x)/\rho_1(\alpha_1))}{z-x}\frac{\dd x}{w_{c,1+}(x)} - \log\rho_1(\alpha_1) \\ & = & o(1) - \log\rho_1(\alpha_1)
\end{eqnarray*}
uniformly on compact subsets of \( \C\setminus\Delta_{0,1} \) when \( c\to0 \). Thus, it follows from the maximum modulus principle that
\begin{equation}
\label{prop-szego3}
D_{c,\rho_1}(z) = 1+o(1) \quad \text{and} \quad G_{c,\rho_1} = \big(1+o(1)\big)\rho_1(\alpha_1)
\end{equation}
locally uniformly in \( \overline\C\setminus\Delta_{0,1} \) as \( c\to 0 \).  One can also see from its very definition in \eqref{prop-szego2} combined with the second formula of \eqref{prop-szego3} that \( G_{c,\rho_1} \) extends to a non-vanishing continuous function of \( c\in[0,1] \) (it is constant for all \( c\geq c^* \)). This observation as well as \eqref{prop-szego2} combined with positivity of \( \rho_1(x) \) on \( \Delta_1 \) show that \( |D_{c,\rho_1\pm}(x)|\sim 1 \) uniformly on \( \Delta_{c,1} \) for all \( c\in(0,1) \). Then the maximum modulus principle for holomorphic functions applied to \( D_{c,\rho_1}(z) \) and \( D_{c,\rho_1}^{-1}(z) \) yields that
\begin{equation}
\label{prop-szego10}
G_{c,\rho_1},|D_{c,\rho_1}(z)| \sim 1,
\end{equation}
uniformly in  \(\overline\C \) for all \( c\in(0,1) \) (notice that \( |D_{c,\rho_1}(z)| \) is a continuous function on the entire sphere \( \overline\C \) independent of \( c \) when \( c\geq c^* \)).

Let \( \Gamma_{c,2} := \chi_c(\boldsymbol\Delta_{c,2}) \), which are clockwise oriented analytic Jordan curves (recall that \( \boldsymbol\Delta_{c,2} \) is oriented so that \( \RS_c^{(0)} \) remains on the left when \( \boldsymbol\Delta_{c,2} \) is traversed in the positive direction and that \( \chi_c(\z) \) is conformal on \( \RS_c \) and maps \( \infty^{(0)} \) into \( \infty \)). The function
\begin{equation}
\label{prop-szego11a}
S_{c,2}(\z) := \exp\left\{ \frac1{2\pi\ic}\int_{\Gamma_{c,2}}\frac{\log (D_{c,1}D_{c,\rho_1})\big(\pi\big(\chi_c^{-1}(s)\big)\big)}{s-\chi_c(\z)}\dd s\right\}
\end{equation}
is holomorphic and bounded in \( \RS_c\setminus\boldsymbol \Delta_{c,2} \) and has value 1 at \( \infty^{(0)} \). It follows from Plemelj-Sokhotski formulae that
\begin{equation}
\label{prop-szego14a}
S_{c,2-}(\x) = S_{c,2+}(\x)(D_{c,1}D_{c,\rho_1})(x), \quad \x\in\boldsymbol \Delta_{c,2}.
\end{equation}
Observe also that \( (D_{c,1}D_{c,\rho_1})(\pi(\z)) \) is holomorphic in a neighborhood of \( \boldsymbol\Delta_{c,2} \). Therefore, \( S_{c,2}(\z) \) can be continued analytically across each side of \( \boldsymbol \Delta_{c,2} \). In fact, this continuation has an integral representation similar to \eqref{prop-szego11a}, where one simply needs to homologously deform \( \Gamma_{c,2} \) within the domain of holomorphy of \( (D_{c,1}D_{c,\rho_1})\big(\pi\big(\chi_c^{-1}(s)\big)\big) \). Moreover, it holds that
\begin{equation}
\label{prop-szego5a}
S_{c,2}(\z) = 1 + o(1) \qasq c\to0 \qandq |S_{c,2}(\z)|\sim 1, \quad c\in(0,c^{**}],
\end{equation}
uniformly on \( \RS_c \) (again, this means including the traces on \( \boldsymbol\Delta_{c,2}\)). Indeed, observe that the analytic curves \( \Gamma_{c,2} \) approach the circle \( \big\{|z-B_{0,2}|=(\beta_2-\alpha_2)/4\big\} \) by \eqref{AngPar2} and \eqref{prop-rec0}. Let \( \delta>0 \) be small enough so that the integrand in \eqref{prop-szego11a} is analytic in a neighborhood of the closure of the annular domain bounded by \( \Gamma_{c,2} \) and \( C_\delta := \{ |z-B_{0,2}| = 2\delta+(\beta_2-\alpha_2)/4  \} \). Assuming that \( C_\delta \) is clockwise oriented, it follows from Cauchy's theorem that \( \Gamma_{c,2} \) can be replaced by \( C_\delta \) whenever \( \z\in\RS_c^{(2)} \), i.e., whenever \( \chi_c(\z) \) is interior or on \( \Gamma_{c,2} \). Then it trivially holds that
\[
|S_{c,2}(z)| \leq \exp\left\{\frac{|C_\delta|}{2\pi\delta}\max_{s\in C_\delta}\left|\log (D_{c,1}D_{c,\rho_1})\big(\pi\big(\chi_c^{-1}(s)\big)\big)\right|\right\}, 
\]
for \( \z\in\RS_c^{(2)} \), where \( |C_\delta| \) is the arclength of \( C_\delta \). The desired limit in \( \RS_c^{(2)} \) now follows from \eqref{prop-szego1a} and \eqref{prop-szego3} while the uniform boundedness follows from \eqref{prop-szego1b} and \eqref{prop-szego10}. Clearly, the estimates in the remaining part of \( \RS_c \) can be obtained analogously by deforming \( \Gamma_{c,2} \) into the circles \( \{ |z-B_{0,2}| = -2\delta+(\beta_2-\alpha_2)/4  \} \).

As a part of the final piece of our construction, let \( \Gamma_{c,1} := \chi_c(\boldsymbol\Delta_{c,1}) \). Similarly to \( \Gamma_{c,2} \), these are clockwise oriented analytic Jordan curves that collapse into a point \( B_{0,1} \) by \eqref{AngPar2} and \eqref{prop-rec0}. Let
\begin{equation}
\label{prop-szego11b}
S_{c,1}(\z) := \exp\left\{ \frac1{2\pi\ic}\int_{\Gamma_{c,1}}\frac{\log\left[S_{\rho_2}\big(\pi\big(\chi_c^{-1}(s)\big)\big)/S_{\rho_2}(\infty)\right]}{s-\chi_c(\z)}\dd s\right\},
\end{equation}
which is a holomorphic and bounded function on \( \RS_c \) that has value \( 1 \) at \( \infty^{(1)} \) and whose traces on \( \boldsymbol\Delta_{c,1} \) are continuous and satisfy
\begin{equation}
\label{prop-szego14b}
S_{c,1-}(\x) = S_{c,1+}(\x)S_{\rho_2}(x)/S_{\rho_2}(\infty), \quad \x\in\boldsymbol \Delta_{c,1},
\end{equation}
by Plemelj-Sokhotski formulae. Notice that all the observation about analytic continuations (contour deformation) made for \( S_{c,2}(\z) \) apply to \( S_{c,1}(\z) \) as well. Since the Cauchy kernel is integrated against the pullback of a fixed function \( S_{\rho_2}(z)/S_{\rho_2}(\infty) \) from \( \Delta_{c,1} \) while the curves \( \Gamma_{c,1} \) collapse into a point, straightforward estimates of Cauchy integrals as well as analytic continuation (deformation of a contour) technique yield that
\begin{equation}
\label{prop-szego5b}
S_{c,1}(\z) = 1 + o(1) \qasq c\to0 \qandq |S_{c,1}(\z)|\sim 1, \quad c\in(0,c^{**}],
\end{equation}
locally uniformly on \( (\RS_c^{(0)}\cup \RS_c^{(2)}\big)\setminus\Delta_{c,1} \) and uniformly on \( \RS_c \), respectively. To examine what happens to \( S_{c,1}(\z) \) on \( \RS_c^{(1)} \), given \( \epsilon>0 \), let \( C_\epsilon:=\{|z-B_{c,1}|=\epsilon\} \) be clockwise oriented circle. It follows from \eqref{prop-rec0} that the Jordan curve \( \chi_c^{-1}(C_\epsilon) \) belongs to \( \RS_c^{(0)} \) and is homologous to \( \boldsymbol\Delta_{c,1} \) for all \( c \) sufficiently small. A straightforward computation shows that
\begin{equation}
\label{prop-szego15}
\int_{C_\epsilon}\frac{\log\left[S_{\rho_2}\big(\pi\big(\chi_c^{-1}(s)\big)\big)/S_{\rho_2}(\infty)\right]}{s-B_{c,1}}\frac{\dd s}{2\pi\ic} = \log\frac{S_{\rho_2}(\alpha_1)}{S_{\rho_2}(\infty)} + \mathcal O\left(\max_{z\in \pi\big(\chi_c^{-1}(C_\epsilon)\big)}\left|\log\frac{S_{\rho_2}(z)}{S_{\rho_2}(\alpha_1)}\right|\right).
\end{equation}
It further follows from \eqref{AngPar2} and \eqref{prop-rec0} that Jordan curves \( \pi\big(\chi_c^{-1}(C_\epsilon)\big) \) converge to the analytic Jordan curve \( (\varphi_2+B_{0,2})^{-1}(C_\epsilon) \) (recall that \( \varphi_2(z)=\varphi^{(0)}(z)\)) and the latter curves collapse into a point \( \alpha_1 \) as \( \epsilon\to 0 \). Hence, by taking the limit as \( c\to 0 \) and then the limit as \( \epsilon\to0 \) of the \( \mathcal O(\cdot) \) in \eqref{prop-szego15} gives \( 0 \). Therefore, analytic continuation (deformation of a contour) technique and \eqref{prop-szego11b} imply that
\begin{equation}
\label{prop-szego16}
\lim_{c\to0}S_{c,1}\big(\infty^{(1)}\big) = \lim_{c\to0}\exp\left\{ \frac1{2\pi\ic}\int_{C_\epsilon}\frac{\log\left[S_{\rho_2}\big(\pi\big(\chi_c^{-1}(s)\big)\big)/S_{\rho_2}(\infty)\right]}{s-B_{c,1}}\dd s\right\} = \frac{S_{\rho_2}(\alpha_1)}{S_{\rho_2}(\infty)}.
\end{equation}

Finally, we are ready to state an alternative formula for the functions \( S_c(\z) \) when \( c\leq c^{**} \). Since relations \eqref{szego-pts2} characterize \( S_c(\z) \) up to multiplication by a cubic root of unity, it follows from the normalization of \( D_{c,1}(z) \) and \( D_{c,\rho_1}(z) \) at infinity, the normalization of \( S_{c,1}(\z) \) and \( S_{c,2}(\z) \) at \( \infty^{(0)} \), and relations \eqref{szego-pts}, \eqref{prop-szego13},  \eqref{prop-szego2},  \eqref{prop-szego14a}, and \eqref{prop-szego14b} that 
\begin{equation}
\label{prop-szego7}
\frac{S_c(\z)}{S_c\big(\infty^{(0)}\big)} = \big(S_{c,1}S_{c,2}\big)(\z)\left\{
\begin{array}{rl}
S_{\rho_2}^{-1}(\infty)(D_{c,1}D_{c,\rho_1}S_{\rho_2})(z), & \z\in\RS_c^{(0)}\setminus\big(\boldsymbol\Delta_{c,1}\cup\boldsymbol\Delta_{c,2}\big), \medskip \\
\ic\frac{\beta_{c,1}-\alpha_1}4G_{c,\rho_1}(D_{c,1}D_{c,\rho_1})^{-1}(z), & \z\in\RS_c^{(1)}\setminus\boldsymbol\Delta_{c,1}, \medskip \\
(S_{\rho_2}(\infty)S_{\rho_2}(z))^{-1}, & \z\in\RS_c^{(2)}\setminus\boldsymbol\Delta_{c,2}.
\end{array}
\right.
\end{equation}
Now, it follows from \eqref{prop-szego5a} and \eqref{prop-szego5b} that
\begin{equation}
\label{prop-szego6a}
S_c^{(2)}(z)/S_c^{(0)}(\infty) = (1+o(1))(S_{\rho_2}(\infty)S_{\rho_2}(z))^{-1}
\end{equation}
uniformly in \( \overline\C \) (that is, including the traces on \( \Delta_2 \)) as \( c\to 0 \). Similarly, it follows from \eqref{prop-szego1a}, \eqref{prop-szego3}, \eqref{prop-szego5a}, and \eqref{prop-szego5b} that
\begin{equation}
\label{prop-szego6b}
S_c^{(0)}(z)/S_c^{(0)}(\infty) = (1+o(1))S_{\rho_2}(z)/S_{\rho_2}(\infty)
\end{equation}
locally uniformly in \( \overline\C\setminus\Delta_{0,1} \) as \( c\to 0 \). Further, it follows from the middle relation in \eqref{szego-pts2} and the last two asymptotic formulae that 
\begin{equation}
\label{prop-szego6c}
\frac{S_c^{(1)}(z)}{S_c^{(0)}(\infty)} = \frac1{S_c^{(0)}(\infty)^3}\frac{S_c^{(0)}(\infty)}{S_c^{(0)}(z)}\frac{S_c^{(0)}(\infty)}{S_c^{(2)}(z)} = (1+o(1))\frac{S_{\rho_2}(\infty)^2}{S_c^{(0)}(\infty)^3}
\end{equation}
locally uniformly in \( \overline\C\setminus\Delta_{0,1} \) as \( c\to 0 \). Since relations \eqref{prop-szego6a}--\eqref{prop-szego6c} also provide asymptotics for the ratios of \( S_c^{(k)}(\infty)/S_c^{(0)}(\infty) \), the limits in \eqref{szego-limit2} easily follow. In fact, we deduce from \eqref{prop-szego6a} and \eqref{prop-szego6c} that
\begin{equation}
\label{prop-szego17a}
S_c^{(2)}(\infty) = (1+o(1))\frac{S_c^{(0)}(\infty)}{S_{\rho_2}^2(\infty)} \qandq S_c^{(1)}(\infty) = (1+o(1))\left(\frac{S_{\rho_2}(\infty)}{S_c^{(0)}(\infty)}\right)^2.
\end{equation}
On the other hand, it follows from the normalization \( D_{c,1}(z) \) and \( D_{c,\rho_1}(z) \) at infinity, \eqref{rhoi}, \eqref{c-rate}, \eqref{prop-szego3}, \eqref{prop-szego5a}, and \eqref{prop-szego16} that
\begin{equation}
\label{prop-szego17b}
\lim_{c\to0} \frac1c\frac{S_c^{(1)}(\infty)}{S_c^{(0)}(\infty)} = \frac{2\pi\mu^\prime_1(\alpha_1)|w_2(\alpha_1)|S_{\rho_2}(\alpha_1)}{S_{\rho_2}(\infty)}.
\end{equation}
Plugging in the second asymptotic formula of \eqref{prop-szego17a} into \eqref{prop-szego17b}  yields the first limit in \eqref{szego-limit3}. The other two  now follow from \eqref{prop-szego17a}.

\subsection{Auxiliary Estimates, II}

The sole purpose of this subsection is to state the following lemma that follows from \eqref{prop-szego1b}, \eqref{prop-szego10}, \eqref{prop-szego5a}, \eqref{prop-szego5b}, \eqref{prop-szego7}, as well as the analogous results for \( c\in[c^*,1) \) and \(c\to1\).

\begin{lemma}
\label{lem:aux3} 
It holds uniformly on \( \RS_c \) for all \( c\in(0,1) \) that
\[
c\left|\frac{S_c^{(0)}(\infty)}{S_c^{(1)}(\infty)}\right|, \, (1-c)\left|\frac{S_c^{(0)}(\infty)}{S_c^{(2)}(\infty)}\right| \sim 1.
\]
Moreover, let \( \delta>0 \) be such that \( 0<\delta c\leq |\Delta_{c,1}|/8 \) and  \( 0<\delta(1-c)\leq |\Delta_{c,2}|/8 \) for all \( c\in(0,1) \). Then it holds for all \( c\in(0,1) \) that
\[
\left|\frac{S_c^{(0)}(z)}{S_c^{(0)}(\infty)}\right| \sim \delta^{-1/4}
\]
uniformly on each circle \( \{|z-\alpha_1|=\delta c \} \), \( \{ |z-\beta_{c,1}|=\delta c\} \), \( \{|z-\alpha_{c,2}|=\delta(1-c)\} \), and \( \{ |z-\beta_2| = \delta(1-c) \} \); and
\[
1 \lesssim \left|\frac{S_c^{(0)}(z)}{S_c^{(0)}(\infty)}\right| \lesssim \delta^{-1/4}
\]
uniformly on \( \{ \delta c \leq \dist(z,\{\alpha_1,\beta_{c,1}\}) \} \) and \( \{ \delta (1-c) \leq \dist(z,\{\alpha_{c,2},\beta_2\}) \} \). In addition, it holds for all \( c\in(0,1) \) and each \( i\in\{1,2\} \) that
\[
\left|\frac{S_c^{(i)}(z)}{S_c^{(i)}(\infty)}\right| \sim \delta^{1/4}
\]
uniformly on circles \( \{ |z-\alpha_{c,i}|=\delta(i-1-(-1)^i c) \} \) and \( \{ |z-\beta_{c,i}| = \delta(i-1-(-1)^i c) \} \); and
\[
\delta^{1/4} \lesssim \left|\frac{S_c^{(i)}(z)}{S_c^{(i)}(\infty)}\right| \lesssim1
\]
uniformly on \( \{ \delta(i-1-(-1)^i c)\leq\dist(z,\{\alpha_{c,i},\beta_{c,i}\}) \} \).
\end{lemma}

\section{Proof of Theorem~\ref{thm:asymp1}}
\label{sec:6}

Let \( \alpha_1\leq x_{\n,1}< x_{\n,2} < \ldots <
x_{\n,n_1}\leq\beta_1 \) be the zeros of \( P_\n(x) \) on \(
\Delta_1 \). Then we can write
\[
P_\n(x) =: P_{\n,1}(x)P_{\n,2}(x), \quad P_{\n,1}(x):=\prod_{i=1}^{n_1}(x-x_{\n,i}).
\]
Observe that the polynomials \( \{P_{\n,1}(x)\}_{\n\in\mathcal N_0}
\) form a normal family in a neighborhood of \( \Delta_2 \). As \(
\deg(P_{\n,2})=n_2 \) and it holds that
\[
\int x^l P_{\n,2}(x)P_{\n,1}(x)\dd\mu_2(x)=0, \quad l\in\{0,\ldots,n_2-1\},
\]
by \eqref{typeII}, the asymptotics of \( P_{\n,2}(z) \) follows
from \cite[Theorem~2.7]{BY10}. Namely, it holds that
\begin{equation}
\label{6.1}
P_{\n,2}(z) = (1+o(1))\big(S_{\rho_2}(z)/S_{\rho_2}(\infty)\big) \left(\prod_{i=1}^{n_1}S(z;x_{\n,i})\right)\varphi_2^{n_2}(z)
\end{equation}
uniformly on compact subsets of \( \overline\C\setminus\Delta_2 \).
Thus, to obtain the asymptotic formula for \( P_\n(z) \), we only
need to show that all the zeros \( \{x_{\n,i}\}_{i=1}^{n_1} \)
approach \( \alpha_1 \). We shall do it in a slightly more general
setting.

\begin{lemma}
\label{lem:6-1} Suppose that \( \mu_2 \) is an
absolutely continuous Szeg\H{o} measure, i.e., \(
\int_{\Delta_2}\log\mu_2^\prime(x)\dd x > - \infty \), and that \(
\mathcal N_0 \) is any marginal sequence, that is, \( n_1/n_2 \to 0
\) as \( |\n| \to \infty \) for \( \n\in\mathcal N_0 \). Assuming
formula \eqref{6.1} remains valid, it holds that \(
x_{\n,n_1}\to\alpha_1 \) as \( |\n| \to \infty \) for \(
\n\in\mathcal N_0 \). Moreover,
\begin{equation}
\label{6.0}
\lim_{|\n|\to\infty,~\n\in\mathcal N_0}\lim_{z\to\infty} \left(\frac{P_{\n+\vec e_i}(z)}{P_\n(z)} - z \right) = -B_{0,i}, \quad i\in\{1,2\}.
\end{equation}
\end{lemma}
\begin{proof}
Assume to the contrary that there exists \( \epsilon>0 \) such that
\( \alpha_1+\epsilon \leq x_{\n,n_1} \) along some subsequence \(
\mathcal N^\prime\subset\mathcal N_0 \). Let \( \rho_{\n,1}(x) :=
P_{\n,1}(x)/(x-x_{\n,n_1}) \). Then it follows from \eqref{typeII}
that
\begin{equation}
\label{6.2}
\int_{\alpha_1}^{x_{\n,n_1}} \rho_{\n,1}^2(x)|P_{\n,2}(x)|(x_{\n,n_1}-x)\dd\mu_1(x) = \int_{x_{\n,n_1}}^{\beta_1} \rho_{\n,1}^2(x)|P_{\n,2}(x)|(x-x_{\n,n_1})\dd\mu_1(x),
\end{equation}
(since all the zeros of \( P_{\n,2}(x) \) belong to \( \Delta_2 \),
it has a constant sign on \( \Delta_1 \)). As the zeros of the
monic polynomial \( P_{\n,1}(z) \) belong to \( \Delta_1 \), we
have that \( |P_{\n,1}(x)|\leq |\beta_1-\alpha_1|^{n_1} \), \(
x\in\Delta_1 \). Moreover, since each \( S(z;x_0) \) is a
non-vanishing function in \( \overline\C\setminus\Delta_2 \),
compactness of \( \Delta_1 \) implies that there exists a constant
\( C_1>1 \) such that \( C_1^{-1}\leq |S(x;x_0)|\leq C_1 \) for any
\( x,x_0\in\Delta_1 \). Therefore, we can deduce from \eqref{6.1}
that
\begin{equation}
\label{6.3}
\int_{x_{\n,n_1}}^{\beta_1} \rho_{\n,1}^2(x)|P_{\n,2}(x)|(x-x_{\n,n_1})\dd\mu_1(x) \leq C_2^{n_1}|\varphi_2(\alpha_1+\epsilon)|^{n_2}
\end{equation}
for some absolute constant \( C_2>0 \). On the other hand, by restricting the interval of integration from \( [\alpha_1,x_{\n,n_1}] \) to \( [\alpha_1,\alpha_1+\epsilon/2] \) and then using \eqref{6.1}, the lower estimate of the Szeg\H{o} functions \( S(z;x_0) \), the facts that \( \mu_1^\prime(x) \) is non-vanishing  and \( |\varphi_2(x)|\) is decreasing for \( x<\alpha_2 \) we get that
\begin{multline}
\label{6.4}
\int_{\alpha_1}^{x_{\n,n_1}} \rho_{\n,1}^2(x)|P_{\n,2}(x)|(x_{\n,n_1}-x)\dd\mu_1(x) \geq C_3^{n_1}|\varphi_2(\alpha_1+\epsilon/2)|^{n_2}\int_{\alpha_1}^{\alpha_1+\epsilon/2}\rho_{\n,1}^2(x)\dd x \\
\geq C_3^{n_1} \min_{\n\in\mathcal N^\prime}
\left(\int_{\alpha_1}^{\alpha_1+\epsilon/2}L_{n_1-1}^2(x)\dd x\right)
|\varphi_2(\alpha_1+\epsilon/2)|^{n_2} \geq
C_4^{n_1}|\varphi_2(\alpha_1+\epsilon/2)|^{n_2}
\end{multline}
for some constants \( C_3,C_4 > 0 \) that might depend on \(
\epsilon \), but are independent of \( \n \), where \( L_n(x) \) is
the \(n\)-th monic orthogonal polynomial with respect to \( \dd x
\) on \( [\alpha_1,\alpha_1+\epsilon/2]\) (rescaled Legendre
polynomial) and the last estimate follows from
\cite[Table~18.3.1]{DLMF}. Since \( n_1/n_2 \to 0 \) and \(
|\varphi_2(x)| \) is decreasing on \( (-\infty,\alpha_2)
\), we have that
\[
C_4^{n_1/n_2}|\varphi_2(\alpha_1+\epsilon/2)|> C_2^{n_1/n_2}|\varphi_2(\alpha_1+\epsilon)|
\]
for all \( |\n| \) large, \( \n\in\mathcal N_0 \). Hence, the above
estimate shows that \eqref{6.3}--\eqref{6.4} are incompatible with
\eqref{6.2}. Thus, it indeed holds that \( x_{\n,n_1}\to\alpha_1 \)
as \( |\n|\to\infty \), \( \n\in\mathcal N_0 \). Further, it holds
that
\[
\lim_{z\to\infty} \left(\frac{P_{\n+\vec e_1,1}(z)}{P_{\n,1}(z)}-z\right)  = -\sum_{i=1}^{n_1+1} x_{\n+\vec e_1,i} + \sum_{i=1}^{n_1} x_{\n,i} = -\alpha_1 + o(1) - \sum_{i=1}^{n_1} \big( x_{\n+\vec e_1,i+1} - x_{\n,i} \big).
\]
It is known that the zeros of \( P_\n(z) \) and \( P_{\n+\vec
e_1}(z) \) interlace, see for example \cite[Lemma~A.2]{uApDenY}.
Therefore,
\[
0 \leq  \sum_{i=1}^{n_1} \big( x_{\n+\vec e_1,i+1} - x_{\n,i} \big) \leq x_{\n+\vec e_1,n_1} - x_{\n,1} = o(1),
\]
where the last conclusion follows from the fact that \(
x_{\n,1},x_{\n+\vec e_1,n_1}\to\alpha_1 \) (observe that \(
\{\n+\vec e_1:\n\in\mathcal N_0\} \) is also a marginal sequence).
Thus,
\begin{equation}
\label{6.5}
\lim_{|\n|\to\infty,~\n\in\mathcal N_0}\lim_{z\to\infty} \left(\frac{P_{\n+\vec e_1,1}(z)}{P_{\n,1}(z)} - z \right) = -\alpha_1.
\end{equation}
Furthermore, it follows from the explicit definition \eqref{Sx0}
that
\[
S^2(z;x_0) = \frac{1-\frac{B_{0,2}+\varphi_2(x_0)}z+\mathcal O(z^{-2})}{1-\frac{B_{0,2}+A_{0,2}\varphi_2^{-1}(x_0)}z+\mathcal O(z^{-2})} \frac{1-\frac{B_{0,2}}z+\mathcal O(z^{-2})}{1-\frac{x_0}z},
\]
where we used \eqref{varphii} to get that \( \varphi_2(z) = z -
B_{0,2} + \mathcal O(z^{-1}) \) as \( z\to\infty \). Since
\begin{equation}
\label{6.7}
B_{0,2} + \varphi_2(x_0) - x_0 - A_{0,2}\varphi_2^{-1}(x_0) = 2(B_{0,2} + \varphi_2(x_0) - x_0),
\end{equation}
we have that \( S(z;x_0) = 1 - (B_{0,2} + \varphi_2(x_0) -
x_0)z^{-1} + \mathcal O\big(z^{-2}\big) \) as \( z\to\infty \).
Now, interlacing of the zeros \( \{x_{\n+\vec
e_1,i}\}_{i=1}^{n_1+1}\) and  \( \{x_{\n,i}\}_{i=1}^{n_1}\), their
convergence to \( \alpha_1 \), and monotonicity of \( \varphi_2(z) \) yield similarly to \eqref{6.5} that
\begin{equation}
\label{6.6}
\lim_{|\n|\to\infty,~\n\in\mathcal N_0}\lim_{z\to\infty} z\left(\frac{\prod_{i=1}^{n_1+1}S(z;x_{\n+\vec e_1,i})}{\prod_{i=1}^{n_1}S(z;x_{\n,i})}-1\right) = -(B_{0,2} + \varphi_2(\alpha_1) - \alpha_1).
\end{equation}
Hence, it follows from \eqref{6.1}, \eqref{6.5}, \eqref{6.6}, and \eqref{AngPar2} that the limit in \eqref{6.0} when \( i=1 \) is equal to
\[
\lim_{|\n|\to\infty,~\n\in\mathcal N_0}\lim_{z\to\infty} \left(\frac{P_{\n+\vec e_1,1}(z)}{P_{\n,1}(z)}\frac{\prod_{i=1}^{n_1+1}S(z;x_{\n+\vec e_1,i})}{\prod_{i=1}^{n_1}S(z;x_{\n,i})} - z \right) = -B_{0,1}.
\]
Since \( \{\n+\vec e_2:\n\in\mathcal N_0\} \) is a marginal
sequence as well and the zeros of \( P_\n(z) \) and \( P_{\n+\vec
e_2}(z) \) also interlace, the limit in \eqref{6.0} for \( i=2 \)
follows similarly to the case \( i=1 \).
\end{proof}

\section{Proof of Theorems~\ref{thm:asymp2}--\ref{thm:asymp4}}
\label{sec:8}

To prove Theorems~\ref{thm:asymp2}--\ref{thm:asymp4} we use the
extension to multiple orthogonal polynomials \cite{GerKVA01} of by
now classical approach of Fokas, Its, and Kitaev \cite{FIK91,FIK92}
connecting orthogonal polynomials to matrix Riemann-Hilbert
problems. The RH problem is then analyzed via the non-linear
steepest descent method of Deift and Zhou \cite{DZ93}.

As was agreed in Section~\ref{sec:3.3}, we label
quantities dependent on \( c_\n \) only by the subindex \( \n \) as
in \( \beta_{\n,1} := \beta_{c_\n,1} \),  \( \Delta_{\n,i} :=
\Delta_{c_\n,i} \), etc. If \( \Delta \) is a closed interval, we
denote by \( \Delta^\circ \) the open interval with the same
endpoints. Moreover, when convenient, we write \( \alpha_{\n,1}
(=\alpha_1)\) and \( \beta_{\n,2}(=\beta_2) \) even though they do
not depend on the index \( \n \).

Throughout this section, the reader must keep in mind the
definition of constants \( c^* \) and \( c^{**} \) in
Proposition~\ref{prop:1}. Moreover, we would like to use the symbol
\( c \) as a free parameter from the interval \( [0,1] \), as was
done in the previous sections. Thus, we slightly modify the
notation from the statement of
Theorems~\ref{thm:asymp2}--\ref{thm:asymp4} and assume that we deal
with a sequence of multi-indices \( \mathcal N_{c_\star} \) such
that
\[
c_\n=n_1/|\n|\to c_\star\in[0,1] \qandq n_1,n_2\to\infty \qasq |\n|\to\infty, \quad \n\in\mathcal N_{c_\star}.
\]

 We let \( [\boldsymbol A]_{i,j} \) to stand for \( (i,j) \)-th entry of a matrix \( \boldsymbol A \) and \( \boldsymbol E_{i,j} \) to be the matrix whose entries are all zero except for \( [\boldsymbol E_{i,j}]_{i,j} =1 \). We set \( \boldsymbol I \) to be the identity matrix, \( \sigma_3 := \diag(1,-1) \) to be the third Pauli matrix, and $\sigma(\n) := \diag\left(|\n|, -n_1,-n_2\right)$. Finally, for compactness of notation, we  introduce transformations $\mathsf{T}_i$, $i\in\{1,2\}$,  that act on $2\times2$ matrices in the following way:
\[
\mathsf{T}_1 \left(\begin{matrix} e_{11} & e_{12} \\ e_{21} & e_{22} \end{matrix}\right) = \left( \begin{matrix}  e_{11} & e_{12} & 0 \\ e_{21} & e_{22} & 0 \\ 0 & 0 & 1\end{matrix} \right) \qandq \mathsf{T}_2 \left(\begin{matrix} e_{11} & e_{12} \\ e_{21} & e_{22} \end{matrix}\right) = \left( \begin{matrix}  e_{11} & 0 & e_{12} \\ 0 & 1 & 0 \\ e_{21} & 0 & e_{22}\end{matrix} \right).
\]

\subsection{Initial RH Problem}

Let the measures \( \mu_1,\mu_2 \) be as in
Theorem~\ref{thm:recurrence} and the functions \(
\rho_1(x),\rho_2(x) \) be given by \eqref{rhoi}. Consider the
following Riemann-Hilbert problem (\rhy): find a \( 2\times2\)
matrix function \( \boldsymbol Y(z) \) such that
\begin{itemize}
\label{rhy}
\item[(a)] $\boldsymbol Y(z)$ is analytic in
    $\C\setminus(\Delta_1\cup \Delta_2)$ and $\displaystyle
    \lim_{z\to\infty} {\boldsymbol Y}(z)z^{-\sigma(\n)} =
    {\boldsymbol I}$;
\item[(b)] $\boldsymbol Y(z)$ has continuous traces on
    $\Delta_i^\circ$ that satisfy \( {\boldsymbol Y}_+(x) =
    {\boldsymbol Y}_-(x)( \boldsymbol I + \rho_i(x) \boldsymbol
    E_{1,i+1}) \), \( i\in\{1,2\} \);
\item[(c)] the entries of the $(i+1)$-st column of $\boldsymbol
    Y(z)$ behave like $\mathcal{O}\left(\log|z-\xi|\right)$ as
    $z\to \xi\in\{\alpha_i,\beta_i\}$, while the remaining
    entries stay bounded, \( i\in\{1,2\} \).
\end{itemize}

\begin{lemma}[Proposition~3.1 of \cite{Y16}]
\label{lem:rhy} Solution of \hyperref[rhy]{\rhy} is unique and
given by
\begin{equation}
\label{eq:y}
\boldsymbol Y(z):= \left(\begin{matrix}
P_\n(z) & R_\n^{(1)}(z) &  R_\n^{(2)}(z) \smallskip\\
m_{\n,1}P_{\n-\vec e_1}(z) & m_{\n,1}R_{\n-\vec e_1}^{(1)}(z) & m_{\n,1}R_{\n-\vec e_1}^{(2)}(z) \\
m_{\n,2}P_{\n-\vec e_2}(z) & m_{\n,2}R_{\n-\vec e_2}^{(1)}(z) & m_{\n,2}R_{\n-\vec e_2}^{(2)}(z)
\end{matrix}\right),
\end{equation}
where \( P_\n(z) \) is the polynomial satisfying \eqref{typeII}, \(
R_\n^{(i)}(z) \), \( i\in\{1,2\} \), are its functions of the
second kind, see \eqref{Rni}, $m_{\n,i}$ are constants such that \(
\displaystyle\lim_{z\to\infty}m_{\n,i}R_{\n-\vec
e_i}^{(i)}(z)z^{n_i}=1 \) and \( \vec e_1:=(1,0) \), \( \vec
e_2:=(0,1) \).
\end{lemma}

\subsection{Opening of the Lenses}
\label{ss:OL}

Given \( c\in(0,1) \) and \( \delta>0 \), denote by \( U_{c,\delta,e} \) an open square with vertices \( e \pm c\delta, e
\pm \ic c\delta \) when \(e\in\{\alpha_1,\beta_{c,1}\} \) and \( e
\pm (1-c)\delta, e \pm \ic (1-c)\delta \) when \( e\in\{
\alpha_{c,2},\beta_2 \} \). Define \( \delta_i(c) \), \(
i\in\{1,2\} \), via
\[
\left\{
\begin{array}{l}
\displaystyle \delta_1(c) := \frac1{8c}\left\{
\begin{array}{ll}
\min\{\beta_{c,1}-\alpha_1,\beta_1-\beta_{c,1}\}, & c<c^*, \smallskip \\
\min\{\beta_1-\alpha_1,\alpha_2-\beta_1\}, & c^*\leq c,
\end{array}
\right. \bigskip \\
\displaystyle  \delta_2(c) := \frac1{8(1-c)}\left\{
\begin{array}{ll}
\min\{\beta_2-\alpha_2,\alpha_2-\beta_1\}, & c\leq c^{**}, \smallskip \\
\min\{\beta_2-\alpha_{c,2},\alpha_{c,2}-\alpha_2\}, &c^{**}<c.
\end{array}
\right.
\end{array}
\right.
\]
Of course, it holds that \( c\delta_1(c) \) (resp. \( (1-c)\delta_2(c) \)) is constant for \( c\geq c^* \) (resp. \( c\leq c^{**}\)). Moreover, \( \delta_1(c) \) (resp. \( \delta_2(c) \)) approaches a non-zero constant  as \( c\to0^+ \) (resp. \( c\to1^- \)) by \eqref{c-rate} and it approaches \( 0 \) as \( c\to c^{*-} \) (resp. \( c\to c^{**+} \)). Set \( \delta(c) := \min\{\delta_1(c),\delta_2(c)\} \). For brevity, we write
\[
U_e := U_{c_\n,\delta,e}, \quad \n\in\mathcal N_{c_\star}, \quad e\in E_\n := E_{c_\n}, \; E_c := \{\alpha_1,\beta_{c,1},\alpha_{c,2},\beta_2\},
\] 
assuming that \(\delta\in(0, \delta(c_\star)) \). In particular, all the domains \( U_e \) are disjoint and \( \beta_1 \not\in \overline
U_{\beta_{c,1}} \) when \( c_\star<c^* \) while \( \alpha_2\not\in
\overline U_{\alpha_{c,2}} \) when \( c_\star>c^{**} \), again, for
all \( |\n| \) large enough, \( \n\in \mathcal N_{c_\star} \).

Section~\ref{sec:4.2}  contains a construction of maps \( \zeta_e(z) \), conformal in \( U_e \), \( e\in E_c \), such that \( \zeta_e(z) \) is real on the real line, vanishes at \( e \), and maps \( (\Delta_{c,1}\cup\Delta_{c,2})\cap U_e \) into the negative reals (these subsets of \( \Delta_{c,1}\cup\Delta_{c,2} \) are covered by the darker shading on Figure~\ref{fig:lens}). Using these conformal maps corresponding to \( c_\n \) for \( \n\in\mathcal N_{c_\star} \), we can select piecewise smooth open Jordan arcs \( \Gamma_{\n,i}^\pm \), connecting \( \alpha_{\n,i} \) to \( \beta_{\n,i} \), defined by the following properties:
\begin{equation}
\label{Ipm1}
\zeta_{\beta_{\n,i}}\big(\Gamma_{\n,i}^\pm\cap U_{\beta_{\n,i}}\big) \subset I_\pm:= \big\{z:\arg(z)=\pm2\pi/3\big\}, \quad \zeta_{\alpha_{\n,i}}\big(\Gamma_{\n,i}^\pm\cap U_{\alpha_{\n,i}}\big) \subset I_\mp,
\end{equation}
and \( \Gamma_{\n,i}^\pm \) consist of straight line segments outside of \( U_{\alpha_{\n,i}} \) and \( U_{\beta_{\n,i}} \), see Figure~\ref{fig:lens}.
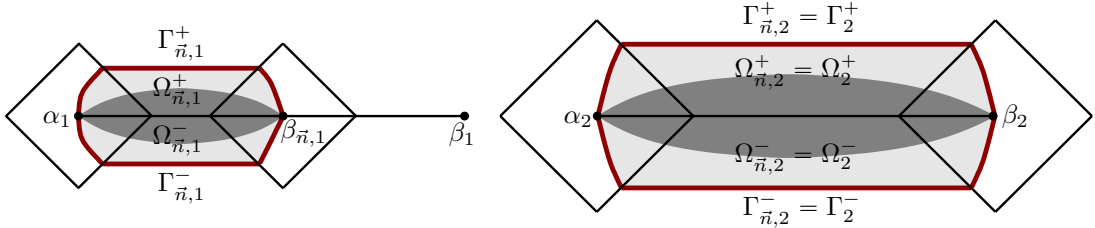
\begin{figure}[!ht]
\centering
\includegraphics[scale=1]{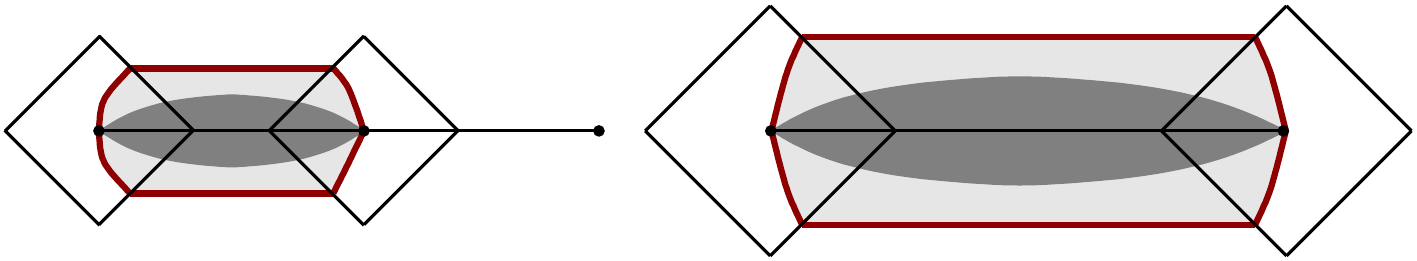}
\begin{picture}(500,0)
\put(390,47){\(\beta_2\)}
\put(226,47){\(\alpha_2\)}
\put(183,40){\(\beta_1\)}
\put(31,47){\(\alpha_1\)}
\put(120,41){\(\beta_{\n,1}\)}
\put(74,75){\(\Gamma_{\n,1}^+\)}
\put(74,21){\(\Gamma_{\n,1}^-\)}
\put(293,84){\(\Gamma_{\n,2}^+=\Gamma_2^+\)}
\put(293,12){\(\Gamma_{\n,2}^-=\Gamma_2^-\)}
\put(72,57){\(\Omega_{\n,1}^+\)}
\put(72,39){\(\Omega_{\n,1}^-\)}
\put(290,64){\(\Omega_{\n,2}^+=\Omega_2^+\)}
\put(290,32){\(\Omega_{\n,2}^-=\Omega_2^-\)}
\end{picture}
\caption{\small The squares \( U_{\alpha_{\n,i}},U_{\beta_{\n,i}} \), and \( U_{\beta_1} \), arcs $\Gamma_{\n,i}^\pm$, domains \( \Omega_{\n,i}^\pm \) (shaded), and the extension domains \( O_{\n,i} \) (darker shaded regions).}
\label{fig:lens}
\end{figure}
When \( c_\star = c^* \), we slightly modify \eqref{Ipm1} and
require that
\begin{equation}
\label{Ipm2}
\widetilde\zeta_{\beta_{\n,1}}\big(\Gamma_{\n,1}^\pm\cap U_{\beta_{\n,1}}\big) \subset I_\pm, \quad \widetilde\zeta_{\beta_{\n,1}}(z) := \zeta_{\beta_{\n,1}}(z) - \zeta_{\beta_{\n,1}}(\beta_1),
\end{equation}
with an analogous modification holding for \( c_\star = c^{**} \)
at \( \alpha_{\n,2} \). We denote by \( \Omega_{\n,i}^\pm \) the
domains delimited by \( \Gamma_{\n,i}^\pm \) and \( \Delta_{\n,i}
\), see Figure~\ref{fig:lens}.

Given $\boldsymbol Y(z)$, the solution of \hyperref[rhy]{\rhy}, set
\begin{equation}
\label{eq:x}
\boldsymbol X(z):= \boldsymbol Y(z) \left\{
\begin{array}{ll}
\mathsf{T}_i\left(\begin{matrix} 1 & 0 \\ \mp1/\rho_i(z) & 1 \end{matrix}\right), & z\in \Omega_{\n,i}^{\pm}, \; i\in\{1,2\}, \medskip \\
\boldsymbol I, & \text{otherwise}.
\end{array}
\right.
\end{equation}
It can be readily verified that $\boldsymbol X(z)$ solves the
following Riemann-Hilbert problem (\rhx):
\begin{itemize}
\label{rhx}
\item[(a)] $\boldsymbol X(z)$ is analytic in
    $\C\setminus\bigcup_{i=1}^2\big(\Delta_i\cup\Gamma_{\n,i}^+\cup\Gamma_{\n,i}^-\big)$
    and $\displaystyle \lim_{z\to\infty} {\boldsymbol
    X}(z)z^{-\sigma(\n)} = {\boldsymbol I}$;
\item[(b)] $\boldsymbol X(z)$ has continuous traces on
    $\bigcup_{i=1}^2\big(\Delta_i^\circ\cup\Gamma_{\n,i}^+\cup\Gamma_{\n,i}^-\big)$
    that satisfy
\[
{\boldsymbol X}_+(s)={\boldsymbol X}_-(s) \left\{
\begin{array}{rl}
\displaystyle \mathsf{T}_i\left(\begin{matrix} 0 & \rho_i(s) \\ -1/\rho_i(s) & 0 \end{matrix}\right), & s\in \Delta_{\n,i},  \medskip \\
\displaystyle \mathsf{T}_i\left(\begin{matrix} 1 & 0 \\ 1/\rho_i(s) & 1 \end{matrix}\right), & s\in \Gamma_{\n,i}^+\cup\Gamma_{\n,i}^-, \medskip \\
\displaystyle \mathsf{T}_i\left(\begin{matrix} 1 & \rho_i(s) \\ 0 & 1 \end{matrix}\right), & s\in \Delta_i^\circ\setminus\Delta_{\n,i},
\end{array}
\right.
\]
for each \( i\in\{1,2\} \);
\item[(c)] the entries of the first and $(i+1)$-st columns of
    $\boldsymbol X(z)$ behave like
    $\mathcal{O}\left(\log|z-\xi|\right)$ as $z\to
    \xi\in\{\alpha_i,\beta_i\}$, while the remaining entries
    stay bounded, \( i\in\{1,2\} \).
\end{itemize}

The following lemma is contained in \cite[Lemma~8.1]{Y16}.

\begin{lemma}[]
\label{lem:rhx} \hyperref[rhx]{\rhx} is solvable if and only if
\hyperref[rhy]{\rhy} is solvable. When solutions of
\hyperref[rhx]{\rhx} and \hyperref[rhy]{\rhy} exist, they are
unique and connected by \eqref{eq:x}.
\end{lemma}

\subsection{Auxiliary Parametrices}

The following Riemann-Hilbert problem (\rhn) is essentially obtained by discarding the jumps of \( \boldsymbol X(z) \) outside of \(
\Delta_{\n,1}\cup\Delta_{\n,2} \):
\begin{itemize}
\label{rhn}
\item[(a)] $\boldsymbol N(z)$ is analytic in
    $\C\setminus(\Delta_{\n,1}\cup \Delta_{\n,2})$ and
    $\displaystyle \lim_{z\to\infty} {\boldsymbol
    N}(z)z^{-\sigma(\n)} = {\boldsymbol I}$;
\item[(b)] $\boldsymbol N(z)$ has continuous traces on
    $\Delta_{\n,i}^\circ$ that satisfy $\boldsymbol N_+(s) =
    \boldsymbol N_-(s)\displaystyle
    \mathsf{T}_i\left(\begin{matrix} 0 & \rho_i(s) \\
    -1/\rho_i(s) & 0 \end{matrix}\right)$;
\item[(c)] it holds that \( \boldsymbol N(z) =
    \boldsymbol{\mathcal O}\big(|z-e|^{-1/4}\big) \) as \(
    z\to e \in E_\n \).
\end{itemize}

Let $S_{\n}(\z):=S_{c_\n}(\z)$ be the one granted by Proposition~\ref{prop:szego}. Put
\[
\boldsymbol S(z):=\diag\big(S_\n^{(0)}(z),S_\n^{(1)}(z),S_\n^{(2)}(z)\big)
\]
for \( z\in\overline \C \setminus (\Delta_{\n,1}\cup\Delta_{\n,2}) \). Further, let  $\Phi_\n(\z)$,  \( w_{\n,i}(z):=w_{c_\n,i}(z) \), and \( \Upsilon_{\n,i}(\z) := \Upsilon_{c_\n,i}(\z) \) be the functions given by \eqref{normalization}, \eqref{wi}, and \eqref{Upsilon}, respectively.  Define
\begin{equation}
\label{matrix-M}
\boldsymbol M(z):=\boldsymbol S^{-1}(\infty)\left(\begin{array}{cccc}
1 & 1/w_{\n,1}(z) & 1/w_{\n,2}(z) \medskip\\
\Upsilon_{\n,1}^{(0)}(z) & \Upsilon_{\n,1}^{(1)}(z)/w_{\n,1}(z) & \Upsilon_{\n,1}^{(2)}(z)/w_{\n,2}(z) \smallskip \\
\Upsilon_{\n,2}^{(0)}(z) & \Upsilon_{\n,2}^{(1)}(z)/w_{\n,1}(z) & \Upsilon_{\n,2}^{(2)}(z)/w_{\n,2}(z)
\end{array}\right)\boldsymbol S(z).
\end{equation}
Then \hyperref[rhn]{\rhn} is solved by \( \boldsymbol N(z) :=
\boldsymbol C (\boldsymbol{MD})(z)\), see \cite[Section~8.2]{Y16},
where $\boldsymbol C$ is a diagonal matrix of constants such that
\begin{equation}
\label{matrix-CD}
\lim_{z\to\infty}\boldsymbol C\boldsymbol D(z)z^{-\sigma(\n)} = \boldsymbol I \qandq \boldsymbol D(z):=\diag\left(\Phi_\n^{(0)}(z),\Phi_\n^{(1)}(z),\Phi_\n^{(2)}(z)\right).
\end{equation}
Since the jump matrix in \hyperref[rhn]{\rhn}(b) has determinant \(
1 \), it follows from \hyperref[rhn]{\rhn}(a,b)  that \(
\det(\boldsymbol N)(z) \) is holomorphic in \( \overline\C \setminus E_\n \) with at most square root singularities at the points of \( E_\n \). Thus,  \( \det(\boldsymbol N)(z) \) is a constant and \( \det(\boldsymbol N)(z)\equiv1 \) by \hyperref[rhn]{\rhn}(a). Therefore, it holds that \( \det(\boldsymbol M)(z)\equiv\det(\boldsymbol D)(z)\equiv\det(\boldsymbol C)=1 \) due to the second relation in \eqref{szego-pts2} and \eqref{normalization}. Moreover, it follows from
\eqref{Up-Pi1} and \eqref{Up-Pi2} that
\begin{equation}
\label{M-inverse}
\boldsymbol M^{-1}(z) = \boldsymbol S^{-1}(z)\left( \begin{matrix} \Pi_\n^{(0)}(z) & \Pi_{\n,1}^{(0)}(z) & \Pi_{\n,2}^{(0)}(z) \medskip \\ w_{\n,1}(z)\Pi_\n^{(1)}(z) & w_{\n,1}(z)\Pi_{\n,1}^{(1)}(z) & w_{\n,1}(z)\Pi_{\n,2}^{(1)}(z) \medskip \\ w_{\n,2}(z)\Pi_\n^{(2)}(z) & w_{\n,2}(z)\Pi_{\n,1}^{(2)}(z) & w_{\n,2}(z)\Pi_{\n,2}^{(2)}(z) \end{matrix} \right)\boldsymbol S(\infty).
\end{equation}
We use the following convention: \( |\boldsymbol A(z)| \lesssim |\boldsymbol B(z)| \) (resp. \( |\boldsymbol A(z)| \sim |\boldsymbol B(z)| \) if all the individual entries satisfy \( |[\boldsymbol A]_{i,j}(z)| \lesssim |[\boldsymbol B]_{i,j}(z)| \) (resp. \( |[\boldsymbol A]_{i,j}(z)| \sim |[\boldsymbol B]_{i,j}(z)| \)). Moreover, if the constants appearing in inequalities  \( \lesssim \) and \( \sim \) do depend on a certain parameter, say \( \delta \), we write \( \lesssim_\delta \) and \( \sim_\delta \). Furthermore, we shall write \( \boldsymbol A(z) = \boldsymbol{\mathcal O}_\delta(1) \) if all the individual entries satisfy \( |[\boldsymbol A]_{i,j}(z)| \lesssim_\delta 1  \).

\begin{lemma}
\label{lem:Mbounds} 
It holds that \( \boldsymbol M^{\pm1}(z) = \boldsymbol{\mathcal O}_\delta(1) \) uniformly for \( z \) such that \( 0<\delta c_\n \leq \dist(z,\{\alpha_1,\beta_{\n,1}\}) \) and \( 0<\delta(1-c_\n) \leq \dist(z,\{\alpha_{\n,2},\beta_2\}) \), where the estimate is independent of the parameter \( c_\n \). Moreover, it holds that \( |\boldsymbol M(z)| \) is
\[
\sim\left( \begin{matrix} \delta^{-1/4} & \delta^{-1/4} & 1-c_\n \smallskip \\ \delta^{-1/4} & \delta^{-1/4} & c_\n(1-c_\n) \smallskip \\  (1-c_\n)\delta^{-1/4} & (1-c_\n)\delta^{-1/4} & 1 \end{matrix} \right) \qandq \sim\left( \begin{matrix} \delta^{-1/4} & c_\n & \delta^{-1/4} \smallskip \\ c_\n\delta^{-1/4} & 1 & c_\n \delta^{-1/4} \smallskip \\  \delta^{-1/4} & c_\n(1-c_\n) & \delta^{-1/4} \end{matrix} \right) 
\]
uniformly on \( |z-\alpha_1| = \delta c_\n \), \( |z-\beta_{\n,1}| = \delta c_\n \) and on \( |z-\alpha_{\n,2}| = \delta (1-c_\n) \), \( |z-\beta_2| = \delta (1-c_\n) \), respectively, where the constants of proportionality are independent of \( c_\n \) and \( \delta \). Finally, it holds that \( \boldsymbol M^{-1}(z)  \) is equal to
\[
\boldsymbol{\mathcal O}\left(\frac1{\delta^{1/4}}\left( \begin{matrix} 1 & 1 & 1-c_\n \smallskip \\ 1 & 1 & 1-c_\n \smallskip \\  (1-c_\n)\delta^{-1/4} & (1-c_\n)\delta^{-1/4} & \delta^{-1/4} \end{matrix} \right)\right) \qandq \boldsymbol{\mathcal O}\left(\frac1{\delta^{1/4}} \left( \begin{matrix} 1 & c_\n & 1 \smallskip \\ c_\n\delta^{-1/4} & \delta^{-1/4} & c_\n \delta^{-1/4} \smallskip \\  1 & c_\n & 1 \end{matrix} \right) \right)
\]
uniformly on \( |z-\alpha_1| = \delta c_\n \), \( |z-\beta_{\n,1}| = \delta c_\n \) and on \( |z-\alpha_{\n,2}| = \delta (1-c_\n) \), \( |z-\beta_2| = \delta (1-c_\n) \), respectively, with \( \boldsymbol{\mathcal O}(\cdot) \) holding independently of \( c_\n \) and \( \delta \).
\end{lemma}
\begin{proof}
Consider first \( z \) on one of the circles from the statement of the lemma. It follows  from \eqref{szego-limit3} and Lemma~\ref{lem:aux3} that
\[
\boldsymbol S(\infty) \sim \diag\left(c_\n^{-1/3}(1-c_\n)^{-1/3},c_\n^{2/3}(1-c_\n)^{-1/3},c_\n^{-1/3}(1-c_\n)^{2/3}\right),
\]
where the constants of proportionality are independent of \( c_\n \). It further follows from Lemma~\ref{lem:aux3} that
\[
|\boldsymbol S(z)|\sim \boldsymbol S(\infty)\diag\left(\delta^{-1/4},\delta^{1/4},1\right) \qandq |\boldsymbol S(z)|\sim \boldsymbol S(\infty)\diag\left(\delta^{-1/4},1,\delta^{1/4}\right)
\]
uniformly on \( |z-\alpha_1| = \delta c_\n \), \( |z-\beta_{\n,1}| = \delta c_\n \) and on \( |z-\alpha_{\n,2}| = \delta (1-c_\n) \), \( |z-\beta_2| = \delta (1-c_\n) \), respectively, where the constants of proportionality are independent of \( c_\n \) and \( \delta \). Moreover, we deduce from Lemma~\ref{lem:aux1a} and \eqref{prop-szego12} that \( \boldsymbol S(\infty)\big|\big(\boldsymbol{MS}^{-1}\big)(z)\big| \) is
\[
\sim \left( \begin{matrix} 1 & c_\n^{-1}\delta^{-1/2} & 1 \smallskip \\ c_\n & \delta^{-1/2} & c_\n^2 \smallskip \\  (1-c_\n)^2 & c_\n^{-1}(1-c_\n)^2\delta^{-1/2} & 1 \end{matrix} \right) \qandq \sim \left( \begin{matrix} 1 & 1 & (1-c_\n)^{-1}\delta^{-1/2} \smallskip \\ c_\n^2 & 1 & c_\n^2(1-c_\n)^{-1}\delta^{-1/2} \smallskip \\  1-c_\n & (1-c_\n)^2 & \delta^{-1/2} \end{matrix} \right)
\]
uniformly on \( |z-\alpha_1| = \delta c_\n \), \( |z-\beta_{\n,1}| = \delta c_\n \) and on \( |z-\alpha_{\n,2}| = \delta (1-c_\n) \), \( |z-\beta_2| = \delta (1-c_\n) \), respectively, where the constants of proportionality are independent of \( c_\n \) and \( \delta \). The combination of the above three estimates yields  the desired asymptotics of \( \boldsymbol M(z) \) on the circles around \( \alpha_1,\beta_{\n,1},\alpha_{\n,2},\beta_2 \).

It further follows from Lemma~\ref{lem:aux3} that
\[
|\boldsymbol S_\pm(x)| \lesssim \boldsymbol S(\infty)\diag\left(\delta^{-1/4},1,1\right)
\]
uniformly for \( x\in(\alpha_1+\delta c_\n,\beta_{\n,1}-\delta c_\n)\cup (\alpha_{\n,2}+\delta(1-c_\n),\beta_2-\delta (1-c_\n)) \)  where the constants of proportionality are independent of \( c_\n \) and \( \delta \). Analogously, it follows from Lemma~\ref{lem:aux1a} and \eqref{prop-szego12} that the above estimate of \( \boldsymbol S(\infty)\big(\boldsymbol{MS}^{-1}\big)(z) \) on the circles remains valid as an upper estimate on \( (\alpha_1+\delta c_\n,\beta_{\n,1}-\delta c_\n)\cup (\alpha_{\n,2}+\delta(1-c_\n),\beta_2-\delta (1-c_\n)) \). The last two observations and the maximum modulus principle for holomorphic functions show that \( \boldsymbol M(z) = \boldsymbol{\mathcal O}_\delta(1) \) uniformly for \( z \) such that \( 0<\delta c_\n \leq \dist(z,\{\alpha_1,\beta_{\n,1}\}) \) and \( 0<\delta(1-c_\n) \leq \dist(z,\{\alpha_{\n,2},\beta_2\}) \), where the estimate is independent of the parameter \( c_\n \).

Finally, as \( \det(\boldsymbol M)(z)\equiv 1 \), the estimates of \( \boldsymbol M^{-1}(z) \) follow in a straightforward fashion from the ones for \( \boldsymbol M(z) \).
\end{proof}

Besides \( \boldsymbol N(z) \), we shall also need matrix functions
that solve \hyperref[rhx]{\rhx} within the domains \( U_e \),
introduced at the beginning of Section~\ref{ss:OL}, with an
additional matching condition on the boundary. More precisely, let
\( \varepsilon_\n \) be given by \eqref{vareps}. For each \(
e\in\{\alpha_1,\beta_{\n,1},\alpha_{\n,2},\beta_2\}
\) we are seeking a solution of the following \rhp$_e$:
\begin{itemize}
\label{rhp}
\item[(a,b,c)] $\boldsymbol P_e(z)$ satisfies
    \hyperref[rhx]{\rhx}(a,b,c) within $U_e$;
\item[(d)] $\boldsymbol P_e(s)=\boldsymbol M(s)(\boldsymbol
    I+\boldsymbol o(1))\boldsymbol D(s)$ uniformly on $\partial
    U_e\setminus\bigcup_{i=1}^2\big(\Delta_i\cup\Gamma_{\n,i}^+\cup\Gamma_{\n,i}^-\big)$,
    where
\[
|[\boldsymbol o(1)]_{j,k}| \leq C\varepsilon_\n \left\{
\begin{array}{ll}
\delta^{-1/2}, & e=\alpha_1, \smallskip \\
\delta^{-3/2}, & e=\beta_{\n,1} \, \text{when} \; c_\star<c^*, \smallskip \\
\big(\delta(z_{c_\star}-\beta_1)\big)^{-1/2}, & e=\beta_1  \; \text{when} \; c_\star>c^*,
\end{array}
\right.
\]
for some constant \( C>0 \) independent of \( \n \) and \( \delta \), and analogous estimates hold around \( \alpha_{\n,2},\beta_2 \) (in the cases
\( c_\star=c^* \) and \( c_\star=c^{**} \) we cannot specify the exact rate of the error term), where the point \( z_c \), or more precisely \( \z_c \) was defined in Proposition~\ref{prop:2}.
\end{itemize}

We will solve \hyperref[rhp]{\rhp$_e$} only for \( e\in\{\alpha_1,\beta_{\n,1}\} \) understanding that the solutions for \( e\in\{ \alpha_{\n,2},\beta_2 \} \) can be constructed similarly. Solution of each \hyperref[rhp]{\rhp$_e$} will require a construction, carried out in the next subsection, of a local conformal map around \( \alpha_1 \) and \( \beta_{\n,1} \). Recall that these maps were already used in  \eqref{Ipm1}.

\subsection{Conformal Maps}
\label{sec:4.2}

In this subsection we construct local conformal maps needed to solve problems \hyperref[rhp]{\rhp\(_e\)}. To this end, recall the definition, given right after \eqref{Hc}, and properties, described in Proposition~\ref{prop:2}, of a function \( h_c(\z) \) that is rational on the surface \( \RS_c \). 

\subsubsection{Local maps around \( \alpha_1\)}
Given \( c\in(0,1) \), define
\begin{equation}
\label{4.1.1}
\zeta_{c,\alpha_1}(z) := \left(\frac14\int_{\alpha_1}^z\left(h_c^{(0)}-h_c^{(1)}\right)(s)\dd s\right)^2, \quad \Re z<\beta_{c,1}.
\end{equation}
Since $h_{c\pm}^{(0)}(x)=h_{c\mp}^{(1)}(x)$ on
$\Delta_{c,1}^\circ$, the function $\zeta_{c,\alpha_1}(z)$ is holomorphic in the
region of definition. When \( \omega \) is a real measure on the real line, it trivially holds that
\[
\int\frac{\dd\omega(x)}{x-(x_0\pm\ic y)} = \int\frac{(x-x_0)\dd\omega(x)}{(x-x_0)^2+y^2} \mp \ic y\int\frac{\dd\omega(x)}{(x-x_0)^2+y^2}.
\]
Therefore, if the traces of \( \int(x-z)^{-1}\dd\omega(x) \) exist at \( x_0 \), they are necessarily conjugate-symmetric. In particular, it follows from \eqref{int-rep} that the integrand in \eqref{4.1.1} is purely imaginary on \( \Delta_{c,1}^\circ \) and therefore \( \zeta_{c,\alpha_1}(x)<0 \) for \( x\in\Delta_{c,1}^\circ \). It also clearly follows from \eqref{int-rep} that $\zeta_{c,\alpha_1}(x)>0$ for $x<\alpha_1$. Moreover, since $h_c(\z)$ has a pole at $\boldsymbol \alpha_1$, a ramification point of \( \RS_c \) of order \(2\), $\zeta_{c,\alpha_1}(z)$ has a simple zero at $\alpha_1$.

\begin{lemma}
\label{lem:4.1} There exist \( \delta_{\alpha_1}>0 \), \(
A_{\alpha_1}>0 \), and \( D_{\alpha_1}>0 \), independent of \( c
\), such that each \( \zeta_{c,\alpha_1}(z) \) is conformal in \(
\{|z-\alpha_1|<\delta_{\alpha_1}c \} \), \( 4A_{\alpha_1}c \leq
|\zeta_{c,\alpha_1}^\prime(\alpha_1)| \), and  \(
|\zeta_{c,\alpha_1}^\prime(z)| \leq D_{\alpha_1}c \) when \(
\{|z-\alpha_1|<\delta_{\alpha_1}c \} \) for all \( c\in(0,1) \).
\end{lemma}
\begin{proof}
We start by proving the estimate on the size of \(
|\zeta_{c,\alpha_1}^\prime(\alpha_1)| \). Assume first that \( c\leq c^*
\). Since \( \boldsymbol \alpha_1 \) is a simple pole of \( h_c(\z) \) and \( h_c^{(0)}(x)<0 \) for \( x<\alpha_1 \) by \eqref{int-rep}, it
holds that \( h_c^{(0)}(x)=u_c(\alpha_1-x)^{-1/2}+ \mathcal O(1) \)
for \( x<\alpha_1 \) and sufficiently close to \( \alpha_1 \), where the branch of the square root is principal and \( u_c<0 \). Since \( h_c^{(1)}(x)=-u_c(\alpha_1-x)^{-1/2}+ \mathcal O(1) \) around \( \alpha_1 \), it can be readily checked that \(
\zeta_{c,\alpha_1}^\prime(\alpha_1)=-u_c^2 \). It was shown in
\cite[Equation~2.7]{ApBogY17} that \( h_c(\z) \) solves
\begin{equation}
\label{4.1.2}
h^3 -\big(1-c+c^2\big)\frac{z-d_c}{\Pi(z)}h - \frac{c-c^2}{\Pi(z)} =0,
\end{equation}
where \( \Pi(z) := (z-\alpha_1)(z-\alpha_2)(z-\beta_2) \) and \(
d_c \) is the point such that the discriminant of \eqref{4.1.2},
whose numerator is a cubic polynomial, vanishes at \( \beta_{c,1}
\) and has an additional double zero. By plugging the identity \( h_c^{(0)}(x)=u_c(\alpha_1-x)^{-1/2}+ \mathcal O(1) \) into
\eqref{4.1.2}, it is easy to verify that
\begin{equation}
\label{4.1.3}
u_c^2 = \big(1-c+c^2\big)\frac{d_c-\alpha_1}{(\alpha_2-\alpha_1)(\beta_2-\alpha_1)}.
\end{equation}
The numerator of the discriminant of \eqref{4.1.2} is equal to
\begin{equation}
\label{4.1.4}
4\big(1-c+c^2\big)^3(z-d_c)^3-27\big(c-c^2\big)^2(z-\alpha_1)(z-\alpha_2)(z-\beta_2)
\end{equation}
and must have a single sign change, which happens at \( \beta_{c,1}
\). If \( d_c\leq \alpha_1 \) were true, then the discriminant would have been positive at \( \alpha_2,\beta_2 \) and non-negative at \( \alpha_1 \), that is, it would have been positive on \( (\alpha_1,\beta_2) \), which contradicts vanishing at \( \beta_{c,1} \). On the other hand if, \( d_c\geq\beta_{c,1} \) were to be true, then the discriminant would have been strictly negative at \( \beta_{c,1} \), which, again, leads to a contradiction. Thus,  \( \alpha_1<d_c<\beta_{c,1} \). Now, \eqref{4.1.2} yields that
\begin{equation}
\label{4.1.5}
\frac{d_c-\alpha_1}c = \frac{1-c}{(\alpha_2-d_c)(\beta_2-d_c)}\frac1{h_c^3(d_c)} \geq \frac{(1-c^*)(\alpha_2-\beta_1)^3}{(\alpha_2-\alpha_1)(\beta_2-\alpha_1)},
\end{equation}
where we used \eqref{int-rep} to observe that \( h_c^{(2)}(d_c)\leq 1/(\alpha_2-\beta_1) \). The above inequality and \eqref{4.1.3} clearly yield the desired
estimate for \( |\zeta_{c,\alpha_1}^\prime(\alpha_1)|=u_c^2 \) when \( c\leq c^* \). In fact, when \( c\to 0 \), it actually follows from the first equality in \eqref{4.1.5} that
\begin{equation}
\label{4.1.6}
\frac{c}{d_c-\alpha_1} = \frac{(\alpha_2-d_c)(\beta_2-d_c)}{1-c}\left(h^{(2)}(d_c)\right)^3 \to (\alpha_2-\alpha_1)(\beta_2-\alpha_1)\left(\int\frac{\dd\omega_2(x)}{x-\alpha_1}\right)^3 = \frac1{|w_2(\alpha_1)|}
\end{equation}
due to \eqref{int-rep}, the last conclusion of Proposition~\ref{prop:1}, and the formula before \eqref{c-rate}. In this case, \eqref{4.1.3} and \eqref{4.1.6} yield that
\begin{equation}
\label{4.1.7}
\zeta_{c,\alpha_1}^\prime(\alpha_1) = -u_c^2 = -c\frac{1+o(1)}{|w_2(\alpha_1)|} \qasq c\to0.
\end{equation}
When \( c\in\big[c^*,c^{**}\big] \) the surface \( \RS_c \) is
always the same. Hence, one can argue using local coordinates that
the pull-backs of \( h_c(\z) \) from a fixed circular neighborhood
of \( \boldsymbol \alpha_1 \) to a fixed neighborhood in \( \C \)
continuously depend on \( c \). Since each \(
|\zeta_{c,\alpha}^\prime(\alpha_1)|>0 \) for \( c\in(0,1) \), the
desired estimate follows from compactness of \( [c^*,c^{**}] \).
When \( c^{**}\leq c \), \( h_c(\z) \) satisfies an equation
similar to \eqref{4.1.2}. Using this equation, we again can argue
that the estimate holds as \( c\to1 \), thus, proving that it holds
uniformly for all \( c\in(0,1) \).

It remains to study conformality of \( \zeta_{c,\alpha_1}(z) \). Denote by \( \delta_{\alpha_1}(c) \) the supremum of \( \delta \) such that \( \zeta_{c,\alpha_1}(z) \) is conformal in \( \{|z-\alpha_1|<2\delta c\} \). We take \( \delta_{\alpha_1} :=\inf_{c\in(0,1)}\delta_{\alpha_1}(c) \). Since \( \delta_{\alpha_1}(c)>0 \) for \( c\in(0,1) \) and continuously depends on \( c \), we only need to study what happens as \( c\to0 \) and \( c\to 1 \) to prove that \( \delta_{\alpha_1}>0 \). Assume first that \( c\to0 \).  Set \( \hat\zeta_{c,\alpha_1}(s) := c^{-2}\zeta_{c,\alpha_1}(z(s)) \), where \( z(s) := \alpha_1 + |\Delta_{c,1}|(1-s)/2 \). Then it follows from \eqref{4.1.1} that
\[
\hat\zeta_{c,\alpha_1}(s) = \left( \frac{|\Delta_{c,1}|}{8c}\int_1^s\big( \hat h_c^{(0)} - \hat h_c^{(1)}\big)(t)\dd t \right)^2,
\]
where \( \hat h_c^{(k)}(s) := h_c^{(k)}(z(s)) \). By using \eqref{4.1.2}, \eqref{4.1.6}, and \eqref{c-rate}, we see that \( \hat h_c \) solves an algebraic equation of the form
\[
\hat h^3  - (1+o(1))\frac{2(1-s)-1-o(1)}{2(1-s)(|w_2(\alpha_1)|^2+o(1))} \hat h - \frac{1+o(1)}{2(1-s)(|w_2(\alpha_1)|^3+o(1))} = 0,
\]
where \( o(1) \) holds uniformly on compact subsets of the plane as \( c\to0 \). The above equation converges to
\begin{equation}
\label{4.1.10}
\left( \hat h^2 + \frac{\hat h}{|w_2(\alpha_1)|} + \frac1{2(1-s)|w_2(\alpha_1)|^2}\right)\left(\hat h - \frac1{|w_2(\alpha_1)|} \right) = 0.
\end{equation}
Since the branches \( \hat h_c^{(0)}(s) \) and \( \hat h_c^{(1)}(s) \)  have \( 1 \) as a branch point, their limits come from the quadratic factor in \eqref{4.1.10}. This observation together with with \eqref{c-rate} readily yield that
that
\begin{equation}
\label{4.1.8}
\hat\zeta_{c,\alpha_1}(s) \to \hat\zeta_{\alpha_1}(s) := \left(\frac12\int_1^s \sqrt\frac{t+1}{t-1}\dd t\right)^2 = \frac14\left(\sqrt{s^2-1} + \log\big( s + \sqrt{s^2-1}\big) \right)^2
\end{equation}
locally uniformly in \( \{|1-s|<2\} \), where the branches of the square roots and the logarithm are principal and therefore \( \hat\zeta_{\alpha_1}(s) \) is holomorphic in \( \C\setminus(-\infty,-1] \) and is positive for \( s\in(1,\infty) \). Using the explicit expression for \( \hat\zeta_{\alpha_1}(s) \), we can conclude that it is conformal in  \( \{|1-s|<2\} \) and therefore \( \liminf_{c\to0}\delta_{\alpha_1}(c)\geq 4|w_2(\alpha_1)|  \) by \eqref{c-rate}. When \( c\to 1 \), we can similarly get from the algebraic equation for \( h_c(\z) \) that \( \zeta_{c,\alpha_1}(z) \) converges to
\begin{equation}
\label{4.1.9}
\left(\frac12\int_{\alpha_1}^z\frac{\dd x}{\sqrt{(x-\alpha_1)(x-\beta_1)}}\right)^2 = \frac14\left( \log \left( \frac{\beta_1+\alpha_1}2 - z - \sqrt{(z-\alpha_1)(z-\beta_1)} \right) - \log\left(\frac{\beta_1-\alpha_1}2 \right) \right)^2,
\end{equation}
which allows us to conclude that \(
\liminf_{c\to1}\delta_{\alpha_1}(c)>0 \) as desired.

Finally, let \( D_{\alpha_1}(c) := c^{-1}\max_{|z-\alpha_1|\leq
\delta_{\alpha_1}c}|\zeta_{c,\alpha_1}^\prime(z)| \). These constants are finite for each \( c\in(0,1) \) since each \( \zeta_{c,\alpha_1}(z) \) is, in fact, analytic in \( \{|z-\alpha_1|<2\delta_{\alpha_1}c\} \). Moreover, since \( \zeta_{c,\alpha_1}(z) \) continuously depends on \( c \), so do the constants  \( D_{\alpha_1}(c) \). Thus, we only need to check their limits as \( c\to 0 \) and \( c\to 1\). The finiteness of \( D_{\alpha_1}:=\sup_{c\in(0,1)}D_{\alpha_1}(c) \) now easily follows from \eqref{c-rate}, \eqref{4.1.8}, and \eqref{4.1.9}.
\end{proof}

\subsubsection{Local maps around \( \beta_{c,1} \) when \( c\in (0,c^*] \)}

Given \( c\in (0,c^*] \), define
\begin{equation}
\label{4.2.1}
\zeta_{\beta_{c,1}}(z) := \left(-\frac34\int_{\beta_{c,1}}^z\left(h_c^{(0)}-h_c^{(1)}\right)(s)\dd s\right)^{2/3}, \quad \alpha_1<\Re z<\alpha_2,
\end{equation}
where the choice of the root function can be made such that \( \zeta_{\beta_{c,1}}(z) \) is holomorphic with a simple zero at \( \beta_{c,1} \) and is positive for \( x>\beta_{c,1} \). Indeed, since \( h_c(\z) \) is bounded at \( \boldsymbol \beta_{c,1} \), which is a ramification point of order \( 2 \), we can write
\begin{equation}
\label{4.2.2}
h_c^{(0)}(x) = h_c(\boldsymbol \beta_{c,1}) - v_c\sqrt{x-\beta_{c,1}} - \mathcal O(x-\beta_{c,1})
\end{equation}
for some number \( v_c \) and \( x>\beta_{c,1} \) sufficiently small. It follows from Proposition~\ref{prop:2} that \( h(\boldsymbol\beta_{c,1}) \) is a non-zero real number. It is also clear from \eqref{int-rep} that \( h_c^{(0)}(x) \) and \( h_c^{(2)}(x) \) assume any non-zero real number somewhere on \( (-\infty,\alpha_1) \cup (\beta_2,\infty) \) and \( (-\infty,\alpha_2)\cup(\beta_2,\infty) \), respectively. Thus, if \( v_c = 0 \), then the function \( h_c(\z) - h(\boldsymbol\beta_{c,1}) \) would have at least four zeros (the zero at \( \boldsymbol\beta_{c,1} \) would be at least a double one), but only three poles, which is impossible. Hence, \( v_c\neq0 \), or more precisely, \( v_c>0 \) since \( h_c^{(0)}(x) \) is a decreasing function on \( (\beta_{c,1},\alpha_2) \) as can be seen \eqref{int-rep}. Therefore, the integrand in \eqref{4.2.1} vanishes as a square root at \( \beta_{c,1} \). Thus, \( \zeta_{\beta_{c,1}}(z) \) has a simple zero there. Again, as in \eqref{4.1.1}, we select such a branch of the root function so that \( \zeta_{\beta_{c,1}}(z) \) is negative on \( \Delta_{c,1}^\circ \). Since the difference \( h_c^{(0)}(x)-h_c^{(1)}(x) \) is real in the gap \( (\beta_{c,1},\alpha_2) \), the map \( \zeta_{\beta_{c,1}}(z) \) is positive there.

\begin{lemma}
\label{lem:4.2} There exist \( \delta_{\beta_1}>0 \) and \(
A_{\beta_1}>0 \), independent of \( c\in(0,c^*] \), such that each
\( \zeta_{\beta_{c,1}}(z) \) is conformal in \(
\{|z-\beta_{c,1}|<\delta_{\beta_1}c \} \) and \(
4A_{\beta_1}c^{-1/3} \leq \zeta_{\beta_{c,1}}^\prime(\beta_{c,1}) \)
for all \( c\in(0,c^*] \).
\end{lemma}
\begin{proof}
Since \( \zeta_{\beta_{c,1}}^\prime(\beta_{c,1}) \neq 0 \) for \( c\in(0,c^*] \), to prove the second claim, we only need to consider what happens as \( c\to 0 \). Similarly to considerations preceding \eqref{4.1.10}, let \( \hat h_c(s) := h_c(\beta_{c,1}+|\Delta_{c,1}|(s-1)/2) \) and \( \hat h(s) \) be the limit of \( \hat h_c(s) \) as \( c\to0 \). Then \eqref{4.1.10} gets replaced by 
\[
\left( \hat h^2 + \frac{\hat h}{|w_2(\alpha_1)|} + \frac1{2(s+1)|w_2(\alpha_1)|^2}\right)\left(\hat h - \frac1{|w_2(\alpha_1)|} \right) = 0.
\]
Since each \( \hat h_c^{(0)}(s) \) has branchpoints at \( \pm1 \) and is negative for \( s<-1 \), see \eqref{int-rep}, the same must be true for their limit \( \hat h^{(0)}(s) \). Thus, solving the above quadratic equation gives us
\[
\hat h^{(0)}(s) = -\frac1{2|w_2(\alpha_1)|}\left(1 + \sqrt{\frac{s-1}{s+1}}\right) = -\frac1{2|w_2(\alpha_1)|} - \frac1{2|w_2(\alpha_1)|}\sqrt{\frac{s-1}2} + \mathcal O(s-1).
\]
Plugging the above limit and the substitution \( x=\beta_{c,1}+|\Delta_{c,1}|(s-1)/2 \) into \eqref{4.2.2} yields
\begin{equation}
\label{4.2.3}
\zeta_{\beta_{c,1}}^\prime(\beta_{c,1}) = v_c^{2/3} = c^{-1/3}\frac{1+o(1)}{(16)^{1/3}|w_2(\alpha_1)|}
\end{equation}
as \( c\to 0 \), where \( v_c \) was introduced in \eqref{4.2.2}. This finishes the proof of the second claim of the lemma. To prove the first one, it is enough to observe that
\begin{equation}
\label{4.2.5}
(3c)^{-2/3}\zeta_{\beta_{c,1}}\big(\beta_{c,1}+|\Delta_{c,1}|(s-1)/2\big) \to \left(\int_1^s \sqrt\frac{t-1}{t+1}\dd t\right)^{2/3} = \left(\sqrt{s^2-1} - \log\big( s + \sqrt{s^2-1}\big) \right)^{2/3}
\end{equation}
as \( c\to0 \), where the limit is conformal around \( 1 \).
\end{proof}

\subsubsection{Local maps around \( \beta_1 \) for \( c \) close to \( c^{*} \) from the right}

This construction will be used only for the ray sequences \( \mathcal N_{c^*} \) with infinitely many indices \( \n \) such that \( c_\n>c^* \). By Proposition~\ref{prop:2}, $h_{c^*}(\z)$ is bounded at $\boldsymbol \beta_1$ while $h_c(\z)$  has a simple pole at $\boldsymbol \beta_1$ for all \( c>c^* \) and a simple zero $\z_c$ that approaches $\boldsymbol \beta_1$ as \( c\to c^{*+} \). Since the functions \( h_c(\z) \) converge around \( \boldsymbol\beta_1 \) to \( h_{c^*}(\z) \) as \( c\to c^{*+} \)  by \eqref{int-rep} and Proposition~\ref{prop:1}, we can write
\begin{equation}
\label{4.3.1}
-\frac34\int_{\beta_1}^z\left(h_c^{(0)}-h_c^{(1)}\right)(s)\dd s = \sqrt{z-\beta_1}\left(z-\beta_1-\epsilon_c\right)f_c(z), \quad \alpha_1<\Re z<\alpha_2,
\end{equation}
for some \( \epsilon_c>0 \) such that $\epsilon_c\to 0^+$ as $c\to c^{*+}$, where $f_c(z)$ is a holomorphic function that is real on \( (\alpha_1,\alpha_2) \) (observe that the Puisuex expansion of \( (h_c^{(0)}-h_c^{(1)})(x) \) around \( \beta_{c,1} \) does not have the integral powers of \( x-\beta_1 \)). Similarly, it holds that \( \zeta_{\beta_{c^*,1}}^{3/2}(z)=(z-\beta_1)^{3/2}f_{c^*}(z) \) for some holomorphic function \( f_{c^*}(z) \) that is real on \( (\alpha_1,\alpha_2) \) and is positive at \( \beta_1 \). Since the right-hand side of \eqref{4.3.1} converges to \( \zeta_{\beta_{c^*,1}}^{3/2}(z) \) as \( c\to c^{*+} \), the functions \( f_c(z) \) converge to \( f_{c^*}(z) \) (in particular, \( f_c(\beta_1)>0 \) for all \( c \) sufficiently close to \( c^* \)).

\begin{lemma}
\label{lem:4.3} There exist \( c^\prime>c^* \) and a fixed
neighborhood of \( \beta_1 \) such that for every \(
c\in(c^*,c^\prime] \) there exists a function \(
\hat\zeta_{c,\beta_1}(z) \), conformal in this neighborhood, such
that
\begin{equation}
\label{4.3.4}
-\frac34\int_{\beta_1}^z\left(h_c^{(0)}-h_c^{(1)}\right)(s)\dd s = \hat\zeta_{c,\beta_1}^{3/2}(z) - \hat\zeta_{c,\beta_1}(\beta_1+\epsilon_c)\hat\zeta_{c,\beta_1}^{1/2}(z).
\end{equation}
(we can adjust the constant \( \delta_{\beta_1}>0 \) from Lemma~\ref{lem:4.2} so that the neighborhood of conformality is given by \( \{|z-\beta_1|< \delta_{\beta_1}c^\prime\} \)). Moreover,  \( \hat\zeta_{c,\beta_1}(z) \) is positive for \( x>\beta_1 \) and converges to \( \zeta_{\beta_{c^*,1}}(z) \) as \( c\to c^{*+} \).
\end{lemma}
\begin{proof}
Let \( F(z;\epsilon) \) be a family of holomorphic and non-vanishing functions in \( \{|z|<r_0\} \) that are positive at the origin and continuously depend on the parameter \( \epsilon\in[0,\epsilon_0] \). Consider the equation
\begin{equation}
\label{4.3.2}
u(z;\epsilon)(u(z;\epsilon)-3p_\epsilon)^2 = 2g(z;\epsilon), \quad g(z;\epsilon) := z(z-\epsilon)^2F(z;\epsilon),
\end{equation}
where \( p_\epsilon>0 \) is a parameter that we shall fix in a moment. The solution of this cubic equation is formally given by
\[
\left\{ \begin{array}{l}
u(z;\epsilon) = 2p_\epsilon + v^{1/3}(z;\epsilon) + p_\epsilon^2v^{-1/3}(z;\epsilon), \medskip \\
v(z;\epsilon) = g(z;\epsilon) - p_\epsilon^3 + \sqrt{g(z;\epsilon)(g(z;\epsilon)-2p_\epsilon^3)}.
\end{array}\right.
\]
Observe that \( g^\prime(x;\epsilon)=(x-\epsilon)\big[(3x-\epsilon)F(x;\epsilon)+x(x-\epsilon)F^\prime(x;\epsilon)\big] \). The expression in the square brackets is negative at \( 0 \) and positive at \( \epsilon \). Since \( F(0;\epsilon)>\delta>0 \), independently of \( \epsilon\in[0,\epsilon_0] \) for some \( \delta,\epsilon_0>0\) sufficiently small, the derivative of the expression in the square brackets, that is, \( 3F(x;\epsilon)+(5x-\epsilon)F^\prime(x;\epsilon)+x(x-\epsilon)F^{\prime\prime}(x;\epsilon)\), is positive on \( [0,\epsilon] \) for all \( \epsilon\in[0,\epsilon_0] \), where we might need to decrease \( \epsilon_0 \) if necessary. Hence, there exists a unique point \( x_\epsilon\in(0,\epsilon) \) such that \( g^\prime(x_\epsilon) =0 \). Then we let
\begin{equation}
\label{4.3.3}
2p_\epsilon^3 := g(x_\epsilon;\epsilon) = \max_{x\in[0,\epsilon]}g(x;\epsilon).
\end{equation}
Since \( g^\prime(x;\epsilon)=(x-\epsilon)\big[2xF(x;\epsilon)+(x-\epsilon)(F(x;\epsilon)+xF^\prime(x;\epsilon))\big] \) and \( F(0;\epsilon)>\delta>0 \), independently of \( \epsilon\in[0,\epsilon_0] \), we can decrease \( r_0 \) if necessary so that \( g^\prime(x;\epsilon)>0 \) for \( x\in(\epsilon;r_0) \) and \( \epsilon\in[0,\epsilon_0] \). Thus, there exists a unique \( y_\epsilon\in(\epsilon,r_0) \) such that \( 2p_\epsilon^3=g(y_\epsilon;\epsilon) \) for all \( \epsilon\in[0,\epsilon_0] \), where, again, we might need to decrease \( \epsilon_0 \). Hence, we can choose \( v(z;\epsilon) \) to be holomorphic in \( \{|z|<r_0\}\setminus[0,y_\epsilon] \) and \( v^{1/3}(z;\epsilon) \) such that \( v^{1/3}(x;\epsilon) \to -p_\epsilon \) as \( x\to0^- \).

Now, since \( g(x;\epsilon) - p_\epsilon^3 \) is real on \( [0,y_\epsilon]
\) and changes sign exactly once on each interval \( [0,x_\epsilon] \), \(
[x_\epsilon,\epsilon] \), and \( [\epsilon,y_\epsilon] \) while the
square root vanishes at the endpoint of these intervals, the change
of the argument of \( v_\pm(x;\epsilon) \) is equal to \( 3\pi \).
Thus, we can define \( v^{1/3}(z;\epsilon) \) holomorphically in
\( \{|z|<r_0\}\setminus[0,y_\epsilon] \) as well, where it also holds that \(
v^{1/3}_+(x;\epsilon)v^{1/3}_-(x;\epsilon) = p_\epsilon^2 \) and \( v_\pm(\epsilon;\epsilon) = -e^{\mp2\pi\ic/3}p_\epsilon\). In this case
\( u(z;\epsilon) \) is in fact holomorphic in \( \{|z|<r_0\} \), has a simple zero at the origin, is positive for \( x>0 \), and satisfies \( u(\epsilon;\epsilon)=2p_\epsilon \). Since \(
u(z;0) = z(2F(z;0))^{1/3} \) and \( u(z;\epsilon) \) continuously
depends on \( \epsilon \), we can decrease \( r_0 \) if necessary so that all the function \( u(z;\epsilon) \) are conformal  in \( \{|z|<r_0\} \).

Let \( u(z;\epsilon_c) \) be the discussed solution of \eqref{4.3.2} and \eqref{4.3.3} with \( F(z;\epsilon_c)=f^2_c(z+\beta_1)/2 \).  Then the desired function \( \hat\zeta_{c,\beta_1}(z) \) is given by \( u(z-\beta_1;\epsilon_c) \).
\end{proof}

\subsubsection{Local maps around \( \beta_1 \) when \( c>c^* \)}

This construction will be used only for the ray sequences \( \mathcal N_{c_\star} \) with \( c_\star>c^* \). Similarly to \eqref{4.1.1}, given \( c\in( c^*,1) \), define
\begin{equation}
\label{4.4.1}
\zeta_{c,\beta_1}(z) := \left(\frac14\int_{\beta_1}^z\left(h_c^{(0)}-h_c^{(1)}\right)(s)\dd s\right)^2, \quad \alpha_1<\Re z<\alpha_2.
\end{equation}
Then \( \zeta_{c,\beta_1}(z) \) is holomorphic in the domain of the definition, has a simple zero
at \( \beta_1 \), is real positive for \( x>\beta_1 \), and is real
negative for \( x<\beta_1 \).

\begin{lemma}
\label{lem:4.4} There exists a continuous and non-vanishing function \( \delta_{\beta_1}(c) \) on \( (c^*,1) \) with non-zero one-sided limit at \( 1 \) such that \( \zeta_{c,\beta_1}(z) \) is conformal in \( \{|z-\beta_1|<\delta_{\beta_1}(c) \} \). Moreover, the constant \( A_{\beta_1} \) in Lemma~\ref{lem:4.2} can be adjusted so that \( 4A_{\beta_1}(z_c-\beta_1) \leq |\zeta_{c,\beta_1}^\prime(\beta_1)| \), where \( \z_c \) is the zero of \( h_c(\z) \) described in Proposition~\ref{prop:2}.
\end{lemma}
\begin{proof}
Since \( \zeta_{c,\beta_1}(z) \) has a simple zero at \( \beta_1 \), \( \delta_{\beta_1}(c) \) is simply the largest radius of
conformality, which is clearly positive. Moreover, when \( c\to 1 \), the limiting behavior of \( \zeta_{c,\beta_1}(z) \) is similar to the one described in \eqref{4.1.9} and therefore \( \lim_{c\to1^-}\delta_{\beta_1}(c)>0 \). To prove the second claim of the lemma observe that \( \zeta_{c,\beta_1}^\prime(\beta_1) = u_c^2 \), where
\[
h_c^{(0)}(x)=u_c(x-\beta_1)^{-1/2} + \tilde h_c^{(0)}(x), \quad \tilde h_c^{(0)}(x)=\mathcal O(1) \qasq x\to\beta_1,
\]
exactly as in Lemma~\ref{lem:4.1}. Thus, we only need to investigate what
happens when \( c\to c^{*+} \) (existence of a limit of \( \zeta_{c,\beta_1}(z) \) as \( c\to1 \), which is conformal around \( \beta_1 \), shows that \( |\zeta_{c,\beta_1}^\prime(\beta_1)| \) is bounded from below as \( c\to 1 \)).   It follows from the second part of Proposition~\ref{prop:1} and \eqref{int-rep} that the Puiseux expansion of \( h_c^{(0)}(x) \) must converge to the Puiseux expansion of \( h_{c^*}^{(0)}(x) \) in some punctured neighborhood of \( \beta_1 \). In particular, we have that \( u_c\to0 \) and \( \tilde h_c^{(0)}(x_c) \to h_{c^*}^{(0)}(\beta_1)=h_{c^*}(\boldsymbol\beta_1) \) as \( c\to c^{*+} \) for any sequence of points \( x_c \to \beta_1^+ \) as \( c\to c^{*+} \). Since \( h_c^{(0)}(z_c)=0 \), it holds that \( u_c(z_c-\beta_1)^{-1/2} = - \tilde h_c^{(0)}(z_c)\to -h_{c^*}(\boldsymbol\beta_1) \) as \( c\to c^{*+} \), from which the estimate follows.
\end{proof}

\subsubsection{Estimates of \(  H_c^{(0)}(z)-H_c^{(1)}(z) \) around \( \Delta_{c,1} \)}

The following lemma will be used in the proof of Lemma~\ref{lem:rhzV}, but is presented here due to its connection to the conformal maps constructed above.

\begin{lemma}
\label{lem:4.5} Let \( H_c(\z) \) be as in \eqref{Hc} and \( \delta_{\beta_1} \) as in Lemma~\ref{lem:4.2}. There exists \( \tilde\delta_{\beta_1}\in(0,\delta_{\beta_1}) \) such that given \( c\in(0,c^*) \) and \( \delta\in(0,\tilde\delta_{\beta_1}) \), it holds that
\begin{equation}
\label{4.5.1}
\left(H_c^{(0)}-H_c^{(1)}\right)(x+\ic y) \leq -B_{\beta_1}\delta^{3/2}c, \quad x\in[\beta_{c,1}+\delta c,\alpha_2-\delta c], \; y\in[-\delta c/2,\delta c/2],
\end{equation}
where \( B_{\beta_1}>0 \) is a constant independent of \( c \) and \( \delta \). Moreover, for any fixed \( \delta> 0 \) small enough there exists \( c_\delta>0 \) and \( \epsilon>0 \) such that
\begin{equation}
\label{4.5.1a}
\left(H_c^{(0)}-H_c^{(1)}\right)(x+\ic y) \leq -\epsilon
\end{equation}
for all \( c\in(0,c_\delta) \), \( x\in[\alpha_1+\delta,\alpha_2-\delta] \), and \( y\in[-\delta/2,\delta/2] \). Finally, for any \( c\in(0,1) \), it holds that
\begin{equation}
\label{4.5.2}
\left(H_c^{(0)}-H_c^{(1)}\right)(x\pm \ic\delta c) \geq B_{\beta_1}\delta^{5/2} c, \quad x\in[\alpha_1,\beta_{c,1}].
\end{equation}
\end{lemma}
\begin{proof}
Since \( h_c(\z) = 2\partial_z H_c(\z) \) and \( \boldsymbol \beta_{c,1} \) is a ramification point of \( \RS_c \) belonging to both \( \RS_c^{(0)} \) and \( \RS_c^{(1)} \), it holds that
\begin{equation}
\label{4.5.3a}
\left(H_c^{(0)}-H_c^{(1)}\right)(z) = \Re\left(\int_{\beta_{c,1}}^z \left( h_c^{(0)}-h_c^{(1)}\right)(s)\dd s\right),  \quad \alpha_1<\Re z<\alpha_2.
\end{equation}
It further follows from \eqref{int-rep} that
\[
\partial_x\Re\left( h_c^{(0)}-h_c^{(1)}\right)(x+\ic y) = \int \frac{y^2-(t-x)^2}{((t-x)^2+y^2)^2}\dd(2\omega_{c,1}+\omega_{c,2})(t)<0
\]
when \( |y|< \delta c\leq \dist(x,\Delta_{c,1}\cup\Delta_{c,2}) \). Therefore, it holds that
\[
\left(H_c^{(0)}-H_c^{(1)}\right)(x+\ic y) \leq \left(H_c^{(0)}-H_c^{(1)}\right)(\beta_{c,1}+\delta c+\ic y)
\]
for all \( x\in[\beta_{c,1}+\delta c,\alpha_2-\delta c] \) and \( y\in[-\delta c/2,\delta c/2] \). Now, by combining \eqref{4.2.1} and \eqref{4.5.3a} we get that
\begin{equation}
\label{4.5.3}
\left(H_c^{(0)}-H_c^{(1)}\right)(z) = -\frac43\Re\left(\zeta_{\beta_{c,1}}^{3/2}(z)\right),  \quad \alpha_1<\Re z<\alpha_2,
\end{equation}
for all \( c\in(0,c^*] \). Take \( \tilde\delta_{\beta_1}\leq \sin(\pi/6)\delta_{\beta_1} \). Since each map \( \zeta_{\beta_{c,1}}(z) \) is conformal in \( |z-\beta_{c,1}|<\delta_{\beta_1}c \) and \( \delta<\sin(\pi/6)\delta_{\beta_1} \), every point \( \beta_{c,1}+\delta c+\ic y \) lies within a disk of conformality of \( \zeta_{\beta_{c,1}}(z) \) when \( |y|<\delta c/2 \). Since \( \Arg(\delta c+ \ic y)\in[-\pi/6,\pi/6] \) when \( |y|<\delta c/2 \) and \( \zeta_{\beta_{c,1}}(x) \) is positive for \( x>\beta_{c,1} \), is negative for \( x<\beta_{c,1} \) and has a positive derivative at \( \beta_{c,1} \), there exists \( \delta_c>0 \) such that
\[
\Re\left(\zeta_{\beta_{c,1}}^{3/2}(\beta_{c,1}+\delta c + \ic y) \right) \geq \frac12\left|\zeta_{\beta_{c,1}}^{3/2}(\beta_{c,1}+\delta c + \ic y)\right|
\]
for all \( |y|<\delta c/2 \) and \( \delta<\delta_c \). Since the maps \( \zeta_{\beta_{c,1}}(z) \) continuously depend on \( c \) and have a rescaled conformal limit as \( c\to0 \), see \eqref{4.2.5}, the constants \( \delta_c \) can be chosen so that \( \delta_c\geq\tilde\delta_{\beta_1}>0 \) for all \( c\in(0,c^*) \) and some \( \tilde\delta_{\beta_1}>0 \). Thus,
\[
\left(H_c^{(0)}-H_c^{(1)}\right)(x+\ic y) \leq -\frac23\left|\zeta_{\beta_{c,1}}^{3/2}(\beta_{c,1}+\delta c +\ic y) \right| \leq - B_{\beta_1}\delta^{3/2} c
\]
for \( x\in[\beta_{c,1}+\delta c,\alpha_2-\delta c] \), \( y\in[-\delta c/2,\delta c/2]\), and a constant \( B_{\beta_1}>0 \) independent of \( c \) by Lemma~\ref{lem:4.2} and \eqref{koebe1/4}, which  finishes the proof of \eqref{4.5.1}.

Estimate \eqref{4.5.1a} follows in straightforward fashion from the observation that the left-hand side of \eqref{4.5.1a} converges to \( V^{\omega_2}(\alpha_1)-V^{\omega_2}(x+\ic y) \) as \( c\to 0 \) uniformly on the considered set by Proposition~\ref{prop:1} and \eqref{Hc}, where \( \omega_2 \) is the arcsine distribution on \( \Delta_2 \).

To prove \eqref{4.5.2}, observe that for each \( x\in \Delta_{c,1} \) fixed, the functions \( \big(H_c^{(0)}-H_c^{(1)}\big)(x\pm \ic y) \) are increasing for \( y\in[0,\infty) \) and vanish at \( y=0 \) by  \eqref{Hc} and \eqref{var}. Moreover, since these functions have the same value at conjugate-symmetric points, it is enough to consider only the upper half-plane. As the right-hand side of \eqref{4.5.2} is positive whenever \( c,\delta>0 \), we can assume without loss of generality that \( \delta<\min\{\delta_{\alpha_1},\delta_{\beta_1},\min_{c\in[c^\prime,1)}\delta_{\beta_1}(c) \} \), where \( \delta_{\alpha_1} \), \( \delta_{\beta_1} \), \( c^\prime \), and \( \delta_{\beta_1}(c) \) were introduced in Lemmas~\ref{lem:4.1}, \ref{lem:4.2}, \ref{lem:4.3}, and \ref{lem:4.4}, respectively. 

Suppose that \( |x+\ic\delta c-\alpha_1|<\delta_{\alpha_1}c \). Then it follows from Lemma~\ref{lem:4.1} together with \eqref{koebe1/4} that
\begin{equation}
\label{4.5.5}
\left|\zeta_{c,\alpha_1}^{1/2}(x+\ic \delta c)\right| \geq (A_{\alpha_1}/4)^{1/2}\delta^{1/2} c.
\end{equation}
It clearly holds that \( \Arg(x+\ic \delta c) \in \big[\arctan(\delta/\delta_{\alpha_1}),\pi/2\big] \). Since \(\zeta_{c,\alpha_1}(z) \) is conformal, negative for \( z>\alpha_1 \), and positive for \( z<\alpha_1 \), there exists \( \delta_c>0 \) such that
\begin{equation}
\label{4.5.6}
\Arg\left(\zeta_{c,\alpha_1}^{1/2}(x+\ic\delta c)\right) \in \big(0,(\pi-\arctan(\delta/\delta_{\alpha_1}))/2\big]
\end{equation}
for all \( \delta\in(0,\delta_c) \). Since the maps \( \zeta_{c,\alpha_1}(z) \) continuously depend on \( c \) and have a rescaled conformal limit as \( c\to0 \), see \eqref{4.1.8}, and a conformal limit as \( c\to1 \), see \eqref{4.1.9}, the constants \( \delta_c \) can be chosen so that \( \delta_c\geq\delta_*>0 \) for all \( c\in(0,1) \). However, as mentioned before, without loss of generality we can consider only \( \delta\in(0,\delta_*) \). Furthermore, similarly to \eqref{4.5.3}, it holds that
\[
\left(H_c^{(0)}-H_c^{(1)}\right)(z) = 4\Re\left(\zeta_{c,\alpha_1}^{1/2}(z)\right),  \quad \Re z<\beta_{c,1},
\]
by \eqref{4.1.1}. Thus, combining the above expression with \eqref{4.5.5} and \eqref{4.5.6} gives us
\begin{equation}
\label{4.5.7}
\left(H_c^{(0)}-H_c^{(1)}\right)(x + \ic\delta c) \geq \sin(\arctan(\delta/\delta_{\alpha_1})/2)\left|\zeta_{c,\alpha_1}^{1/2}(x+\ic \delta c)\right| \geq B^\prime\delta^{3/2}c
\end{equation}
for some \( B^\prime>0 \), independent of \( c \) and \( \delta \).

Now, we shall examine what happens when \( x \) lies in the vicinity of \( \beta_{c,1} \). Unfortunately, there are three different constructions of the conformal maps in this case. Thus, we first assume that \( c\in(0,c^*] \) and \( |x+\ic\delta-\beta_{c,1}|<\delta_{\beta_1}c \), see Lemma~\ref{lem:4.2}. Then it follows from Lemma~\ref{lem:4.2} and \eqref{koebe1/4} that
\[
\left|\zeta_{\beta_{c,1}}^{3/2}(x+\ic \delta c)\right| \geq (A_{\beta_1}/4)^{3/2}\delta^{3/2} c.
\]
In the considered case \( \Arg(x+\ic \delta c) \in \big[\pi/2,\pi-\arctan(\delta/\delta_{\beta_1})\big] \). Since the conformal maps \( \zeta_{\beta_{c,1}}^{3/2}(z) \) continuously depend on \( c \), have a rescaled limit when \( c\to0 \), see \eqref{4.2.5}, are positive for \( z>\beta_{c,1} \) and negative for \( x<\beta_{c,1} \), \eqref{4.5.6} gets now replaced by
\begin{equation}
\label{4.5.8}
\Arg\left(\zeta_{\beta_{c,1}}^{3/2}(x+\ic\delta c)\right) \in \big(5\pi/8,(3\pi-\arctan(\delta/\delta_{\alpha_1}))/2\big]
\end{equation}
for all \( \delta\in(0,\delta_*) \) and a possibly adjusted constant \( \delta_*>0 \). Thus, combining the above observations with \eqref{4.5.3} gives us that
\begin{equation}
\label{4.5.9}
\left(H_c^{(0)}-H_c^{(1)}\right)(x + \ic\delta c) \geq (4/3)\sin(\arctan(\delta/\delta_{\beta_1})/2) \left|\zeta_{\beta_{c,1}}^{3/2}(x+\ic \delta c)\right| \geq B^{\prime\prime}\delta^{5/2}c
\end{equation}
for some \( B^{\prime\prime}>0 \), independent of \( \delta \) and \( c \). Let now \( c^\prime \) be the same as in Lemma~\ref{lem:4.3} and \( |x+\ic\delta-\beta_1|<\delta_{\beta_1}c \) for any \( c\in(c^*,c^\prime] \), again, see Lemma~\ref{lem:4.3}. Then it follows from \eqref{4.3.4} that
\[
\left(H_c^{(0)}-H_c^{(1)}\right)(z) = -\frac43\Re\left(\hat\zeta_{c,\beta_1}^{3/2}(z) - \hat\zeta_{c,\beta_1}(\beta_1+\epsilon_c)\hat\zeta_{c,\beta_1}^{1/2}(z)\right).
\]
Since \( \hat\zeta_{c,\beta_1}(x) \) is positive for \( x>\beta_1 \) and negative for \( x<\beta_1 \), it holds that
\[
-\frac43\Re\left(\hat\zeta_{c,\beta_1}^{3/2}(z) - \hat\zeta_{c,\beta_1}(\beta_1+\epsilon_c)\hat\zeta_{c,\beta_1}^{1/2}(z)\right) > -\frac43\Re\left(\hat\zeta_{c,\beta_1}^{3/2}(z)\right)
\]
for \( z \) with \( \Arg(z)\in(0,\pi) \). Since the maps \( \hat\zeta_{c,\beta_1}(z) \) continuously depend on \( c\in[c^*,c^\prime] \), where we set \( \hat\zeta_{c^*,\beta_1}(z):=\zeta_{c^*,\beta_1}(z)\), see Lemma~\ref{lem:4.3}, the constant \( \delta_* \) can be adjusted so that \eqref{4.5.8} remains valid with \( \zeta_{c,\beta_1}(z) \) replaced by \( \hat\zeta_{c,\beta_1}(z) \) for \( |\delta|<\delta_* \) and \( c\in[c^*,c^\prime] \). Hence, we can proceed exactly as in the case \( c\in(0,c^*] \), perhaps, at the expense of possibly adjusting the constant \( B^{\prime\prime} \) in \eqref{4.5.9}. Further, when \( c\in[c^\prime,1) \), it follows from \eqref{4.4.1} that
\[
\left(H_c^{(0)}-H_c^{(1)}\right)(z) = 4\Re\left(\zeta_{c,\beta_1}^{1/2}(z)\right),  \quad \alpha_1<\Re z<\alpha_2.
\]
It also follows from Proposition~\ref{prop:2} and Lemma~\ref{lem:4.4} that \( |\zeta_{c,\beta_1}^\prime(\beta_1)| \) is bounded away from \( 0 \) independently of \( c\in[c^\prime,1) \) (the bound does depend on \( c^\prime \)). Notice also that in this case \eqref{4.5.6} remains valid with \( \delta_{\alpha_1} \) replaced by \( \min_{c\in[c^\prime,1)}\delta_{\beta_1}(c) \). Therefore, \eqref{4.5.7} remains valid as well, where we need to replace \( \zeta_{c,\alpha_1}(z) \) by \( \zeta_{c,\beta_1}(z) \) and, perhaps, adjust \( B^\prime \).

It only remains to examine what happens when \( \alpha_1+\delta^\prime c\leq x\leq \beta_{c,1}-\delta^\prime c \) for some \( \delta^\prime>0 \). To this end, let us denote by \( \tilde h_c(x) \) the following function:
\begin{eqnarray*}
\tilde h_c(x) &:=& 2\ic\Im\big(h_{c+}^{(0)}(x)\big) = h_{c+}^{(0)}(x) - h_{c-}^{(0)}(x) = h_{c+}^{(0)}(x) - h_{c+}^{(1)}(x) \\
&=& 2\ic\Im\big(h_{c-}^{(1)}(x)\big) = -2\ic\Im\big(h_{c+}^{(1)}(x)\big) = -2\ic\Im\big(h_{c-}^{(0)}(x)\big), \quad x\in\Delta_{c,1}^\circ.
\end{eqnarray*}
Let us show that \( \tilde h_c(x)\neq0 \) for \( x\in\Delta_{c,1}^\circ \). Indeed, if \( \tilde h_c(x^\prime)=0 \) for some \( x^\prime\in\Delta_{c,1}^\circ \), then \( h_{c+}^{(0)}(x^\prime) = h_{c-}^{(0)}(x^\prime) = h_{c+}^{(1)}(x^\prime)=h_{c-}^{(1)}(x^\prime) \) and this value is real. That is, there exist \( \x^\prime,\x^{\prime\prime}\in\boldsymbol\Delta_{c,1} \) (\(\pi(\x^\prime)=\pi(\x^{\prime\prime})=x^\prime\)) at which \( h_c(\z) \) assumes the same non-zero real value. On the other hand, when \( c\in(c^*,c^{**}) \), \( h_c(\z) \) has simple poles at \( \boldsymbol\alpha_1, \boldsymbol\beta_1, \boldsymbol\alpha_2, \boldsymbol\beta_2 \). Therefore, it can be clearly seen from \eqref{int-rep} that \( h_c^{(0)}(x) \) assumes every non-zero real value twice, once on \( (-\infty,\alpha_1)\cup(\beta_2,\infty) \) and once on \( (\beta_1,\alpha_2) \). Furthermore, \eqref{int-rep} also shows that \( h_c^{(1)}(x) \) and \( h_c^{(2)}(x) \) assume every non-zero real value once on \( (-\infty,\alpha_1)\cup(\beta_1,\infty) \) and \( (-\infty,\alpha_2)\cup(\beta_2,\infty) \), respectively. As \( h_c(\z) \) has four zeros/poles, it assumes every value exactly four times. Thus, if \( \tilde h_c(x^\prime) \) were zero, \( h_c(\z) \) would assume a given real value six times, which is impossible. Since the proof for the case \( c\in(0,c^*]\cup[c^{**},1) \) is quite similar, the claim follows.

For the next step, we would like to argue that
\[
\tilde h_{\min}:=\inf_{c\in(0,1)}\min_{\alpha_1+\delta^\prime c\leq x\leq \beta_{c,1}-\delta^\prime c}|\tilde h_c(x)|>0.
\]
For that, it will be convenient to consider the rescaled function \( \hat{\tilde h}_c(s) :=\tilde h_c(\beta_{c,1}+|\Delta_{c,1}|(s-1)/2) \). These functions are purely-imaginary and non-vanishing on \( (-1,1) \). It follows from \eqref{c-rate} that there exists \( \delta^{\prime\prime}>0 \) such that
\[
\tilde h_{\min} \geq \inf_{c\in(0,1)}\min_{-1+\delta^{\prime\prime}\leq s\leq 1-\delta^{\prime\prime}}|\hat{\tilde h}_c(s)|.
\]
For each \( c \) fixed, the minimum over \( s \) is clearly non-zero and continuously depends on \( c \). On the other hand, exactly as in Lemma~\ref{lem:4.2}, it holds that
\begin{equation}
\label{4.5.10}
\hat{\tilde h}_c(s) \to -\frac{\ic}{|w_2(\alpha_1)|}\sqrt{\frac{1-s}{1+s}}
\end{equation}
as \( c\to 0 \) uniformly on \( [-1+\delta^{\prime\prime},1-\delta^{\prime\prime}] \), which again, has a non-zero minimum of the absolute  value. Moreover, a computation similar to the one leading to \eqref{4.1.9} gives us that
\begin{equation}
\label{4.5.11}
\hat{\tilde h}_c(s) \to -\frac{4\ic}{\sqrt{\beta_1-\alpha_1}}\frac1{\sqrt{1-s^2}}
\end{equation}
as \( c\to 1 \) uniformly on \( [-1+\delta^{\prime\prime},1-\delta^{\prime\prime}] \), which also has a non-zero minimum of the absolute  value. Hence, it indeed holds that \( \tilde h_{\min}>0 \).

Now, observe that \( \tilde h_c(x) \) is a trace of a function analytic across \( \Delta_{c,1}^\circ \), namely, of
\[
\tilde h_c(z) := \left\{
\begin{array}{ll}
h_c^{(0)}(z) - h_c^{(1)}(z), & \Im z>0, \smallskip \\
h_c^{(1)}(z) - h_c^{(0)}(z), & \Im z<0.
\end{array}
\right.
\]
Therefore, for each \( x^\prime\in[\alpha_1+\delta^\prime c,\beta_{c,1}-\delta^\prime c] \) fixed, there exists \( \delta(c;x^\prime)>0 \) such that
\begin{equation}
\label{4.5.12}
\big| \widetilde H_c(z;x^\prime) \big| \geq (\tilde h_{\min}/4)|z-x^\prime|, \quad \widetilde H_c(z;x^\prime) := \int_{x^\prime}^z \tilde h_c(s)\dd s,
\end{equation}
for all \( |z-x^\prime|<\delta(c;x^\prime)c \) by \eqref{koebe1/4}. Notice that \( \delta(x^\prime)c \) can be taken to be the radius of the largest disk of conformality of \( \widetilde H_c(z;x^\prime) \). Observe also that \( \delta(c;x^\prime) \) continuously depends on \( x^\prime \) and therefore there exists \( \delta(c)>0 \) such that \( \delta(c;x^\prime)\geq \delta(c) \) for all \( x^\prime\in[\alpha_1+\delta^\prime c,\beta_{c,1}-\delta^\prime c] \). Since \( \delta(c) \) can be made to continuously depend on \( c \) and the limits \eqref{4.5.10} and \eqref{4.5.11} hold not only on \( (-1,1) \), but in some neighborhood of \( (-1,1) \) as well, the constant \( \delta_* \) can be adjusted so that \( \delta(c)>\delta_* \) for all \( c\in(0,1) \).

Since the functions \( \widetilde H_c(z;x^\prime) \) are conformal in \( |z-x^\prime|<\delta_*c \) for each \( x^\prime\in[\alpha_1+\delta^\prime c,\beta_{c,1}-\delta^\prime c] \) and are purely imaginary on the real axis, the same continuity and compactness arguments we have been employing throughout the lemma imply that
\begin{equation}
\label{4.5.13}
\Re\left(\widetilde H_c(x^\prime+\ic y;x^\prime)\right) \geq C \big| \widetilde H_c(x^\prime+\ic y;x^\prime) \big|
\end{equation}
for all \( y\in(0,\delta_* c) \) and \( x^\prime\in[\alpha_1+\delta^\prime c,\beta_{c,1}-\delta^\prime c] \), where \( C>0 \) is constant independent of \( c \). Since \( h_c(\z) = 2\partial_z H_c(\z) \), it follows from \eqref{4.5.12} and \eqref{4.5.13} that
\begin{equation}
\label{4.5.14}
\left(H_c^{(0)}-H_c^{(1)}\right)(x + \ic\delta c) = \Re\left(\widetilde H_c(x+\ic\delta c;x)\right) \geq (C\tilde h_{\min}/4)\delta c.
\end{equation}
The estimate in \eqref{4.5.2} now follows from \eqref{4.5.7}, \eqref{4.5.9}, and \eqref{4.5.14}.
\end{proof}

\subsection{Local Parametrices} 
\label{ss:LP}

Below, we construct solutions of \hyperref[rhp]{\rhp$_e$} for \( e\in\{\alpha_1,\beta_{\n,1}\} \), \( \n\in\mathcal N_{c_\star} \).  Recall that the squares \( U_e \) have diagonals of length \( 2\delta c\), where \( \delta\leq \delta(c_\star) \) see Section~\ref{ss:OL}. Additionally, we assume that \( \delta\leq\min\{\delta_{\alpha_1},\delta_{\beta_1}\} \) or \( \delta\leq\min\{\delta_{\alpha_1},\delta_{\beta_1}(c_\star)\} \), depending on \( c_\star \), see Lemmas~\ref{lem:4.1}--\ref{lem:4.4}. Then the maps constructed in Section~\ref{sec:4.2} are conformal in the corresponding squares \( U_e \).

\subsubsection{Matrix \( \boldsymbol P_{\alpha_1}(z) \)}
\label{sss:7.5.1}

Let  $\boldsymbol\Psi(\zeta)$ be a matrix-valued function such that
\begin{itemize}
\label{rhpsi}
\item[(a)] $\boldsymbol\Psi(\zeta)$ is holomorphic in
    $\C\setminus\big(I_+\cup I_-\cup(-\infty,0]\big)$, see
    \eqref{Ipm1};
\item[(b)] $\boldsymbol\Psi(\zeta)$ has continuous traces on
    $I_+\cup I_-\cup(-\infty,0)$ that satisfy
\[
\boldsymbol\Psi_+(\zeta) = \boldsymbol\Psi_-(\zeta)
\left\{
\begin{array}{rl}
\left(\begin{matrix} 0 & 1 \\ -1 & 0 \end{matrix}\right), & \zeta\in(-\infty,0), \medskip \\
\left(\begin{matrix} 1 & 0 \\ 1 & 1 \end{matrix}\right), & \zeta\in I_\pm,
\end{array}
\right.
\]
where \( I_\pm  \) are oriented towards the origin;
\item[(c)] \( \boldsymbol\Psi(\zeta)=\boldsymbol{\mathcal
    O}(\log|\zeta|) \) as $\zeta\to0$;
\item[(d)] $\boldsymbol\Psi(\zeta)$ has the following behavior
    near $\infty$:
\[
\boldsymbol\Psi(\zeta) = \frac{\zeta^{-\sigma_3/4}}{\sqrt2}\left(\begin{matrix} 1 & \ic \\ \ic & 1 \end{matrix}\right)\left(\boldsymbol I+\mathcal{O}\left(\zeta^{-1/2}\right)\right)\exp\left\{2\zeta^{1/2}\sigma_3\right\}
\]
uniformly in $\C\setminus\big(I_+\cup I_-\cup(-\infty,0]\big)$.
\end{itemize}
Solution of \hyperref[rhpsi]{\rhpsi} was constructed explicitly in \cite{KMcLVAV04} with the help of modified Bessel and Hankel functions. Observe that the jump matrices in  \hyperref[rhpsi]{\rhpsi}(b) have determinant one. Therefore, it follows from \hyperref[rhpsi]{\rhpsi}(d) that \( \det(\boldsymbol\Psi(\zeta)) \equiv \sqrt 2 \).

Let \( \zeta_{\n,\alpha_1}(z) := \zeta_{c_\n,\alpha_1}(z) \), see \eqref{4.1.1}, which is conformal in \( U_{\alpha_1} \). It holds due to Lemma~\ref{lem:4.1} and \eqref{koebe1/4} that
\begin{equation}
\label{rhpsi-2}
\left\{|z|<A_{\alpha_1}\delta n_1^2\right\}\subset|\n|^2\zeta_{\n,\alpha_1}(U_{\alpha_1}),
\end{equation}
where \( A_{\alpha_1} \) is independent of \( \delta \) and \( c_\n=n_1/|\n| \). It also follows from \eqref{PhiInt} and \eqref{4.1.1} that
\begin{equation}
\label{rhpsi-1}
\zeta_{\n,\alpha_1}(z) = \left(\frac1{4|\n|}\log\left(\Phi_\n^{(0)}(z)/\Phi_\n^{(1)}(z)\right)\right)^2, \quad z\in U_{\alpha_1}.
\end{equation}

Let \( \boldsymbol D(z) \) be given by \eqref{matrix-CD}. Note also that the matrix \( \sigma_3\boldsymbol\Psi(\zeta)\sigma_3 \) also satisfies \hyperref[rhpsi]{\rhpsi}, but with the orientation of all the rays in \hyperref[rhpsi]{\rhpsi}(b) reversed and \( \ic \) replaced by \( -\ic \) in the asymptotic formula of \hyperref[rhpsi]{\rhpsi}(d). Relation \eqref{rhpsi-1} and  \hyperref[rhpsi]{\rhpsi}(a,b,c) imply
that the matrix
\begin{equation}
\label{rhpsi-3}
\boldsymbol P_{\alpha_1}(z) := \boldsymbol E_{\alpha_1}(z) \mathsf{T}_1\left(\big(\sigma_3\boldsymbol\Psi\sigma_3\big)\left(|\n|^2\zeta_{\n,\alpha_1}(z)\right) \rho_1^{-\sigma_3/2}(z)\left(\Phi^{(0)}_{\n}/\Phi^{(1)}_{\n}\right)^{-\sigma_3/2}(z)\right)\boldsymbol D(z),
\end{equation}
satisfies \hyperref[rhp]{\rhp$_{\alpha_1}$}(a,b,c) for any
holomorphic prefactor $\boldsymbol E_{\alpha_1}(z)$. As
$\zeta_+^{1/4}=\mathrm{i}\zeta^{1/4}_-$ on \( (-\infty,0) \), where
we take the principal branch, it can be easily checked that
\[
\frac{\zeta_+^{-\sigma_3/4}}{\sqrt2}\left(\begin{matrix} 1 & -\ic \\ -\ic & 1 \end{matrix}\right) = \frac{\zeta_-^{-\sigma_3/4}}{\sqrt2}\left(\begin{matrix} 1 & -\ic \\ -\ic & 1 \end{matrix}\right) \left(\begin{matrix} 0 & -1 \\ 1 & 0 \end{matrix}\right)
\]
there. Then \hyperref[rhn]{\rhn}(b) implies that
\begin{equation}
\label{rhpsi-4}
\boldsymbol E_{\alpha_1}(z) := \boldsymbol M(z)\mathsf{T}_1\left(\frac{(|\n|^2\zeta_{\n,\alpha_1}(z)\big)^{-\sigma_3/4}}{\sqrt2}\left(\begin{matrix} 1 & -\ic \\ -\ic & 1 \end{matrix}\right) \rho_1^{-\sigma_3/2}(z)\right)^{-1}
\end{equation}
is holomorphic in $U_{\alpha_1}\setminus\{\alpha_1\}$. Since the first and second columns of $\boldsymbol M(z)$ has at most quarter root singularities at \( \alpha_1 \) and the third one is bounded, see Lemma~\ref{lem:Mbounds}, $\boldsymbol E_{\alpha_1}(z)$ is in fact holomorphic in $U_{\alpha_1}$ as desired.  Finally, \hyperref[rhp]{\rhp$_{\alpha_1}$}(d) follows from \hyperref[rhpsi]{\rhpsi}(d) and \eqref{rhpsi-2}.

Recall that \( \det(\boldsymbol M(z)) \equiv \det(\boldsymbol D(z)) \equiv 1 \) as explained between \eqref{matrix-M} and \eqref{matrix-CD}. Hence, it holds that \( \det(\boldsymbol E_{\alpha_1}(z)) \equiv 1/\sqrt2 \) and respectively \( \det(\boldsymbol P_{\alpha_1}(z))\equiv 1 \).

\subsubsection{Matrix \( \boldsymbol P_{\beta_{\n,1}}(z) \) when \( c_\star \leq c^* \) and \( c_\n\leq c^* \)}

Below, given \( \mathcal N_{c_\star} \), with \( c_\star \leq c^*
\), we solve \hyperref[rhp]{\rhp$_{\beta_{\n,1}}$} along the
subsequence \( \mathcal N_{c_\star}^\leq:=\big\{ \n\in\mathcal
N_{c_\star}: c_\n\leq c^*\big\} \), when such a subsequence is
infinite.  Clearly, \( \mathcal N_{c_\star}^\leq \) only omits
finitely many terms from  \( \mathcal N_{c_\star} \) when \(
c_\star<c^* \).

Given \( \sigma\in\C\setminus(-\infty,0) \) and \(
s\in(-\infty,\infty) \), let  $\boldsymbol\Phi_\sigma(\zeta;s)$ be
a matrix-valued function such that
\begin{itemize}
\label{rhphi}
\item[(a)] $\boldsymbol\Phi_\sigma(\zeta;s)$ is holomorphic in
    $\C\setminus\big(I_+\cup I_-\cup(-\infty,\infty)\big)$;
\item[(b)] $\boldsymbol\Phi_\sigma(\zeta;s)$ has continuous
    traces on $I_+\cup I_-\cup(-\infty,0)\cup(0,\infty)$ that
    satisfy
\[
\boldsymbol\Phi_{\sigma+}(\zeta;s) = \boldsymbol\Phi_{\sigma-}(\zeta;s)
\left\{
\begin{array}{rl}
\left(\begin{matrix} 0 & 1 \\ -1 & 0 \end{matrix}\right), & \zeta\in(-\infty,0), \medskip \\
\left(\begin{matrix} 1 & 0 \\ 1 & 1 \end{matrix}\right), & \zeta\in I_\pm, \medskip \\
\left(\begin{matrix} 1 & \sigma \\ 0 & 1 \end{matrix}\right), & \zeta\in(0,\infty);
\end{array}
\right.
\]
\item[(c)] \( \boldsymbol\Phi_1(\zeta;s)=\boldsymbol{\mathcal
    O}(1) \) and \(
    \boldsymbol\Phi_\sigma(\zeta;s)=\boldsymbol{\mathcal
    O}(\log|\zeta|) \) when \( \sigma\neq1 \) as $\zeta\to0$;
\item[(d)] $\boldsymbol\Phi(\zeta;s)$ has the following
    behavior near $\infty$:
\[
\boldsymbol\Phi_\sigma(\zeta;s) = \frac{\zeta^{-\sigma_3/4}}{\sqrt2}\left(\begin{matrix} 1 & \ic \\ \ic & 1 \end{matrix}\right) \left(\boldsymbol I+\mathcal{O}\left(\zeta^{-1/2}\right)\right) \exp\left\{-\frac23(\zeta+s)^{3/2}\sigma_3\right\}
\]
uniformly in $\C\setminus\big(I_+\cup
I_-\cup(-\infty,\infty)\big)$.
\end{itemize}

As in the previous subsection, notice that \( \det(\boldsymbol\Phi_\sigma(\zeta;s)) \equiv \sqrt 2 \).

Besides \hyperref[rhphi]{\rhphi$_\sigma$}, we shall also need
\rhwphi~ obtained from \hyperref[rhphi]{\rhphi$_0$} by replacing
\hyperref[rhphi]{\rhphi$_0$}(d) with
\begin{itemize}
\label{rhwphi}
\item[(d)] $\widetilde{\boldsymbol\Phi}(\zeta;s)$ has the
    following behavior near $\infty$:
\[
\widetilde{\boldsymbol\Phi}(\zeta;s) = \frac{\zeta^{-\sigma_3/4}}{\sqrt2}\left(\begin{matrix} 1 & \ic \\ \ic & 1 \end{matrix}\right) \left(\boldsymbol I+\mathcal{O}\left(\zeta^{-1/2}\right)\right) \exp\left\{-\frac23\left(\zeta^{3/2}+s\zeta^{1/2}\right)\sigma_3\right\}.
\]
\end{itemize}

When $\sigma=1$ and $s=0$, the Riemann-Hilbert problem
\hyperref[rhphi]{\rhphi$_1$} is well known \cite{DKMLVZ99b}
and is solved using Airy functions. In fact, in this case
\hyperref[rhphi]{\rhphi$_1$}(d) can be improved to
\begin{equation}
\label{rhphi-1}
\boldsymbol\Phi_1(\zeta;0) = \frac{\zeta^{-\sigma_3/4}}{\sqrt2}\left(\begin{matrix} 1 & \ic \\ \ic & 1 \end{matrix}\right) \left(\boldsymbol I+\mathcal{O}\left(\zeta^{-3/2}\right)\right) \exp\left\{-\frac23\zeta^{3/2}\sigma_3\right\}
\end{equation}
uniformly in $\C\setminus\big(I_+\cup
I_-\cup(-\infty,\infty)\big)$. More generally, when $\sigma=1$, the solvability of
these two problems for all $s\in(-\infty,\infty)$ was shown in
\cite{IKOs08} with further properties investigated in
\cite{IKOs09}. The solvability of the general case
$\sigma\in\C\setminus(-\infty,0)$ was obtained in \cite{XuZh11}.
In \cite[Theorem~4.1]{Y16} it was shown that
\hyperref[rhphi]{\rhphi$_\sigma$}(d) can be replaced by
\begin{equation}
\label{rhphi-2}
\boldsymbol\Phi_\sigma(\zeta;s) = \frac{\zeta^{-\sigma_3/4}}{\sqrt2} \left(\begin{matrix} 1 & \mathrm{i} \\ \mathrm{i} & 1 \end{matrix}\right) \left( \boldsymbol I + \mathcal{O}\left(\sqrt{\frac{|s|+1}{|\zeta|+1}}\right) \right)  \exp\left\{-\frac23(\zeta+s)^{3/2}\sigma_3\right\}
\end{equation}
which holds uniformly for $\zeta\in\C\setminus\big(I_+\cup
I_-\cup(-\infty,\infty)\big)$ and $s\in(-\infty,\infty)$ when \(
\sigma\neq 0 \), and uniformly for $s\in[0,\infty)$ when
$\sigma=0$; and that \hyperref[rhwphi]{\rhwphi}(d) can be replaced
by
\begin{equation}
\label{rhphi-3}
\widetilde{\boldsymbol\Phi}(\zeta;s) = \frac{\zeta^{-\sigma_3/4}}{\sqrt2} \left(\begin{matrix} 1 & \mathrm{i} \\ \mathrm{i} & 1 \end{matrix}\right) \left( \boldsymbol I + \mathcal{O}\left(\sqrt{\frac{|s|+1}{|\zeta|+1}}\right) \right)  \exp\left\{-\frac23\left(\zeta^{3/2}+s\zeta^{1/2}\right)\sigma_3\right\}
\end{equation}
uniformly for $\zeta\in\C\setminus\big(I_+\cup
I_-\cup(-\infty,0]\big)$ and $s\in(-\infty,0]$.

Let \( \zeta_{\beta_{\n,1}}(z) := \zeta_{\beta_{c_\n,1}}(z) \) be the functions defined in \eqref{4.2.1} that are conformal in \( U_{\beta_{\n,1}} \), see Lemma~\ref{lem:4.2}. It follows from \eqref{PhiInt} and \eqref{4.2.1} that
\begin{equation}
\label{rhphi-4}
\zeta_{\beta_{\n,1}}(z) = \left(-\frac3{4|\n|}\log\left(\Phi_{\n}^{(0)}/\Phi_{\n}^{(1)}\right)\right)^{2/3}, \quad z\in U_{\beta_{\n,1}}.
\end{equation}
According to Lemma~\ref{lem:4.2} and \eqref{koebe1/4}, it holds that
\begin{equation}
\label{rhphi-5}
\left\{|z|<A_{\beta_1}\delta n_1^{2/3}\right\}\subset|\n|^{2/3}\zeta_{\beta_{\n,1}}(U_{\beta_{\n,1}}),
\end{equation}
where \( A_{\beta_1} \) is independent of \( \n \) with \( \n \in \mathcal N_{c_\star}^\leq \).

Assume now that \( c_\star<c^* \). Recall that is this case \( \beta_1\in U_{\beta_{\n,1}} \) for all \( |\n| \) large enough. Relation \eqref{rhphi-4} and
\hyperref[rhphi]{\rhphi$_1$}(a,b,c) imply that the matrix
\begin{equation}
\label{rhphi-6}
\boldsymbol P_{\beta_{\n,1}}(z) := \boldsymbol E_{\beta_{\n,1}}(z) \mathsf{T}_1\left(\boldsymbol\Phi_1\left(|\n|^{2/3}\zeta_{\beta_{\n,1}}(z);0\right) \rho_1^{-\sigma_3/2}(z)\left(\Phi^{(0)}_{\n}/\Phi^{(1)}_{\n}\right)^{-\sigma_3/2}(z)\right)\boldsymbol D(z),
\end{equation}
satisfies \hyperref[rhp]{\rhp$_{\beta_{\n,1}}$}(a,b,c) for any
holomorphic prefactor $\boldsymbol E_{\beta_{\n,1}}(z)$. As in the
previous subsection, \hyperref[rhn]{\rhn}(b) implies that
\begin{equation}
\label{rhphi-7}
\boldsymbol E_{\beta_{\n,1}}(z) := \boldsymbol M(z)\mathsf{T}_1\left(\frac{(|\n|^{2/3}\zeta_{\beta_{\n,1}}(z)\big)^{-\sigma_3/4}}{\sqrt2}\left(\begin{matrix} 1 & \ic \\ \ic & 1 \end{matrix}\right) \rho_1^{-\sigma_3/2}(z)\right)^{-1}
\end{equation}
is holomorphic in $U_{\beta_{\n,1}}$. Requirement
\hyperref[rhp]{\rhp$_{\beta_{\n,1}}$}(d) now follows from
\eqref{rhphi-1} and \eqref{rhphi-5}.

Assume now that \( c_\star=c^* \) and recall \eqref{Ipm2}. Observe
also that \( \beta_{\n,1}\leq\beta_1 \) for \( \n\in\mathcal
N_{c^*}^\leq \) and therefore \( s_\n :=
|\n|^{2/3}\zeta_{\beta_{\n,1}}(\beta_1)\geq0 \). Then, similarly to
\eqref{rhphi-6}, we get from \eqref{rhphi-4} and
\hyperref[rhphi]{\rhphi$_0$}(a,b,c) that
\begin{equation}
\label{rhphi-8}
\boldsymbol P_{\beta_{\n,1}}(z) := \boldsymbol E_{\beta_{\n,1}}(z) \mathsf{T}_1\left(\boldsymbol\Phi_0\left(|\n|^{2/3}\widetilde\zeta_{\beta_{\n,1}}(z);s_\n\right) \rho_1^{-\sigma_3/2}(z)\left(\Phi^{(0)}_{\n}/\Phi^{(1)}_{\n}\right)^{-\sigma_3/2}(z)\right)\boldsymbol D(z),
\end{equation}
satisfies \hyperref[rhp]{\rhp$_{\beta_{\n,1}}$}(a,b,c), where
holomorphic prefactor $\boldsymbol E_{\beta_{\n,1}}(z)$ is again
given by \eqref{rhphi-7}. Then it follows from \eqref{rhphi-2} and
\eqref{rhphi-4} that
\begin{multline*}
\big(\boldsymbol M^{-1}\boldsymbol P_{\beta_{\n,1}}\boldsymbol
D^{-1}\big)(s) = \mathsf{T}_1\left(
\rho_1^{\sigma_3/2}(s)\frac1{\sqrt2}\left(\begin{matrix} 1 & -\ic
\\ -\ic & 1\end{matrix}\right) \left( 1 +
\frac{\zeta_{\beta_{\n,1}}(\beta_1)}{\widetilde\zeta_{\beta_{\n,1}}(s)}
\right)^{\sigma_3/4}\frac1{\sqrt2}\left(\begin{matrix} 1 & \ic \\
\ic & 1\end{matrix}\right) \times \right. \\ \left. \times
\left(\boldsymbol I + \boldsymbol{\mathcal
O}\left(\sqrt{|\n|^{-2/3}+\zeta_{\beta_{\n,1}}(\beta_1)}\right)
\right) \rho_1^{-\sigma_3/2}(s)\right)
\end{multline*}
for \( s\in\partial U_{\beta_{\n,1}} \). Since \(
\zeta_{\beta_{\n,1}}(\beta_1) \to 0 \) as \( |\n|\to\infty \), \(
\n\in\mathcal N_{c^*}^\leq \), and \(
\widetilde\zeta_{\beta_{\n,1}}(z) \) is bounded below in modulus on
\( \partial U_{\beta_{\n,1}} \),
\hyperref[rhp]{\rhp$_{\beta_{\n,1}}$}(d) follows. As in the previous subsection, we point out that \( \det(\boldsymbol P_{\beta_{\n,1}}(z)) \equiv 1 \).

\subsubsection{Matrix \( \boldsymbol P_{\beta_{\n,1}}(z) \) when \( c_\star=c^* \) and \( c_\n>c^* \)}

Below, we solve \hyperref[rhp]{\rhp$_{\beta_{\n,1}}$} along the subsequence \( \mathcal N_{c^*}^>:=\big\{ \n\in\mathcal N_{c^*}: c_\n > c^*\big\} \), when such a subsequence is infinite. Let \( \hat\zeta_{\n,\beta_1}(z) := \hat\zeta_{c_\n,\beta_1}(z) \) be the conformal map in \( U_{\beta_1} \) constructed in Lemma~\ref{lem:4.3}. As before, it follows from \eqref{PhiInt} that
\begin{equation}
\label{rhwphi-1}
\hat\zeta_{\n,\beta_1}^{3/2}(z) - \hat\zeta_{\n,\beta_1}(\beta_1+\epsilon_\n)\hat\zeta_{\n,\beta_1}^{1/2}(z)= -\frac3{4|\n|}\log\left(\Phi_{\n}^{(0)}/\Phi_{\n}^{(1)}\right), \quad z\in U_{\beta_1}.
\end{equation}
Let \( s_\n := - |\n|^{2/3}\hat\zeta_{\n,\beta_1}(\beta_1+\epsilon_\n) \). As above, it follows from \eqref{rhwphi-1} and
\hyperref[rhwphi]{\rhwphi} that
\[
\boldsymbol P_{\beta_1}(z) := \boldsymbol E_{\beta_1}(z) \mathsf{T}_1\left(\widetilde{\boldsymbol\Phi}\left(|\n|^{2/3}\hat\zeta_{\n,\beta_1}(z);s_\n\right) \rho_1^{-\sigma_3/2}(z) \left(\Phi^{(0)}_\n/\Phi^{(1)}_\n\right)^{-\sigma_3/2}\right)\boldsymbol D(z),
\]
satisfies \hyperref[rhp]{\rhp$_{\beta_1}$}, where $\boldsymbol E_{\beta_1}(z)$ is given by \eqref{rhphi-7} with \( \zeta_{\beta_\n,1}(z) \) replaced by \( \hat\zeta_{\n,\beta_1}(z) \), and it follows from \eqref{rhphi-3} that \hyperref[rhp]{\rhp$_{\beta_1}$}(d) is satisfied with
\[
o(1) = \mathcal O\left(\max\left\{\hat\zeta_{\n,\beta_1}^{1/2}(\beta_1+\epsilon_{\vec n}),|\n|^{-1/3}\right\}\right).
\]
Again, we stress that \( \det(\boldsymbol P_{\beta_1}(z)) \equiv 1 \).

\subsubsection{Matrix \( \boldsymbol P_{\beta_1}(z) \) when \( c_\star>c^* \)}

The construction of \( \boldsymbol P_{\beta_1}(z) \) in the considered case is absolutely identical to the one of \( \boldsymbol P_{\alpha_1}(z) \) in Section~\ref{sss:7.5.1}.

Clearly, we can assume that \( \n\in\mathcal N_{c_\star} \) is such that \( c_\n>c^* \). Let \( \zeta_{\n,\beta_1}(z) := \zeta_{c_\n,\beta_1}(z) \) be the conformal map defined in \eqref{4.4.1}, whose properties were described in Lemma~\ref{lem:4.4}. It follows from \eqref{PhiInt}
and \eqref{4.4.1} that
\[
\zeta_{\n,\beta_1}(z) = \left(\frac1{4|\n|}\log\left(\Phi_\n^{(0)}/\Phi_\n^{(1)}\right)\right)^2, \quad z\in U_{\beta_1}.
\]
According to Lemma~\ref{lem:4.4} and \eqref{koebe1/4} theorem and since \( n_1^2\leq |\n|^2 \), it holds that
\[
\left\{|z|<A_{\beta_1}\delta (z_{c_\star}-\beta_1) n_1^2\right\}\subset|\n|^2\zeta_{\n,\beta_1}(U_{\beta_1}),
\]
where \( \delta_{\beta_1}(c) \) is continuous and non-vanishing on \( (c^*,1] \). Similarly to \eqref{rhpsi-3}, a solution of  \hyperref[rhp]{\rhp$_{\beta_1}$} is given by
\[
\boldsymbol P_{\beta_1}(z) := \boldsymbol E_{\beta_1}(z) \mathsf{T}_1\left(\boldsymbol\Psi\left(|\n|^2\zeta_{\n,\beta_1}(z)\right) \rho_1^{-\sigma_3/2}(z)\left(\Phi^{(0)}_{\n}/\Phi^{(1)}_{\n}\right)^{-\sigma_3/2}(z)\right)\boldsymbol D(z),
\]
where
\[
\boldsymbol E_{\beta_1}(z) := \boldsymbol M(z)\mathsf{T}_1\left(\frac{(|\n|^2\zeta_{\n,\beta_1}(z)\big)^{-\sigma_3/4}}{\sqrt2}\left(\begin{matrix} 1 & \ic \\ \ic & 1 \end{matrix}\right) \rho_1^{-\sigma_3/2}(z)\right)^{-1}.
\]
It again holds that \( \det(\boldsymbol P_{\beta_1}(z)) \equiv 1 \).

\subsection{Solution of \hyperref[rhx]{\rhx}}

Set \( U_\n := U_{\alpha_1}\cup U_{\beta_{\n,1}} \cup U_{\alpha_{\n,2}} \cup U_{\beta_2} \) and \(
\Gamma_\n :=
\Gamma_{\n,1}^+\cup\Gamma_{\n,1}^-\cup\Gamma_{\n,2}^+\cup\Gamma_{\n,2}^-
\). Put
\[
\Sigma_{\n,\delta} := \partial U_\n \cup \left(\left(\Gamma_\n\cup[\beta_{\n,1},\beta_1]\cup[\alpha_2,\alpha_{\n,2}]\right) \setminus \overline U_\n\right),
\]
see Figure~\ref{fig:sigma-lens}.
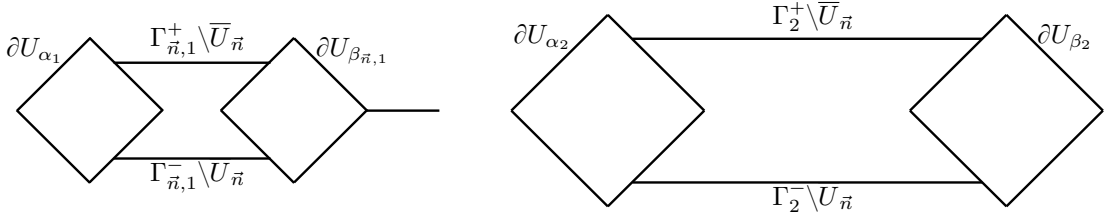
\begin{figure}[!ht]
\centering
\includegraphics[scale=1]{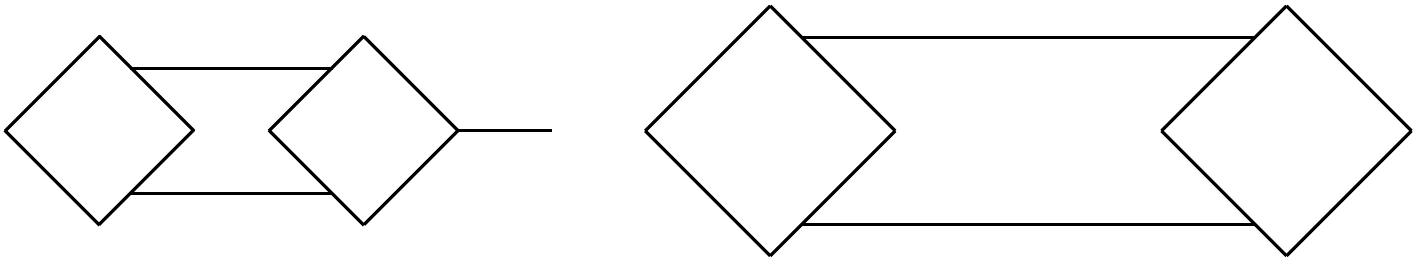}
\begin{picture}(500,0)
\put(399,75){\(\partial U_{\beta_2}\)}
\put(203,75){\(\partial U_{\alpha_2}\)}
\put(13,70){\(\partial U_{\alpha_1}\)}
\put(129,70){\(\partial U_{\beta_{\n,1}}\)}
\put(67,74){\(\Gamma_{\n,1}^+\setminus \overline U_\n\)}
\put(67,23){\(\Gamma_{\n,1}^-\setminus \overline U_\n\)}
\put(300,82){\(\Gamma_2^+\setminus \overline U_\n\)}
\put(300,14){\(\Gamma_2^-\setminus \overline U_\n\)}
\end{picture}
\caption{\small Lens \( \Sigma_{\n,\delta} \) consisting of two connected components \( \Sigma_{\n,\delta,1} \) (the left one) and \( \Sigma_{\n,\delta,2} \) (the right one).}
\label{fig:sigma-lens}
\end{figure}
For definiteness, we agree that all the segments in \( \Sigma_{\n,\delta} \) are oriented from left to right and all the polygons are oriented counter-clockwise. We shall further denote by \( \Sigma_{\n,\delta,1} \) and \( \Sigma_{\n,\delta,2} \) the left and right, respectively, connected components of \( \Sigma_{\n,\delta} \).

For what is to come, we shall need uniform boundedness of the
Cauchy operators on \( \Sigma_{\n,\delta} \). For convenience, we
formulate this claim as a lemma.

\begin{lemma}
\label{lem:cauchy} Given \( r>1 \), there exists a constant \(
C_r>0 \) such that for all \( \delta>0 \) it holds that
\[
\| \mathcal C_\pm f \|_{L^r(\Sigma_{\n,\delta})} \leq C_r \| f \|_{L^r(\Sigma_{\n,\delta})},
\]
where \( \mathcal Cf(z) = \frac1{2\pi\ic}
\int_{\Sigma_{\n,\delta}}\frac{f(t)\dd t}{t-z} \) and \( \mathcal
C_\pm f(s) \) are the traces of \( \mathcal Cf(z) \) on the left
(\(-\)) and right (\(+\)) hand-sides of \( \Sigma_{\n,\delta} \).
\end{lemma}
\begin{proof}
Recall the following known fact, see  \cite[Equation (7.11)]{Deift},
if \( R_1,R_2 \) are two semi-infinite rays with a common endpoint,
then
\begin{equation}
\label{harmonic}
\|\mathcal C_{R_1}f\|_{L^r(R_2)} \leq C_r \|f\|_{L^r(R_1)},
\end{equation}
for some constant \( C_r>0 \) (we can take \( C_2=1 \)), where \( \mathcal C_{R_1} \) is the Cauchy operator defined on \( R_1 \). Moreover, the same estimate holds when \( R_2 =R_1 \) and \( \mathcal C_{R_1} \) is replaced by the trace operators \( \mathcal C_{R_1\pm} \), see \cite[Equations~(7.5)--(7.7)]{Deift}. Trivially, the same estimate holds when \( R_2 \) is replaced by an interval disjoint from \( R_1 \) (may be for an adjusted constant \( C_r \)). Since we can embed any two segments with a common endpoint into semi-infinite rays with a common endpoint and embed a function from \( L^r \) space of a segment into \( L^r \) space of the corresponding ray by extending it by zero, the desired estimate then follows from~\eqref{harmonic} (again, with an adjusted constant~\( C_r \)).
 \end{proof}

Given the global parametrix $\boldsymbol N(z)=\boldsymbol C(\boldsymbol{MD})(z)$ solving \hyperref[rhn]{\rhn}, see \eqref{matrix-M} and \eqref{matrix-CD}, and local parametrices $\boldsymbol P_e(z)$ solving \hyperref[rhp]{\rhp$_e$} and constructed in the previous section, consider the following Riemann-Hilbert Problem (\rhz):
\begin{itemize}
\label{rhz}
\item[(a)] $\boldsymbol{Z}(z)$ is a holomorphic matrix function
    in $\overline\C\setminus\Sigma_{\n,\delta}$ and
    $\boldsymbol{Z}(\infty)=\boldsymbol{I}$;
\item[(b)] $\boldsymbol{Z}(z)$ has continuous traces on  $\Sigma_{\n,\delta}^\circ$  that satisfy
\[
\boldsymbol{Z}_+(s)  = \boldsymbol{Z}_-(s) \left\{
\begin{array}{ll}
(\boldsymbol{MD})(s) \mathsf{T}_i\left(\begin{matrix} 1 & 0 \\ 1/\rho_i(s) & 1 \end{matrix}\right) (\boldsymbol{MD})^{-1}(s), & s\in\Gamma_\n\setminus \overline U_\n, \medskip \\
(\boldsymbol{MD})(s) \mathsf{T}_i\left(\begin{matrix} 1 & \rho_i(s) \\ 0 & 1 \end{matrix}\right) (\boldsymbol{MD})^{-1}(s), & s\in\Delta_i\setminus(\Delta_{\n,i}\cup\overline U_\n), \medskip \\
\boldsymbol P_e(s) (\boldsymbol{MD})^{-1}(s), & s\in\partial U_e, \;\; e\in\{\alpha_1,\beta_{\n,1},\alpha_{\n,2},\beta_2\};
\end{array}
\right.
\]
\item[(c)] around the points of \( \Sigma_{\n,\delta} \setminus \big( \Sigma_{\n,\delta}^\circ \cup \{\beta_1,\alpha_2\} \big) \) the function $\boldsymbol{Z}(z)$ is bounded and around \( \beta_1 \) (resp. \( \alpha_2 \)) its entries are bounded except for those in the second (resp. third) column that behave like \( \mathcal O(\log|z-\beta_1|) \) (resp. \( \mathcal O(\log|z-\alpha_2|) \)).
\end{itemize}

To show existence and prove size estimates of the matrix function \( \boldsymbol Z(z) \), let us first estimate the size of its jump:
\begin{equation}
\label{rhz-V}
 \boldsymbol V(s) := \boldsymbol Z_-^{-1}(s)\boldsymbol{Z}_+(s) - \boldsymbol I, \quad s\in\Sigma_{\n,\delta}.
\end{equation}
More precisely, the following lemma holds.
\begin{lemma}
\label{lem:rhzV}
Let \( \boldsymbol V(s) \) be given by \eqref{rhz-V} and \hyperref[rhz]{\rhz}(b). Then it holds that
\begin{equation}
\label{rhz-2}
\|\boldsymbol V\|_{L^\infty(\Sigma_{\n,\delta})} \lesssim \frac{\varepsilon_\n}{\delta^4} \left\{
\begin{array}{rl}
1, & c_\star\in[0,c^*)\cup(c^{**},1], \smallskip \\
\min\{z_{c_\star}-\beta_1,\alpha_2-z_{c_\star}\}^{-1/2}, & c_\star\in(c^*,c^{**}),
\end{array}
\right.
\end{equation}
with the constant in \( \lesssim \) being independent of \( \delta \) and \( \n \).  Moreover, it also holds that \( \|\boldsymbol V\|_{L^\infty(\Sigma_{\n,\delta})}=o(1) \) when \( c_\star\in\{c^*,c^{**}\} \). 
\end{lemma}
\begin{proof}
We shall prove \eqref{rhz-2} separately for different parts of \( \Sigma_{\n,\delta} \). In fact, we shall do it only on \( \Sigma_{\n,\delta,1} \) understanding that the estimates on \( \Sigma_{\n,\delta,2} \) can be carried out in the same fashion. For \( s\in\partial U_e \), \( e\in\{\alpha_1,\beta_{\n,1}\} \), it holds that \( \boldsymbol V(s) = \boldsymbol P_e(s) (\boldsymbol{MD})^{-1}(s) - \boldsymbol I \). Therefore, the desired estimate \eqref{rhz-2} follows from Lemma~\ref{lem:Mbounds} and \hyperref[rhp]{\rhp$_e$}(d). Let now \( s=x\in \Delta_1\setminus (\Delta_{\n,1}\cup \overline U_\n) \), which is non-empty when \( c_\star<c^* \). In this case, it holds that
\[
\boldsymbol V(x) = (\boldsymbol {MD})(x) \mathsf{T}_1 \left(\begin{matrix} 1 & \rho_1(x) \\ 0 & 1 \end{matrix}\right) (\boldsymbol{MD})^{-1}(x) - \boldsymbol I = \rho_1(x)\frac{\Phi_{\n}^{(0)}(x)}{\Phi_{\n}^{(1)}(x)}\boldsymbol M(x)\boldsymbol E_{1,2}\boldsymbol M^{-1}(x).
\]
Estimate \eqref{rhz-2} now follows from Lemma~\ref{lem:Mbounds} and the estimate
\begin{equation}
\label{rhz-0}
\left| \Phi_{\n}^{(0)}(x)/\Phi_{\n}^{(1)}(x)\right| = \exp\left\{|\n|\left(H_\n^{(0)}(x)-H_\n^{(1)}(x)\right)\right\} \leq \exp\left\{-B_{\beta_1}\delta^{3/2}n_1\right\} \leq \frac{\varepsilon_\n}{B_{\beta_1}\delta^{3/2}},
\end{equation}
see \eqref{PhiInt} and \eqref{4.5.1}. Lastly, let \( s\in \Gamma_{\n,1}^\pm\setminus U_\n \). Then it holds that
\[
\boldsymbol V(s) = (\boldsymbol {MD})(s) \mathsf{T}_1 \left(\begin{matrix} 1 & 0 \\ 1/\rho_1(s) & 1 \end{matrix}\right) (\boldsymbol{MD})^{-1}(s) - \boldsymbol I = \frac1{\rho_1(s)}\frac{\Phi_\n^{(1)}(s)}{\Phi_\n^{(0)}(s)}\boldsymbol M(s)\boldsymbol E_{2,1}\boldsymbol M^{-1}(s).
\]
The desired estimate \eqref{rhz-2} can be deduced exactly as in the second step of the proof with \eqref{4.5.2} used instead of \eqref{4.5.1}.
\end{proof}

It is essentially a standard argument in the theory of orthogonal polynomials to deduce existence of  \( \boldsymbol Z(z) \) from Lemma~\ref{lem:rhzV}, see \cite[Chapter~7]{Deift}.

\begin{lemma}
\label{lem:rhz} Given \( \mathcal N_{c_\star} \), \(
c_\star\in[0,1] \), there exists a constant \( M(\mathcal
N_{c_\star}) \) such that a solution of \hyperref[rhz]{\rhz} exists
for all $|\n|\geq M(\mathcal N_{c_\star})$ and it satisfies
\begin{equation}
\label{rhz-1}
\max_{i,j}\left|\big[\boldsymbol{Z}(z) - \boldsymbol{I}\big]_{i,j} \right| \lesssim \delta^{-1}\|\boldsymbol V\|_{L^\infty(\Sigma_{\n,\delta})}
\end{equation}
for all \( z\in \overline\C \) when \( c_\star\in[c^*,c^{**}] \), \( |z-\beta_1|\geq\delta/5 \) when \( c_\star\in(0,c^*) \), \( \dist(z,\{\alpha_1,\beta_1\})\geq\delta/5 \) when \( c_\star=0 \), \( |z-\alpha_2|\geq\delta/5 \) when \( c_\star\in(c^{**},1) \), and \( \dist(z,\{\alpha_2,\beta_2\})\geq\delta/5 \) when \( c_\star=1 \), where the constant in \( \lesssim \) is independent of \( \delta \) and \( \n \).
\end{lemma}
\begin{proof}
Let \( \mathcal C \) and \( \mathcal C_- \) be the operators defined in Lemma~\ref{lem:cauchy} and \( \mathcal C_{\boldsymbol V}:L^r(\Sigma_{\n,\delta})\to L^r(\Sigma_{\n,\delta}) \), \( r>1 \), be an operator defined by \(  \mathcal C_{\boldsymbol V}\boldsymbol F := \mathcal C_-(\boldsymbol{FV}) \) for any \( 2\times 2 \) matrix function \( \boldsymbol F(s) \) in \( L^r(\Sigma_{\n,\delta}) \). Then it follows from Lemmas~\ref{lem:cauchy} and~\ref{lem:rhzV} that
\begin{equation}
\label{rhz-4}
\|\mathcal C_{\boldsymbol V}\|_r \leq C_r\|\boldsymbol V\|_{L^\infty(\Sigma_{\n,\delta})} = o(1).
\end{equation}
Let \( M(\mathcal N_{c_\star}) \) be such that the above norm is less than \( 1/2 \) for all \( \n\in\mathcal N_{c_\star} \), \( |\n|\geq M(\mathcal N_{c_\star}) \). Then the operator \( \mathcal I - \mathcal C_{\boldsymbol V} \) is invertible in \( L^r(\Sigma_{\n,\delta}) \) for all such \( \n \). Hence, one can readily verify that
\[
\boldsymbol Z(z) = \boldsymbol I + \mathcal C(\boldsymbol{UV})(z), \quad \boldsymbol U(s) := (\mathcal I - \mathcal C_{\boldsymbol V})^{-1}(\boldsymbol I)(s).
\]
The above formula and H\"older inequality immediately yield that
\begin{equation}
\label{rhz-5}
\max_{i,j}\left|\big[\boldsymbol{Z}(z) - \boldsymbol{I}\big]_{i,j} \right| \lesssim \frac{\|\boldsymbol{UV}\|_{L^r(\Sigma_{\n,\delta})}}{\dist(z,\Sigma_{\n,\delta})} \lesssim \delta^{-1} \|\boldsymbol V\|_{L^\infty(\Sigma_{\n,\delta})}
\end{equation}
for \( \dist(z,\Sigma_{\n,\delta})\geq \delta/5 \), where the constant in \( \lesssim \) is independent of \( \n \) and \( \delta \) (it involves the arclengths of \( \Sigma_{\n,\delta} \), but the latter are uniformly bounded above and below). 

It can be readily seen from \hyperref[rhz]{\rhz}(b) that \( \boldsymbol V(s) \) can be analytically continued off each connected component of \( \Sigma_{\n,\delta}^\circ \). Hence, solutions of \hyperref[rhz]{\rhz} for the same value of \( \n \) and different values of \( \delta \) are, in fact, analytic continuations of each other. Thus, using \eqref{rhz-5} together with  \eqref{rhz-5} where \( \delta \) is replaced by \( \delta/2 \), we get that \eqref{rhz-5} in fact holds for \( \dist(z,([\beta_{c_\star,1},\beta_1]\cup[\alpha_2,\alpha_{c_\star,2}])\setminus U_\n)\geq \delta/5\). The set \( ([\beta_{c_\star,1},\beta_1]\cup[\alpha_2,\alpha_{c_\star,2}])\setminus U_\n \) is not empty only when \( c_\star\in[0,c^*)\cup(c^{**},1] \). In particular, we have finished the proof of the lemma for \( c_\star\in[c^*,c^{**}] \). When \( c_\star\in(0,c^*) \), set \( I_{\n,\delta}:=[\beta_1+\ic\delta c_\star/3,\beta_1]\cup(\beta_{c_\star,1}+\ic\delta c_\star/3,\beta_1+\ic\delta c_\star/3)\setminus \overline U_\n \) and let \( O_{\n,\delta} \) be the bounded domain delimited by \( \partial U_\n \), \( I_{\n,\delta} \), and \( [\beta_{c_\star,1},\beta_1)\setminus \overline U_\n \). Observer that \( \boldsymbol V(s) \) extends as an analytic matrix function into \( O_{\n,\delta} \) and still satisfies \eqref{rhz-0} there by \eqref{4.5.1}. Thus, we can analytically continue \( \boldsymbol Z(s) \) into \( O_{\n,\delta} \) by multiplying it by \( \boldsymbol I+ \boldsymbol V(z) \)  there. This continuation will still have a jump matrix satisfying \eqref{rhz-2} and therefore itself will satisfy \eqref{rhz-5} away from its jump contour. This finishes the proof of the lemma when \( c_\star\in(0,c^*)\cup(c^{**},1) \) (the proof for the case \( c_\star\in(c^{**},1) \) is identical). The proof in the case \( c_\star=0 \) (and therefore in the case \( c_\star=1 \)) is similar and uses \eqref{4.5.1a} instead of \eqref{4.5.1}.

The fact that the above constructed matrix \( \boldsymbol Z(z) \) has behavior as described in \hyperref[rhz]{\rhz}(c) follows from the fact that it admits an explicit local parametrix around \( \beta_1 \) (resp. \(\alpha_2\)) when \(c_\star<c^* \) (resp. \(c_\star>c^{**}\)), see \cite[Sections~8.3 and~9.1]{Y16}.
\end{proof}

The following lemma immediately follows from Lemma~\ref{lem:rhz}.

\begin{lemma}
\label{lem:final} A solution of \hyperref[rhx]{\rhx} is given by
\begin{equation}
\label{X}
 \boldsymbol X(z) := \boldsymbol C\boldsymbol Z(z) \left\{
 \begin{array}{rl}
(\boldsymbol{MD})(z), & z\in \overline\C\setminus\overline U_\n, \medskip \\
\boldsymbol P_e(z), & z\in U_e, \;\; e\in\{\alpha_1,\beta_{\n,1},\alpha_{\n,2},\beta_2\},
\end{array}
 \right.
 \end{equation}
 where \( \boldsymbol Z(z) \) solves \hyperref[rhz]{\rhz}, \( \boldsymbol N(z) := \boldsymbol C(\boldsymbol{MD})(z) \) solves \hyperref[rhn]{\rhn}, see \eqref{matrix-M}--\eqref{matrix-CD}, and \( \boldsymbol P_e(z) \) solve \hyperref[rhp]{\rhp$_e$}, see Section~\ref{ss:LP}.
\end{lemma}

\subsection{Proof of Theorems~\ref{thm:asymp2}--\ref{thm:asymp4}} We are now ready to prove the main results of Section~\ref{sec:3}. We stop using the notation \( c_\star \) and resume writing \( c \) as in the statements of Theorems~\ref{thm:asymp2}--\ref{thm:asymp4}.

\subsubsection{Proof of Theorem~\ref{thm:asymp2}}

Let \( K \) be a closed subset of \( \overline\C\setminus(
\Delta_{c,1}\cup\Delta_{c,2}) \).  It follows from
Proposition~\ref{prop:1} that the constant \( \delta \) in the
definition of the contour \( \Sigma_{\n,\delta} \) can be adjusted so that
\( K \) lies outside of each \( \overline\Omega_{\n,i}^\pm \) as
well as \( \overline U_\n \) for all \( |\n| \) large enough. Then
it holds that
\begin{equation}
\label{Y1}
\boldsymbol Y(z) = \boldsymbol C (\boldsymbol{ZMD})(z), \quad z\in K,
\end{equation}
by \eqref{eq:x} and  Lemma~\ref{lem:final}, where we need to write
\( \boldsymbol Y_\pm(z) \) and \( \boldsymbol Z_\pm(z) \) for \(
z\in \Delta_i \setminus \Delta_{c,i} \), \( i\in\{1,2\} \). Set
\begin{equation}
\label{defB}
B_k(z) := [\boldsymbol{Z}(z)]_{1,k+1} - \delta_{0k} = o(1) , \quad k\in\{0,1,2\},
\end{equation}
where  \( \delta_{ij} \) is the usual Kronecker symbol. Observe that \( B_k(\infty) =0 \) and
\begin{equation}
\label{Best}
|B_k(z)| = \left\{ \begin{array}{ll} \mathcal O_{\delta,c}(\varepsilon_\n), & c\not\in\{c^*,c^{**}\}, \smallskip \\ o_\delta(1), & c\in\{c^*,c^{**}\}, \end{array} \right.
\end{equation}
uniformly in \( \overline\C\setminus\{\alpha_1,\beta_1\} \) when \( c=0 \), in \( \overline\C\setminus\{\beta_1\} \) when \( c\in(0,c^*) \), in \( \overline\C \) when \( c\in[c^*,c^{**}] \), in \( \overline\C\setminus\{\alpha_2\} \) when \( c\in(c^{**},1) \), and in \( \overline\C\setminus\{\alpha_2,\beta_2\} \) when \( c=1 \) by \eqref{rhz-2} and \eqref{rhz-1}, where the dependence on \( c \) of \( \mathcal O_{\delta,c}(\varepsilon_\n) \) is uniform on compact subsets of \( [0,c^*)\cup(c^{**},1]\). Then it follows from \eqref{eq:y},
\eqref{Y1}, the definition of \( \boldsymbol M(z) \) in
\eqref{matrix-M}, and of \( \boldsymbol C \), \( \boldsymbol D(z)
\) in \eqref{matrix-CD} that
\begin{eqnarray*}
P_\n(z) &=& [\boldsymbol Y(z)]_{1,1} = [\boldsymbol C]_{1,1}[(\boldsymbol{ZM})(z)]_{1,1}[\boldsymbol D(z)]_{1,1} \\
& = & \gamma_\n S_\n^{(0)}(z) \left( 1+ B_0(z) + s_{\n,1}B_1(z)\Upsilon_{\n,1}^{(0)}(z) + s_{\n,2}B_2(z)\Upsilon_{\n,2}^{(0)}(z) \right) \Phi_\n^{(0)}(z),
\end{eqnarray*}
where \( s_{\n,i}:=S_\n^{(0)}(\infty)/S_\n^{(i)}(\infty) \), \( i\in\{1,2\} \). The first asymptotic formula of the theorem now follows from \eqref{Best}, \eqref{lem51-0}--\eqref{lem51-3}, and \eqref{szego-limit3}.

Let now \( K \) be a closed subset of \(
\Delta_{c,1}^\circ\cup\Delta_{c,2}^\circ \). Again, we can adjust
\( \delta \) so that \( K \) does not intersect \( \overline U_\n
\) for all \( |\n| \) large enough. Hence,
\begin{equation}
\label{Y2}
\boldsymbol Y_{\pm}(x) = \boldsymbol C (\boldsymbol{ZM}_{\pm}\boldsymbol D_{\pm})(x)(\boldsymbol I \pm \rho_i^{-1}(x) \boldsymbol E_{i+1,1} ), \quad x\in K\cap\Delta_{c,i},
\end{equation}
for \( i\in\{1,2\} \), again by \eqref{eq:x} and Lemma~\ref{lem:final}. Thus, we get for \( x\in K \cap \Delta_{c,i} \) that
\begin{multline*}
P_\n(x) = \gamma_\n\big(S_\n\Phi_\n\big)_\pm^{(0)}(x)\left(1+B_0(x)
+ B_1(x)\Upsilon_{\n,1\pm}^{(0)}(x) +
B_2(x)\Upsilon_{\n,2\pm}^{(0)}(x)\right) \\ \pm
\gamma_\n(\rho_iw_{\n,i\pm})^{-1}(x)\big(S_\n\Phi_\n\big)_\pm^{(i)}(x)\left(1+B_0(x)
+ B_1(x)\Upsilon_{\n,1\pm}^{(i)}(x) +
B_2(x)\Upsilon_{\n,2\pm}^{(i)}(x)\right).
\end{multline*}
Since \( F^{(0)}_\pm(x) = F^{(i)}_\mp(x) \) on \( \Delta_{\n,i} \)
for any rational function \( F(\z) \) on \( \RS_\n \), the
second asymptotic formula of the theorem now follows from
\eqref{szego-pts2}, \eqref{Best}, and
\eqref{lem51-0}--\eqref{lem51-3}.

\subsubsection{Proof of Theorem~\ref{thm:asymp3}}

Similarly to the matrix \( \boldsymbol Y(z) \) defined in
\eqref{eq:y}, set
\begin{equation}
\label{hatY}
\widehat{\boldsymbol Y}(z) := \left(\begin{matrix} L_\n(z) & -A_\n^{(1)}(z) & -A_\n^{(2)}(z) \medskip \\ -d_{\n,1} L_{\n+\vec e_1}(z) &  d_{\n,1} A_{\n+\vec e_1}^{(1)}(z) & d_{\n,1}A_{\n+\vec e_1}^{(2)}(z) \medskip \\ -d_{\n,2} L_{\n+\vec e_2}(z) & d_{\n,2} A_{\n+\vec e_2}^{(1)}(z) & d_{\n,2}A_{\n+\vec e_2}^{(2)}(z) \end{matrix}\right),
\end{equation}
where the constants \( d_{\n,i} \) are chosen so that the polynomials
\( d_{\n,i}A_{\n+\vec e_i}^{(i)}(z) \) are monic. It was shown in
\cite[Theorem~4.1]{GerKVA01} that
\begin{equation}
\label{hatY1}
\widehat{\boldsymbol Y}(z) = \big(\boldsymbol Y^\mathsf{T}(z)\big)^{-1}.
\end{equation}
Hence, it follows from \eqref{Y1} that on closed subsets of \(
\overline\C\setminus \big(\Delta_{c,1}\cup \Delta_{c,2}\big) \) it
holds that
\[
\widehat{\boldsymbol Y}(z) = \boldsymbol C^{-1}\big(\boldsymbol Z^{-1}\big)^{\mathsf T}(z)\big(\boldsymbol M^{-1}\big)^\mathsf{T}(z)\boldsymbol D^{-1}(z)
\]
(as before, the contour \( \Sigma_{\n,\delta} \) can be adjusted to
accommodate any such closed set, moreover, one needs to write \(
\widehat{\boldsymbol Y}_\pm(z) \) for \( z\in
\Delta_i\setminus\Delta_{c,i} \)). The above equation and
\eqref{hatY} yield that
\begin{equation}
\label{Y3}
A_\n^{(i)}(z) = -\big[\boldsymbol C^{-1}\big(\boldsymbol Z^{-1}\big)^{\mathsf T}(z)\big(\boldsymbol M^{-1}\big)^\mathsf{T}(z)\boldsymbol D^{-1}(z)\big]_{1,i+1}, \quad z\in K.
\end{equation}
Let us rewrite \eqref{M-inverse} as
\[
\boldsymbol M^{-1}(z) =: \diag\left(\frac1{S_\n^{(0)}(z)},\frac{w_{\n,1}(z)}{S_\n^{(1)}(z)},\frac{w_{\n,2}(z)}{S_\n^{(2)}(z)}\right)\boldsymbol \Pi(z) \boldsymbol S(\infty),
\]
which serves as a definition of the matrix \( \boldsymbol\Pi(z) \).
Notice that \( \tau_\n \), defined in the statement of the theorem,
is equal to \( [\boldsymbol C]_{1,1} \). Thus, it follows from
\eqref{Y3} that
\begin{equation}
\label{Y4}
A_\n^{(i)}(z) = -\big[\big(\boldsymbol Z^{-1}\big)^{\mathsf T}(z)\boldsymbol S(\infty)\boldsymbol \Pi^\mathsf{T}(z)\big]_{1,i+1}\frac{w_{\n,i}(z)}{\tau_\n\big(S_\n\Phi_\n\big)^{(i)}(z)}, \quad z\in K.
\end{equation}
Similarly to \eqref{defB}, set
\[
\widehat B_k(z) := \left[\big(\boldsymbol Z^{-1}\big)^\mathsf{T}(z)\right]_{1,k+1} - \delta_{0k} = o(1) , \quad k\in\{0,1,2\}.
\]
Observe that all the jump matrices in \hyperref[rhz]{\rhz}(b) have determinant one. Since \( \boldsymbol Z(\infty)=\boldsymbol I \), we therefore get that \( \det(\boldsymbol Z(z))\equiv 1 \). Hence, the functions \( \widehat B_k(z) \) do obey the estimate of \eqref{Best} as well. Again, it holds that \( \widehat B_k(\infty)=0 \). Thus,
\begin{multline*}
\big[\big(\boldsymbol Z^{-1}\big)^{\mathsf T}(z)\boldsymbol
S(\infty)\boldsymbol \Pi^\mathsf{T}(z)\big]_{1,i+1} =
S_\n^{(0)}(\infty)\left(\Pi_\n^{(i)}(z) + \widehat
B_0(z)\Pi_\n^{(i)}(z) \right. \\ + s_{\n,1}^{-1}\widehat
B_1(z)\Pi_{\n,1}^{(i)}(z) \left. + s_{\n,2}^{-1}\widehat
B_2(z)\Pi_{\n,2}^{(i)}(z) \right), \quad z\in K,
\end{multline*}
where, as before, \( s_{\n,l} =
S_\n^{(0)}(\infty)/S_\n^{(l)}(\infty) \). Now, observe that
\[
\Pi_{\n,l}(\z) / \Pi_\n(\z) = -A_{\n,l}^{-1}\Upsilon_{\n,l}(\z), \quad l\in\{1,2\},
\]
which follows from comparing zero/pole divisors and the
normalizations at \( \infty^{(0)} \) of the left- and right-hand
sides of the above equality (recall that \( \Pi_\n^{(0)}(\infty)=1
\) and \( \Pi_{\n,l}^{(0)}(z) = -z^{-1} + \mathcal O(z^{-2}) \),
which can be seen from \eqref{Up-Pi1}). Therefore, it follows from
\eqref{Y4} that
\begin{equation}
\label{Y5}
A_\n^{(i)}(z) = - \left(1 + \widehat B_0(z) - \frac{\Upsilon_{\n,1}^{(i)}(z)}{s_{\n,1}A_{\n,1}} \widehat B_1(z) - \frac{\Upsilon_{\n,2}^{(i)}(z)}{s_{\n,2}A_{\n,2}} \widehat B_2(z) \right)\frac{\big(\Pi_\n^{(i)}w_{\n,i}\big)(z)}{\gamma_\n\big(S_\n\Phi_\n\big)^{(i)}(z)}.
\end{equation}
Hence, the first asymptotic formula of the theorem follows from
\eqref{Best}, \eqref{lem51-0}--\eqref{lem51-3} (here, one
needs to recall that \( \widehat B_l(\infty)=0 \) and therefore the
estimate for \( (\Upsilon_{\n,l}^{(l)}\widehat B_l)(z) \) around
infinity follows from the maximum principle), \eqref{szego-limit3}, and the fact that \( A_{\n,1}\sim c_\n^2 \) shown in the proof of Lemma~\ref{lem:aux1}. When \( c=0 \) and \(
i=1 \), we also deduce from \eqref{Y5} and the maximum modulus principle
that
\[
A_\n^{(1)}(z) = \frac{o(1)}{c_\n^2}\frac{S_\n^{(1)}(\infty)}{S_\n^{(1)}(z)} \frac{\big(\Pi_\n^{(1)}w_{\n,1}\big)(z)} {\tau_\n\Phi_\n^{(1)}(z)} = \frac{o(1)}{c_\n^2} \frac{\big(\Pi_\n^{(1)}w_{\n,1}\big)(z)}{\tau_\n\Phi_\n^{(1)}(z)},
\]
where we also used \eqref{szego-limit2} and \( o(1) \) behaves like the right-hand side of \eqref{Best}. Recall that \(
\Pi_\n^{(1)}(z) \) has a double zero at infinity. Therefore,
\[
\big| \big(\Pi_\n^{(1)}w_{\n,1}^2\big)(z) \big| = \left|\left(\Upsilon_{\n,2}^{(0)}\Upsilon_{\n,1}^{(2)}-\Upsilon_{\n,2}^{(2)}\Upsilon_{\n,1}^{(0)}\right)(z)  \frac{w_{\n,1}(z)}{w_{\n,2}(z)}\right| = \mathcal O\big( c_\n^2 \big)
\]
uniformly on closed subsets \( \C\setminus\Delta_{0,1} \) by
\eqref{Up-Pi2}, \eqref{lem51-1}--\eqref{lem51-3}, and the maximum
modulus principle. Clearly, the last two estimates prove the second
asymptotic formula of the theorem (the case \( c=1 \) and \( i=2 \)
can be treated similarly).

Finally, \eqref{Y2} and \eqref{hatY1} give us
\[
\widehat{\boldsymbol Y}_{\pm}(x) = \boldsymbol C^{-1}\big(\boldsymbol Z^{-1}\big)^{\mathsf T}(x)\big(\boldsymbol M_{\pm}^{-1}\big)^\mathsf{T}(x)\boldsymbol D_{\pm}^{-1}(x)\big(\boldsymbol I\mp \rho_i^{-1}(x)\boldsymbol E_{1,i+1}\big)
\]
on any compact subset of \( \Delta_{c,i}^\circ \), \( i\in\{1,2\}
\). Analogously to \eqref{Y5}, the above formula yields that
\begin{multline*}
A_\n^{(i)}(x) = - \left(1 + \widehat B_0(x) -
\frac{\Upsilon_{\n,1\pm}^{(i)}(x)}{s_{\n,1}A_{\n,1}} \widehat
B_1(x) - \frac{\Upsilon_{\n,2\pm}^{(i)}(x)}{s_{\n,2}A_{\n,2}}
\widehat B_2(x) \right)
\frac{\big(\Pi_\n^{(i)}w_{\n,i}\big)_\pm(x)}{\gamma_\n\big(S_\n\Phi_\n\big)_\pm^{(i)}(x)}
\\ \pm \rho_i^{-1}(x) \left(1 + \widehat B_0(x) -
\frac{\Upsilon_{\n,1\pm}^{(0)}(x)}{s_{\n,1}A_{\n,1}} \widehat
B_1(x) - \frac{\Upsilon_{\n,2\pm}^{(0)}(x)}{s_{\n,2}A_{\n,2}}
\widehat B_2(x) \right)
\frac{\Pi_{\n\pm}^{(0)}(x)}{\gamma_\n\big(S_\n\Phi_\n\big)_\pm^{(0)}(x)}.
\end{multline*}
Once again, \eqref{Best} and \eqref{lem51-0}--\eqref{lem51-3}
imply that
\[
A_\n^{(i)}(x) = -(1+o(1))\frac{\big(\Pi_\n^{(i)}w_{\n,i}\big)_\pm(x)}{\gamma_\n\big(S_\n\Phi_\n\big)_\pm^{(i)}(x)}  \pm (1+o(1))\rho_i^{-1}(x)\frac{\Pi_{\n\pm}^{(0)}(x)}{\gamma_\n\big(S_\n\Phi_\n\big)_\pm^{(0)}(x)}
\]
uniformly on compact subsets of \( \Delta_{c,i}^\circ \). Since
\[
\mp\rho_i^{-1}(x)\Pi_{\n\pm}^{(0)}(x)/(S_\n\Phi_\n)_\pm^{(0)}(x) = \big(\Pi_{\n\mp}^{(i)}w_{\n,i\mp}\big)(x)/(S_\n\Phi_\n)_\mp^{(i)}(x), \quad x\in\Delta_{\n,i},
\]
by \eqref{szego-pts2}, the last asymptotic formula of the theorem
follows.

\subsubsection{Proof of Theorem~\ref{thm:asymp4}}

As in the previous two subsections, given a closed set \( K \) in \(
\overline\C\setminus (\Delta_1\cup\Delta_2) \), we can adjust the
contour \( \Sigma_{\n,\delta} \) so that \( K \) lies in the unbounded
component of its complement. Hence, using the
notation of the previous two subsections, we get from \eqref{eq:y},
\eqref{matrix-M}, \eqref{matrix-CD}, \eqref{Y1}, and \eqref{Best}
that
\[
R_\n^{(i)}(z) = \gamma_\n S_\n^{(i)}(z)w_{\n,i}^{-1}(z) \left( 1+ B_0(z) + s_{\n,1}B_1(z)\Upsilon_{\n,1}^{(i)}(z) + s_{\n,2}B_2(z)\Upsilon_{\n,2}^{(i)}(z) \right) \Phi_\n^{(i)}(z)
\]
for \( z\in K \), \( i\in\{1,2\} \). The first asymptotic formula
of the theorem now follows from \eqref{Best},
\eqref{lem51-0}--\eqref{lem51-3}, \eqref{szego-limit3}, and the
maximum modulus principle applied to \(
(\Upsilon_{\n,i}^{(i)}B_i)(z) \) to extend the desired estimates to
the neighborhood of infinity. As in the proof of Theorem
\ref{thm:asymp3}, it holds when \( c=0 \) and \( i=1 \) that
\[
R_\n^{(1)}(z) = o(1) \tau_\n\Phi_\n^{(1)}(z)w_{\n,1}^{-1}(z)
\]
uniformly on closed subsets of \( \overline\C\setminus\Delta_{0,1}
\) by \eqref{lem51-1}--\eqref{lem51-3} and
\eqref{szego-limit2}--\eqref{szego-limit3}. Since an analogous
formula holds for \( c=1 \) and \( i=2 \), the second asymptotic
formula of the theorem follows.

Finally, it follows from \eqref{hatY} and \eqref{hatY1} that
\[
L_\n(z) = \left(1 + \widehat B_0(z) - \frac{\Upsilon_{\n,1}^{(0)}(z)}{s_{\n,1}A_{\n,1}} \widehat B_1(z) - \frac{\Upsilon_{\n,2}^{(0)}(z)}{s_{\n,2}A_{\n,2}} \widehat B_2(z) \right)\frac{\Pi_\n^{(0)}(z)}{\gamma_\n\big(S_\n\Phi_\n\big)^{(0)}(z)}
\]
on closed subsets of \(
\overline\C\setminus(\Delta_{c,1}\cup\Delta_{c,2}) \), from which
the last asymptotic formula of the theorem follows, as usual, by
\eqref{Best} (holding for \( \widehat B_k(z) \) as well), \eqref{lem51-0}--\eqref{lem51-3}, \eqref{szego-limit3}, and since \( A_{\n,1}\sim c_\n^2 \) as shown in Lemma~\ref{lem:aux1}.

\section{Proof of Theorem~\ref{thm:recurrence}}
\label{sec:9}

While proving Theorem~\ref{thm:recurrence} we first consider the
case of fully marginal sequences and then consider separately the
asymptotic behavior of \( a_{\n,1},a_{\n,2} \) and  \(
b_{\n,1},b_{\n,2} \).

\subsection{Fully Marginal Ray Sequences}

In this section we only consider sequences \( \mathcal N_0 \) and
\( \mathcal N_1 \) satisfying \eqref{fullymarginal}. Again, we
present the proof only in the case of \( c=0 \). Recurrence formula
\eqref{recurrence} for \( P_\n(x) \) can be rewritten as
\begin{equation}
\label{8.1}
z - b_{\n,i} = \frac{P_{\n+\vec e_i}(z)}{P_\n(z)} + a_{\n,1}\frac{P_{\n-\vec e_1}(z)}{P_\n(z)} + a_{\n,2}\frac{P_{\n-\vec e_2}(z)}{P_\n(x)}, \quad i\in\{1,2\}.
\end{equation}
One can easily see from \eqref{8.1} that
\begin{equation}
\label{8.2}
b_{\n,i} = -\lim_{z\to\infty} \left(  \frac{P_{\n+\vec e_i}(z)}{P_\n(z)} - z \right).
\end{equation}
Thus, the limiting behavior of \( b_{\n,1},b_{\n,2} \) follows from
Theorem~\ref{thm:asymp1} and \eqref{6.0} in Lemma~\ref{lem:6-1}.
Moreover, since the rays \( \big\{\n\pm\vec e_i:\n\in\mathcal
N_0\big\} \) are also fully marginal, we can use
Theorem~\ref{thm:asymp1} to rewrite \eqref{8.1} for \( i=2 \) as
\begin{equation}
\label{8.3}
z - b_{\n,2} = (1+o(1))\varphi_2(z) + (1+o(1))\frac{a_{\n,1}}{S(z;\alpha_1)(z-\alpha_1)} + (1+o(1))\frac{a_{\n,2}}{\varphi_2(z)}.
\end{equation}
Recall that \( S(z;\alpha_1) = 1 - (B_{0,1}  - \alpha_1)/z + \mathcal O\big(z^{-2}\big) \) by \eqref{6.7} and \eqref{AngPar2}. Hence,
if we use \eqref{varphii} to obtain the first four terms of the
power series expansion of  \( \varphi_2(z) \) at infinity, we then
can rewrite \eqref{8.3} as
\begin{multline}
\label{8.4} z - b_{\n,2} = (1+o(1))\left(z - B_{0,2} - \frac{A_{0,2}}z - \frac{A_{0,2}B_{0,2}}{z^2} + \mathcal O\left(\frac1{z^3}\right)\right) + \\ +  \frac{a_{\n,1}}z\left(1 + \frac{B_{0,1}}z + \mathcal O\left(\frac1{z^2}\right)\right) + \frac{a_{\n,2}}z\left(1 + \frac{B_{0,2}}z + \mathcal O\left(\frac1{z^2}\right)\right).
\end{multline}
It follows immediately from \eqref{8.4} that
\[
a_{\n,1} + a_{\n,2} = (1+o(1))A_{0,2} \qandq B_{0,1}a_{\n,1} + B_{0,2}a_{\n,2} =  (1+o(1))B_{0,2}A_{0,2},
\]
from which the limits of \( a_{\n,1},a_{\n,2} \) easily follow
(recall that \( \varphi_2(z) \) is non-vanishing).

\subsection{Asymptotics of \( a_{\n,1},a_{\n,2} \) along Non-fully Marginal Sequences}

From now on we are assuming that ray sequences \( \mathcal N_c \) satisfy
\eqref{vareps}. It can be deduced from orthogonality relations
\eqref{typeII} and definition \eqref{Rni} that
\[
R_\n^{(i)}(z) = -\frac{h_{\n,i}}{2\pi\mathrm i}\frac1{z^{n_i+1}} + \mathcal O\big(z^{-n_i-2}\big), \quad h_{\n,i} := \int P_\n(x)x^{n_i}\dd\mu_i(x),
\]
\( i\in\{1,2\} \). In particular, we have that
$m_{\n,i}=-2\pi\ic/h_{\n-\vec e_i,i}$ in \eqref{eq:y}. Then it follows from the first and second asymptotic formulae of Theorem~\ref{thm:asymp4}, the definition of constants \( \gamma_\n \) and \( \tau_\n \) in Theorems~\ref{thm:asymp2} and~\ref{thm:asymp3}, respectively, and the definition of the matrix \( \boldsymbol C \) in \eqref{matrix-CD} that
\begin{equation}
\label{8.5}
-\frac{h_{\n,i}}{2\pi\ic} =\frac{1+o(1)}{s_{\n,i}}\frac{[\boldsymbol C]_{1,1}}{[\boldsymbol C]_{i+1,i+1}} \qorq -\frac{h_{\n,i}}{2\pi\mathrm i} = o(1)\frac{[\boldsymbol C]_{1,1}}{[\boldsymbol C]_{i+1,i+1}}
\end{equation}
where, as before, \(
s_{\n,i}=S_\n^{(0)}(\infty)/S_\n^{(i)}(\infty) \), \( i\in\{1,2\}
\), the first formula holds for \( i\in\{1,2\} \) when \( c\in(0,1)
\), \( i=2 \) when \( c=0 \), and \( i=1 \) when \( c=1 \), and the
second formula holds for the remaining cases. Furthermore, we get
from \eqref{eq:y} that
\begin{equation}
\label{8.6}
-\frac{2\pi\ic}{h_{\n-\vec e_i,i}} = m_{\n,i} = \lim_{z\to\infty} z^{1-|\n|}[\boldsymbol Y(z)]_{i+1,1}.
\end{equation}
Analogously to the computation after
\eqref{Y1}--\eqref{Best} we get that \( [\boldsymbol Y(z)]_{i+1,1}
\) is equal to
\begin{equation}
\label{8.7}
[\boldsymbol C]_{i+1,i+1} \frac{S_\n^{(0)}(z)}{S_\n^{(0)}(\infty)}  \left( s_{\n,i}\Upsilon_{n,i}^{(0)}(z) + B_{0,i}(z) + s_{\n,1}B_{1,i}(z)\Upsilon_{\n,1}^{(0)}(z) + s_{\n,2}B_{2,i}(z)\Upsilon_{\n,2}^{(0)}(z) \right) \Phi_\n^{(0)}(z)
\end{equation}
in a neighborhood of infinity, where \( B_{k,i}(z) := [\boldsymbol Z(z)]_{i+1,k+1} - \delta_{ik} \), \( k\in\{0,1,2\} \), satisfy \eqref{Best}. Since \( B_{k,i}(\infty)=0 \) and \(
\Upsilon_{n,i}^{(0)}(z) = A_{\n,i}z^{-1} + \mathcal
O\big(z^{-2}\big) \) as \(z\to\infty\), see \eqref{Upsilon}, we get that
\begin{equation}
\label{8.8}
-\frac{2\pi\ic}{h_{\n-\vec e_i,i}} = \big( s_{\n,i} A_{\n,i} + o(1) \big) \frac{[\boldsymbol C]_{i+1,i+1}}{[\boldsymbol C]_{1,1}}.
\end{equation}
Now, it is well known, see for example \cite[Lemma~A.1]{uApDenY},
that \( a_{\n,i} = h_{\n,i}/ h_{\n-\vec e_i,i} \). Therefore, it
follows from \eqref{8.5} and \eqref{8.8} that
\[
a_{\n,i}  = (1+o(1)) \big( A_{\n,i} +s_{\n,i}^{-1} o(1) \big) \qorq a_{\n,i}  = o(1)\big( s_{\n,i}A_{\n,i} + o(1) \big)
\]
\( i\in\{1,2\} \), where the first formula holds for \( i\in\{1,2\}
\) when \( c\in(0,1) \), \( i=2 \) when \( c=0 \), and \( i=1 \)
when \( c=1 \), and the second formula holds for the remaining
cases. The desired limits of \( a_{\n,i} \) therefore follow from
continuity of the constants \( A_{c,i} \) with respect to the
parameter \( c \), see Proposition~\ref{prop:angelesco}, asymptotic
formulae \eqref{szego-limit3}, and the estimates \( A_{c,1}\sim c^2
\) as \( c\to 0 \) (\( A_{c,2}\sim (1-c)^2 \) as \( c\to 1 \)), see
\eqref{lem51-6} and after.

\subsection{Asymptotics of \( b_{\n,1},b_{\n,2}\) along Non-fully Marginal Sequences}

Excluding the cases \( i=1 \) when \( c=0 \) and \( i=2 \) when \(
c=1 \), we get from \eqref{8.6}--\eqref{8.8} and
\eqref{lem51-1}--\eqref{lem51-3} that
\begin{equation}
\label{8.9}
P_{\n-\vec e_i}(z) = (1 + o(1)) A_{\n,i}^{-1}\Upsilon_{\n,i}^{(0)}(z)\gamma_\n\big(S_\n\Phi_\n)^{(0)}(z)
\end{equation}
in some neighborhood of the point at infinity. Replacing the
sequence \( \mathcal N_c \) with \( \{\n+\vec e_i:\n\in \mathcal
N_c\} \), we get from \eqref{8.2}, Theorem~\ref{thm:asymp2}, and
\eqref{8.9} that
\[
b_{\n,i} =  -(1 + o(1)) \lim_{z\to\infty}\left(\frac{A_{\n+\vec e_i}}{\Upsilon^{(0)}_{\n+\vec e_i,i}(z)}-z\right) = (1 + o(1)) B_{\n+\vec e_i},
\]
where we also used \eqref{AngPar1} and \eqref{Upsilon}. The desired claim now follows from
Proposition~\ref{prop:angelesco}.

Out of the two exceptional cases, we shall only consider the case
\( i = 1 \) when \( c=0 \) understanding that the other one can be
treated similarly. Assume for the moment that the measure \( \mu_2
\) is, in fact, the arcsine distribution on \( \Delta_2 \), that
is,
\begin{equation}
\label{8.10}
\dd\mu_2(x) = \frac{\dd x}{2\pi\sqrt{(x-\alpha_2)(\beta_2-x)}} = -\frac{\dd x}{2\pi\ic w_{2+}(x)}.
\end{equation}
Recall the notation of Section~\ref{sec:6} where we wrote \(
P_\n(z) = P_{\n,1}(z)P_{\n,2}(z) \) with polynomial \( P_{\n,i}(z)
\) having all its zeros on \( \Delta_i \). We would like to show
that when \( \mu_2 \) is of the form \eqref{8.10}, formula
\eqref{6.1} still holds along any marginal ray sequence \( \mathcal N_0
\). To this end, we shall use \( 2\times 2 \) Riemann-Hilbert
analysis of orthogonal polynomials. Since this method has been
described in detail in Section~\ref{sec:8}, we shall only outline
the main steps.

It follows from \eqref{typeII} and \eqref{8.10} that the
Riemann-Hilbert problem
\begin{itemize}
\item[(a)] $\boldsymbol Y(z)$ is analytic in
    $\C\setminus\Delta_2$ and $\displaystyle \lim_{z\to\infty}
    {\boldsymbol Y}(z)z^{-n_2} = {\boldsymbol I}$;
\item[(b)] $\boldsymbol Y(z)$ has continuous traces on each
    $\Delta_2^\circ$ that satisfy \( {\boldsymbol Y}_+(x) =
    {\boldsymbol Y}_-(x)\left(\begin{matrix} 1 &
    (P_{\n,1}/w_{2+})(x) \\ 0 & 1 \end{matrix}\right) \);
\item[(c)] the entries of the first column of $\boldsymbol Y(z)$ are bounded and the entries of the second column behave like \( \mathcal O(|z-\xi|^{-1/2}) \) as $z\to\xi\in\{\alpha_2,\beta_2\}$;
\end{itemize}
is solved by
\[
\boldsymbol Y(z):= \left(\begin{matrix}
P_{\n,2}(z) & R_\n^{(2)}(z) \smallskip\\
m_{\n,2}^\star P_{\n,2}^\star(z) & m_{\n,2}^\star R_{\n,2}^\star(z)
\end{matrix}\right),
\]
where \( P_{\n,2}^\star(z) \) is the monic polynomial of degree \(
n_2-1 \) orthogonal to lower degree polynomials with respect to the
weight \( P_{\n,1}(x)\dd\mu_2(x) \) and
\[
R_{\n,2}^\star(z) = \frac1{2\pi\ic}\int \frac{P_{\n,2}^\star(x)P_{\n,1}(x)\dd\mu_2(x)}{x-z} = \frac1{m_{\n,2}^\star z^{n_2}} + \mathcal O\big(z^{-n_2-1}\big).
\]
Let \( \Gamma_2 \) be a Jordan curve encircling \( \Delta_2 \)
counter-clockwise and containing \( \Delta_1 \) in its exterior.
Set
\[
\boldsymbol X(z):= \boldsymbol Y(z) \left\{
\begin{array}{ll}
\left(\begin{matrix} 1 & 0 \\ -(w_2/P_{\n,1})(z) & 1 \end{matrix}\right) &  z\in\Omega_2, \medskip \\
\boldsymbol I & \text{otherwise},
\end{array}
\right.
\]
where \( \Omega_2 \) is the interior domain of \( \Gamma_2 \). Then
$\boldsymbol X(z)$ solves the following Riemann-Hilbert problem:
\begin{itemize}
\item[(a)] $\boldsymbol X(z)$ is analytic in
    $\C\setminus(\Delta_2\cup\Gamma_2)$ and $\displaystyle
    \lim_{z\to\infty} {\boldsymbol X}(z)z^{-n_2} = {\boldsymbol
    I}$;
\item[(b)] $\boldsymbol X(z)$ has continuous traces on
    $\Delta_2^\circ\cup\Gamma_2$ that satisfy
\[
{\boldsymbol X}_+(s)={\boldsymbol X}_-(s) \left\{
\begin{array}{rl}
\left(\begin{matrix} 0 & (P_{\n,1}/w_{2+})(s) \\ -(w_{2+}/P_{\n,1})(s) & 0 \end{matrix}\right), & s\in \Delta_2,  \medskip \\
\left(\begin{matrix} 1 & 0 \\ (w_2/P_{\n,1})(s) & 1 \end{matrix}\right), & s\in \Gamma_2;
\end{array}
\right.
\]
\item[(c)] the entries of the first column of $\boldsymbol Y(z)$ are bounded and the entries of the second column behave like \( \mathcal O(|z-\xi|^{-1/2}) \) as $z\to\xi\in\{\alpha_2,\beta_2\}$.
\end{itemize}
The solution of the above Riemann-Hilbert problem is given by \(
\boldsymbol X(z) = \boldsymbol C(\boldsymbol{ZL})(z) \), where
\[
\boldsymbol L(z) := \left(\begin{matrix} 1  & 1/w_2(z) \smallskip \\ 1/\widetilde\varphi_2(z) & \widetilde\varphi_2(z)/w_2(z) \end{matrix} \right) \big( S_\n\widetilde\varphi_2^{n_2}\big)^{\sigma_3}(z)
\]
with (compare to \eqref{Sx0} and observe that \(
\widetilde\varphi_{2+}(x) \widetilde\varphi_{2-}(x)\equiv 1 \) on
\( \Delta_2 \))
\[
\widetilde\varphi_2(z) := A_{0,2}^{-1/2}\varphi_2(z) \qandq S_\n(z) := \prod_{i=1}^{n_1} \left( \frac{\widetilde\varphi_2(z) - \widetilde\varphi_2(x_{\n,i})}{ \widetilde\varphi_2(z) \widetilde\varphi_2(x_{\n,i}) -1} \frac{\widetilde\varphi_2(z)}{z-x_{\n,i}}\right)^{1/2},
\]
\( \boldsymbol C \) is a diagonal matrix of constants such that \(
\lim_{z\to\infty} \boldsymbol{CL}(z)z^{-n_2\sigma_3} = \boldsymbol
I \), and \( \boldsymbol Z(z) \)  solves the following
Riemann-Hilbert problem:
\begin{itemize}
\item[(a)] $\boldsymbol Z(z)$ is a holomorphic matrix function
    in $\overline\C\setminus\Gamma_2$ and
    $\boldsymbol{Z}(\infty)=\boldsymbol{I}$;
\item[(b)] $\boldsymbol Z(z)$ has continuous traces on
    $\Gamma_2$ that satisfy \( \boldsymbol{Z}_+(s)  =
    \boldsymbol{Z}_-(s) \boldsymbol L(s) \left(\begin{matrix} 1
    & 0 \\ (w_2/P_{\n,1})(s) & 1 \end{matrix}\right)\boldsymbol
    L^{-1}(s) \).
\end{itemize}
Indeed, as in Section~\ref{sec:8}, we only need to verify that the
jump of \( \boldsymbol Z(z) \) on \( \Gamma_2 \) can be estimated
as \( \boldsymbol I + o(1) \) as \( n_2 \to\infty \), \(
\n\in\mathcal N_0 \). The latter is equal to
\[
\boldsymbol I + \frac1{(w_2P_{n,1}S_\n^2\widetilde\varphi_2^{2n_2})(s)}\left( \begin{matrix} \widetilde\varphi_2(s) & -1 \smallskip \\ \widetilde\varphi_2^2(s) & -\widetilde\varphi_2(s) \end{matrix} \right).
\]
Observe that
\[
(P_{n,1}S_\n^2)(s) = \varphi_2^{n_1}(s) \prod_{i=1}^{n_1} b(s;x_{\n,i}), \quad b(z;x_0) := \frac{\widetilde\varphi_2(z) - \widetilde\varphi_2(x_0)}{ \widetilde\varphi_2(z) \widetilde\varphi_2(x_0) -1}.
\]
Notice that \( \inf_{s\in\Gamma_2}|\widetilde\varphi_2(s)|>1 \) and
\(  \inf_{s\in\Gamma_2,x_0\in\Delta_1}|b(s;x_0)|>0 \) by the
compactness of \( \Delta_1 \) and \( \Gamma_2 \). Therefore, there
exist positive constants \( C_1>1 \) and \( C_2<1 \) such that
\[
\sup_{s\in\Gamma}|(w_2P_{n,1}S_\n^2\widetilde\varphi_2^{2n_2})(s)|^{-1} \leq C_1^{n_1}C_2^{2n_2+n_1} = (C_1^{n_1/(2n_2+n_1)}C_2)^{2n_2+n_1} = o(1)
\]
as \( n_1/n_2 \to 0 \). This finishes the proof of the identity \(
\boldsymbol X(z) = \boldsymbol C(\boldsymbol{ZL})(z) \) from which
\eqref{6.1} easily follows. Observe that \( \mu_2 \) as in
\eqref{8.10} is a Szeg\H{o} weight. Hence, Lemma~\ref{lem:6-1} is
applicable. Therefore,
\begin{equation}
\label{8.11}
\lim_{|\n|\to\infty,~\n\in\mathcal N_0}\lim_{z\to\infty} \left(\frac{P_{\n+\vec e_1}(z)}{P_\n(z)} - z \right) = -B_{0,1}
\end{equation}
by \eqref{6.0}. On the other hand, it should be clear from the
above argument that the proof in Section~\ref{sec:8} will work if
\( \mu_2 \) is as in \eqref{8.10}. Therefore,
Theorem~\ref{thm:asymp2} for such a choice of \( \mu_2 \) gives us
that
\begin{equation}
\label{8.12}
\frac{P_{\n+\vec e_1}(z)}{P_\n(z)} = (1 + o(1)) \frac{\gamma_{\n+\vec e_1}(S_{\n+\vec e_1}\Phi_{\n+\vec e_1})^{(0)}(z)}{\gamma_\n(S_\n\Phi_\n)^{(0)}(z)}
\end{equation}
in a neighborhood of the point at infinity. It follows from
\eqref{8.11}, \eqref{8.12}, and \eqref{szego-limit2} that
\begin{equation}
\label{8.13}
\lim_{|\n|\to\infty,~\n\in\mathcal N_0}\lim_{z\to\infty} \left(\frac{\tau_{\n+\vec e_1}\Phi_{\n+\vec e_1}^{(0)}(z)}{\tau_\n\Phi_\n^{(0)}(z)} - z \right) = -B_{0,1},
\end{equation}
where \( \tau_\n \) was defined in Theorem~\ref{thm:asymp3}.
Observe that \eqref{8.13} is a statement about Riemann surfaces \(
\RS_\n \) for \( \n\in\mathcal N_0 \) and is independent of the
original measures \( \mu_1,\mu_2 \). By Theorem~\ref{thm:asymp2}, \eqref{8.12} holds for
measures \( \mu_1,\mu_2 \) as in Theorem~\ref{thm:recurrence},
which we are currently proving. Hence, polynomials \( P_\n(z) \), \( \n\in\mathcal N_0 \), satisfy
\eqref{8.11} by \eqref{8.13} and \eqref{szego-limit2}. The final
claim of the theorem now follows from \eqref{8.2}.

\appendix

\section{}
\label{appendix}

In this Appendix, we will study the operators $\mathcal{L}_c^{(1)}$ and $\mathcal{L}_c^{(2 )}$ defined in  \eqref{sad6}. As we have already mentioned in Section~\ref{ss:2.2}, these operators appear in \cite[Formula~(4.20)]{uApDenY} used with $\vec{\kappa}=\vec{e}_1$ and $\vec{\kappa}=\vec{e}_2$, respectively. The analysis in this section is  fairly standard for the Spectral Theory of Jacobi matrices on trees  (see, e.g., \cite{at2} where the Laplacian and its perturbations were studied for some trees with the finite cone type). However, to make the paper self-contained, we provide  complete proofs. That will also emphasize the connection between the  quantities used in Spectral Theory, such as $m$--functions to be defined a few lines below, and the quantities  standard in the asymptotical analysis of multiple orthogonal polynomials, e.g., function $\chi_c^{(0)}$.\smallskip

We denote by $\delta^{(Y)}$ the delta function (Kronecker symbol) of the vertex~$Y$. Consider two functions
\begin{equation}\label{sad12}
m_I(z):=\langle (\mathcal{L}_c^{(1)}-z)^{-1}\delta^{(O)},\delta^{(O)}\rangle, \quad m_{II}(z):=\langle (\mathcal{L}_{c}^{(2)}-z)^{-1}\delta^{(O)},\delta^{(O)}\rangle\,.
\end{equation}
Given the function \( \chi_c(\z) \) from Proposition~\ref{prop:angelesco} and $c\in (0,1)$, \cite[Equation (4.22)]{uApDenY} yields that
\[
m_I(z) =\frac{-1}{\chi_c^{(0)}(z)-B_{c,1}}, \quad m_{II}(z) = \frac{-1}{\chi_c^{(0)}(z)-B_{c,2}}\,,
\]
where, as usual, $\chi_c^{(0)}(z)$ are the values taken from the zero-th sheet $\RS_c^{(0)}$. By the Spectral Theorem \cite{akh2}, they can also be written in the form
\[
m_I(z)=\int_{\R}\frac{\dd\sigma^{(1)}_O(x)}{x-z}, \quad m_{II}(z)=\int_{\R}\frac{\dd\sigma^{(2)}_O(x)}{x-z}\,,
\]
where $\sigma^{(l)}_O$  is the spectral measure of $\delta^{(O)}$ with respect to $\mathcal{L}_c^{(l)}$, \( l\in\{1,2\}\). The properties of the conformal map $\chi_c(\z)$ imply that the functions $m_I(z)$ and $m_{II}(z)$ satisfy: \smallskip
\\
{\bf (A)}  $m_I(z)$ and $m_{II}(z)$ have no poles since $\chi_c^{(0)}(z)\neq B_{c,j}$ for $z\in \RS_c^{(0)}$ by conformality; \smallskip
\\
{\bf (B)} both $m_I(z)$ and $m_{II}(z)$ are Herglotz-Nevanlinna functions in $\C^+$, i.e., they are analytic, have positive imaginary part, and are continuous up to the boundary. Moreover, $\Im m_{I}(x)=\Im m_{II}(x)=0$ for $x\in \R\backslash (\Delta_{c,1}\cup \Delta_{c,2})$ and $\Im m_{I}^+(x)>0,\Im m_{II}^+(x)>0$ for $x\in\Delta_{c,1}^\circ\cup\Delta_{c,2}^\circ$. \smallskip

We will use the following notation. If $Y,Z\in\mathcal{V}$ and $Y\sim Z$, then deleting the edge $(Y,Z)$ that connects them leaves us with two subtrees. The one containing $Y$ will be called $\mathcal{T}_{[Z,Y]}$, the other one will be called $\mathcal{T}_{[Y,Z]}$. The restriction of any Jacobi matrix $\mathcal{J}$ to a subtree $\mathcal{T}'$ will be denoted by~$\mathcal{J}_{\mathcal{T}'}$.

We learned from {\bf  (A)} and {\bf (B)} above that $\sigma_O^{(1)}$ and $\sigma_{O}^{(2)}$ are absolutely continuous measures with supports equal to $\Delta_{c,1}\cup\Delta_{c,2}$. We need this for the following lemma. 

\begin{lemma}If $c\in (0,1)$, then $\mathcal{L}_c^{(1)}$ and $\mathcal{L}_c^{(2)}$ have no eigenvalues.\label{sad13}
\end{lemma}
\begin{proof}
Suppose that $\mathcal{L}_c^{(l)}$, \( l\in\{1,2\} \), has an eigenvector $\Psi$. Since $\sigma_{O}^{(l)}$ is purely absolutely continuous as just explained, the restriction of $\mathcal{L}_c^{(l)}$ to the cyclic subspace generated by $\delta^{(O)}$ has no eigenvalues by the spectral theorem. Therefore, we must have $\Psi_O=0$. Now, consider the restrictions of $\Psi$ to $\mathcal{T}_{[O,O_{(ch),1}]}$ and to $\mathcal{T}_{[O,O_{(ch),2}]}$. One of these functions is not identically equal to zero and the one that is not must be an eigenvector of the corresponding operator: either $\mathcal{J}_{\mathcal{T}_{[O,O_{(ch),1}]}}$ or $\mathcal{J}_{\mathcal{T}_{[O,O_{(ch),2}]}}$. By construction, these operators are identical to either $\mathcal{L}_c^{(1)}$ or $\mathcal{L}_c^{(2)}$ and, as we established earlier, this implies that $\Psi_{O_{(ch),1}}=\Psi_{O_{(ch),2}}=0$. Repeating the argument, we can now show that $\Psi=0$ identically on the whole tree which gives a contradiction.
\end{proof}

The following observation holds for a general Jacobi matrix \eqref{Isad1} and \eqref{Isad2}. Let $\sigma_Y$ denote the spectral measure of $\delta^{(Y)}$ with respect to $\mathcal{J}$, i.e.,
\begin{equation}
\label{mly2}
m_Y(z):= \left\langle(\mathcal{J}-z)^{-1}\delta^{(Y)},\delta^{(Y)}\right\rangle=\int_{\R}\frac{d\sigma_{Y}(x)}{x-z},\quad  z\in \C^+\,.
\end{equation}
If we delete all edges connecting $Y$ to its neighbors, say \( l \) of them, we will be left with the vertex $Y$ and $l$ subtrees \( \{\mathcal{T}_{[Y,Y_j]}\}_{j=1}^l \). The restrictions of $\mathcal{J}$ to these subtrees are also Jacobi matrices and we previously denoted  them by $\mathcal{J}_{\mathcal{T}_{[Y,Y_j]}}$. Let
\begin{equation}
\label{mly3}
{m}_{[Y,Y_j]}(z):= \left\langle(\mathcal{J}_{\mathcal{T}_{[Y,Y_j]}}-z)^{-1}\delta^{(Y_j)},\delta^{(Y_j)}\right\rangle=\int_{\R}\frac{d\sigma_{[Y,Y_j]}(x)}{x-z},\quad z\in \C^+\,.
\end{equation}
Then the following lemma holds.

\begin{lemma} 
For every $z\in \C^+$, we have
\begin{equation}
\label{Isad13}
m_Y(z)=\frac1{V_Y-\sum_{j=1}^l W_{Y_j,Y}{m}_{[Y,Y_j]}(z)-z}.
\end{equation}
\end{lemma}
\begin{proof}
Let \( f:=(\mathcal{J}-z)^{-1}\delta^{(Y)}\). Clearly, \( \mathcal Jf= zf+\delta^{(Y)} \), that is,
\begin{equation}
\label{Isad9}
(\mathcal Jf)_X = \left\{
\begin{array}{l}
V_Xf_X + \sum_{Z\sim X} W^{1/2}_{Z,X}f_X = zf_X, \quad X\neq Y, \bigskip \\
V_Yf_Y + \sum_{j=1}^l W^{1/2}_{Y_j,Y}f_{Y_j} = zf_Y+1, \quad X=Y.
\end{array}
\right.
\end{equation}
Set \( f^{(j)} := -\big(W_{Y,Y_j}^{1/2}f_Y\big)^{-1}f_{|\mathcal V_{[Y,Y_j]}} \), which is a renormalized restriction of \( f \) to the set of vertices \( \mathcal V_{[Y,Y_j]} \) of \( \mathcal T_{[Y,Y_j]} \). Observe that
\begin{equation}
\label{mly1}
\left(\mathcal J_{\mathcal T_{[Y,Y_j]}}f^{(j)}\right)_X = \left\{
\begin{array}{l}
\left( \mathcal Jf^{(j)}\right)_X = zf_X^{(j)}, \quad X\neq Y_j, \bigskip \\
V_{Y_j}f_{Y_j}^{(j)} + \sum_{Z\sim Y_j,Z\neq Y} W^{1/2}_{Z,Y_j}f_Z^{(j)} = zf_{Y_j}^{(j)} + 1, \quad X=Y_j, 
\end{array}
\right.
\end{equation}
where both relations follow from the first line of \eqref{Isad9} (for the second relation we need to separate the summand corresponding to \( Z=Y \), bring it to the other side of the equation, and then divide by it). It follows immediately from \eqref{mly1} that
\[
\mathcal J_{\mathcal T_{[Y,Y_j]}}f^{(j)} = zf^{(j)} + \delta^{Y_j} \quad \Rightarrow \quad f^{(j)} = (\mathcal{J}_{\mathcal{T}_{[Y,Y_j]}}-z)^{-1}\delta^{(Y_j)}.
\]
The claim of the lemma follows from the second equality in \eqref{Isad9} since \( f_Y= \langle(\mathcal{J}-z)^{-1}\delta^{(Y)},\delta^{(Y)}\rangle = m_Y(z)\) and similarly \( f_{Y_j} = -\big(W_{Y,Y_j}^{1/2}f_Y\big) f_{Y_j}^{(j)} =-W^{1/2}_{Y_j,Y}m_Y(z)m_{[Y,Y_j]}(z)\).
\end{proof}

\noindent {\bf Remark.} The recursion relations for $m$-functions, such as the one in formula \eqref{Isad13}, are well-known and have been used previously, e.g., \cite{ao,d2,kl}.

\medskip
Let us now return to the operators \( \mathcal J=\mathcal{L}_c^{(l)} \), \( l\in\{1,2\} \).  Take any vertex $Y\neq O$. Deleting the edge $(Y,Y_{(p)})$ leaves us with two subtrees. As before, we denote by $\mathcal{T}_{[Y,Y_{(p)}]}$ the one containing $Y_{(p)}$, and let \( m_Y^{(l)}(z) \) and \( m_{[Y,Z]}^{(l)}(z) \) to be given by \eqref{mly2} and \eqref{mly3}, respectively (with \( \mathcal J=\mathcal{L}_c^{(l)} \)).

\begin{lemma}\label{Isad20}
For every $Y\neq O$, the function $m^{(l)}_{[Y,Y_{(p)}]}(z)$ is meromorphic in $\overline\C\setminus(\Delta_{c,1}\cup\Delta_{c,2})$ and the function $m_Y^{(l)}(z)$ is analytic there.
\end{lemma}
\begin{proof} Recall that the functions $m_I(z)$ and $m_{II}(z)$ are in fact analytic in $\overline\C\setminus(\Delta_{c,1}\cup\Delta_{c,2})$. We shall prove the desired claims inductively on \( n \), the distance from $Y$ to the root $O$. Assume first that $n=1$. Let \( \iota \) be the type of \( Y \). Formula \eqref{Isad13} applied at the vertex \( O \) to the operator \( \mathcal{L}_c^{(l)} \) restricted to the subtree \( \mathcal T_{[Y,O]} \) gives
\[
m^{(l)}_{[Y,O]}(z) = \frac1{B_{c,l}-A_{c,3-\iota}m_{[O,Z]}^{(l)}(z)-z},
\]
where \( Z \) is the other ``child'' of \( O \) and we used an obvious fact that the restriction of \( \mathcal{L}_c^{(l)} \)  from \( \mathcal T_{[Y,O]} \) to the subtree \( \mathcal T_{[O,Z]} \) is the same as the restriction of \( \mathcal{L}_c^{(l)} \)  from \( \mathcal T \) to \( \mathcal T_{[O,Z]} \). Since the restriction of \( \mathcal{L}_c^{(l)} \) to \( T_{[O,Z]} \) is \( \mathcal{L}_c^{(3-\iota)} \), $m_{[O,Z]}^{(l)}(z) \) is equal to either \( m_I(z) \) when \( \iota=2 \) or \( m_{II}(z) \) when \( \iota=1 \). In any case, $m_{[Y,O]}^{(l)}(z)$ is meromorphic outside  $\Delta_{c,1}\cup\Delta_{c,2}$.

Suppose now that the claims are true for all vertices up to the distance $n$. Consider any $Y$ such that its distance from the root is $n+1$. Let \( \iota \) be the type of \( Y \). As in the first part of the proof, apply \eqref{Isad13} at the vertex \( Y_{(p)} \) of the subtree \( \mathcal T_{[Y,Y_{(p)}]} \) to get
\[
m_{[Y,Y_{(p)}]}^{(l)}(z) = \frac1{B_{c,\iota_{p}} - A_{c,\iota_{(p)}}m^{(l)}_{[Y_{(p)},(Y_{(p)})_{(p)}]}(z)  - A_{c,3-\iota}m^{(l)}_{[Y_{(p)}, Z]}(z) -z}.
\]
where \( \iota_{(p)} \) is the type of \( Y_{(p)} \) and \( Z \) is the ``sibling'' of \( Y \). The first function in the denominator is meromorphic outside $\Delta_{c,1}\cup\Delta_{c,2}$ by the inductive assumption and the other one is either $m_I(z)$ or $m_{II}(z)$. Thus, $m_{[Y,Y_{(p)}]}^{(l)}$ is also meromorphic  outside $\Delta_{c,1}\cup\Delta_{c,2}$. This way we get the claim for $n+1$ and so we proved the first statement of the lemma. 

Now, apply \eqref{Isad13} to $m_Y^{(l)}(z)$. The functions involved are $m_{[Y,Y_{(ch),j}]}(z)$, \(j\in\{1,2\} \),  and $m_{[Y,Y_{(p)}]}(z)$. The first two are $m_I(z),m_{II}(z)$ and they are analytic in the considered domain. The third one is meromorphic there by the first statement of the lemma. Notice  that $m^{(l)}_Y(z)$ can not have poles by Lemma \ref{sad13} thus it is analytic outside $\Delta_{c,1}\cup\Delta_{c,2}$.
\end{proof}

\begin{lemma} 
\label{sad14}
Let $Y\in \mathcal{V}$ and $c\in (0,1)$. If $\sigma_Y^{(l)}$ is the spectral measure of $Y$ with respect to $\mathcal{L}_c^{(l)}$, \( l\in\{1,2\} \), then it is absolutely continuous and its support is equal to $\Delta_{c,1}\cup\Delta_{c,2}$.
\end{lemma}
\begin{proof}
The measure $\sigma_O^{(l)}$ is purely absolutely continuous and is supported on $\Delta_{c,1}\cup\Delta_{c,2}$ as explained before Lemma~\ref{sad13}. Fix $Y\neq O$ and let \( \iota_Y \) be the type of \( Y \). Further, let \( m_Y^{(l)}(z) \) and \( m_{[Y,Z]}^{(l)}(z) \) be given by \eqref{mly2} and \eqref{mly3}, respectively, with $\mathcal{J}=\mathcal{L}_c^{(l)}$. Then is follows from \eqref{sad6} and \eqref{Isad13} that
\begin{multline*}
\Im m_Y^{(l)}(E+\ic\epsilon)=  \frac{A_{c,\iota_Y}\Im m^{(l)}_{[Y,Y_{(p)}]}(E+\ic\epsilon)+\sum_{i=1}^2 A_{c,i}\Im {m}^{(l)}_{[Y,Y_{(ch),i}]}(E+\ic\epsilon)+\epsilon}{|B_{c,\iota_Y}- A_{c,\iota_Y} m^{(l)}_{[Y,Y_{(p)}]}(E+\ic\epsilon)-\sum_{i=1}^2 A_{c,i} {m}^{(l)}_{[Y,Y_{(ch),i}]}(E+\ic\epsilon)  -(E+\ic\epsilon)|^2} \\
\le \frac1{A_{c,\iota_Y}\Im m^{(l)}_{[Y,Y_{(p)}]}(E+\ic\epsilon)+\sum_{i=1}^2 A_{c,i}\Im {m}^{(l)}_{[Y,Y_{(ch),i}]}(E+\ic\epsilon) + \epsilon} \le \frac1{A_{c,l}\Im {m}^{(l)}_{[Y,Y_{(ch),l}]}(E+\ic\epsilon)},
\end{multline*}
because the imaginary parts of all $m$-functions are positive in $\C^+$. Notice now that the restriction of \( \mathcal L_c^{(l)} \) to any subtree of the type $\mathcal{T}_{[Z_{(p)},Z]}$ is in fact equal to either \( \mathcal L_c^{(1)} \) or \( \mathcal L_c^{(2)} \). Therefore, $m^{(l)}_{[Y,Y_{(ch),l}]}$ is either $m_I$ or $m_{II}$. The properties {\bf (A)} and {\bf (B)} of $m_I$ and $m_{II}$ listed above can be now applied to get
\[
\sup_{E\in I,0<\epsilon<1}|\Im m_Y^{(l)}(E+\ic\epsilon)|<\infty
\]
for every interval $I\subset \Delta_{c,1}\cup\Delta_{c,2}$. This implies that $\sigma_Y^{(l)}$ is purely absolutely continuous on $I$. By Lemma \ref{Isad20}, the measure $\sigma_Y^{(l)}$ is supported inside $\Delta_{c,1}\cup\Delta_{c,2}$ and Lemma~\ref{sad13} implies that it has no mass points. Therefore, we conclude that $\sigma_Y^{(l)}$ is purely absolutely continuous, as claimed.
\end{proof}

\begin{theorem}
\label{sad15}
We have that \( \sigma(\mathcal{L}_c^{(l)})=\sigma_{\rm ess}(\mathcal{L}_c^{(l)})=\Delta_{c,1}\cup\Delta_{c,2} \), \( l\in\{1,2\} \), where, as before, we understand that  $\Delta_{0,1}:=\{\alpha_1\}$ and $\Delta_{1,2}:=\{\beta_2\}$.
\end{theorem}
\begin{proof}
If $c\in (0,1)$, Lemma~\ref{sad14} shows that $\delta^{(Y)}$ belongs to the absolutely continuous subspace of $\mathcal{L}_c^{(l)}$ for all $Y$. Since all linear combinations of $\delta^{(Y)}$ must belong to this subspace and are dense in $\ell^2(\mathcal{V})$, this subspace is in fact the whole space $\ell^2(\mathcal{V})$. Thus, \( \sigma(\mathcal{L}_c^{(l)})=\sigma_{\rm ess}(\mathcal{L}_c^{(l)}) \) and it is equal to \( \Delta_{c,1}\cup\Delta_{c,2} \) by Lemma~\ref{sad14} and the Spectral Theorem.

Let $c\in \{0,1\}$. We shall consider $\mathcal{L}_0^{(2)}$ only, other cases  can be handled similarly. By \eqref{AngPar2}, we have $A_{0,1}=0$ and $A_{0,2}>0$. Thus, the operator $\mathcal{L}^{(2)}_0$ decouples into the following direct sum 
\begin{equation}
\label{sad5}
\mathcal{L}^{(2)}_0=\mathcal{A}_1\bigoplus \left(\bigoplus_{n=1}^\infty \mathcal{A}_2\right)
\end{equation}
where $\mathcal{A}_1$ is one-sided Jacobi matrix
\[
\mathcal{A}_1:=\left(\begin{matrix}
B_{0,2} &  \sqrt{A_{0,2}} & 0&0\\
\sqrt{A_{0,2}}& B_{0,2}  &\sqrt{A_{0,2}}&0\\
0& \sqrt{A_{0,2}}& B_{0,2}  &\sqrt{A_{0,2}}\\
\ldots &\ldots &\ldots &\ldots
\end{matrix}\right)
\]
and $\mathcal{A}_2$ is one-sided Jacobi matrix given by
\[
\mathcal{A}_2:=\left(\begin{matrix}
B_{0,1} &  \sqrt{A_{0,2}} & 0&0\\
\sqrt{A_{0,2}}& B_{0,2}  &\sqrt{A_{0,2}}&0\\
0& \sqrt{A_{0,2}}& B_{0,2}  &\sqrt{A_{0,2}}\\
\ldots &\ldots &\ldots &\ldots
\end{matrix}\right)\,.
\]
This direct sum decomposition implies that \( \sigma(\mathcal{L}^{(2)}_0)=\sigma(\mathcal{A}_1)\cup\sigma(\mathcal{A}_2) \). It is well known that $\sigma(\mathcal{A}_1)=[B_{0,2}-2\sqrt{A_{0,2}},B_{0,2}+ 2\sqrt{A_{0,2}}]=[\alpha_2,\beta_2]$, see \eqref{AngPar2} for the second equality, and that
\[
\widehat m_1(z) := \langle (\mathcal{A}_1-z)^{-1}\delta^{(0)},\delta^{(0)}\rangle =\frac{B_{0,2}-z+\sqrt{(z-B_{0,2})^2-4A_{0,2}}}{2A_{0,2}}=\frac{B_{0,2}-z+w_2(z)}{2A_{0,2}}.
\]
Furthermore, since the restriction of \( \mathcal A_2 \) from \( \ell^2(\ZZ) \) to \( \ell^2(\N) \) is equal to \( \mathcal A_1 \) and therefore \( m_{[0,1]}(z)=\widehat m_1(z) \) in the notation of \eqref{mly3}, we get from \eqref{Isad13} that
\[
\widehat m_2(z) := \langle (\mathcal{A}_2-z)^{-1}\delta^{(0)},\delta^{(0)}\rangle =\frac{-1}{A_{0,2}\widehat m_1(z)+z-B_{0,1}},
\]
where $w_2(z)$ was introduced in the Proposition  \ref{prop:angelesco}. One can readily check that $\Im \widehat m_2(x)>0$ for $x\in(\alpha_2,\beta_2)$, $\Im \widehat m_2(x)=0$ for \( x\not\in[\alpha_2,\beta_2] \), and that \( \widehat m_2(z) \) has the unique pole at a point $\widecheck x\in \R$ given by 
\begin{equation}
\label{mly4}
A_{0,2}\widehat m_1(\widecheck x)+\widecheck x-B_{0,1}=0
\end{equation}
which implies that $\widecheck x=\alpha_1$ thanks to \eqref{AngPar2}. In other words, $\sigma(\mathcal{A}_2)=\alpha_1\cup[\alpha_2,\beta_2]$. Now, the statement about the spectrum and essential spectrum follows from direct sum decomposition \eqref{sad5}.
\end{proof}

\end{document}